\documentclass[reqno, 10pt]{amsart}
\usepackage{upgreek}
\usepackage{mathtools}
\usepackage{amssymb}
\usepackage{mathrsfs}
\usepackage{amsthm}
\usepackage{enumerate}
\numberwithin{equation}{section} 

\usepackage{soul}
\usepackage{times}

\usepackage{textcomp}

\usepackage[scr=boondox]{mathalpha}
\usepackage{srcltx}
\usepackage{thmtools}
\makeatletter
\renewcommand\thmt@autorefsetup{%
  \@xa\def\csname\thmt@envname autorefname\@xa\endcsname\@xa{\thmt@thmname}%
}
\makeatother

\usepackage[left=1in, right=1in, top=1.25in, bottom=1.25in]{geometry}

\usepackage{tikz}
\usetikzlibrary{shapes,arrows,arrows.meta}

\usepackage{multicol}
\usepackage{graphicx}
\usepackage[colorlinks,allcolors=blue!75!black]{hyperref}
\usepackage{cancel}

\makeatletter
\newcommand{\leqnomode}{\tagsleft@true\let\veqno\@@leqno}
\makeatother

\allowdisplaybreaks

\newcommand{\bbL}{\mathbb{L}}

\newcommand{\bbV}{\mathbb{V}}

\def\bA{\mathbf{A}}
\def\bB{\mathbf{B}}
\def\bC{\mathbf{C}}

\def\bH{\mathbf{H}}
\def\bI{\mathbf{I}}
\def\bL{\mathbf{L}}
\def\bM{\mathbf{M}}
\def\bN{\mathbf{N}}
\def\bP{\mathbf{P}}

\def\bT{\mathbf{T}}

\def\bV{\mathbf{V}}
\def\bW{\mathbf{W}}

\def\ba{\mathbf{a}}
\def\bb{\mathbf{b}}
\def\be{\mathbf{e}}
\def\bg{\mathbf{g}}
\def\bh{\mathbf{h}}
\def\bj{\mathbf{j}}

\def\bp{\mathbf{p}}
\def\br{\mathbf{r}}
\def\bt{\mathbf{t}}
\def\bu{\mathbf{u}}
\def\bv{\mathbf{v}}
\def\bw{\mathbf{w}}
\def\bx{\mathbf{x}}
\def\by{\mathbf{y}}
\def\bz{\mathbf{z}}
\renewcommand{\bf}{\mathbf{f}}

\def\sfd{\mathsf{d}}

\def\balpha{\boldsymbol{\alpha}}
\def\bbeta{\boldsymbol{\beta}}

\def\bchi{\boldsymbol{\chi}}

\def\bgamma{\boldsymbol{\gamma}}
\def\bvarphi{\boldsymbol{\varphi}}

\def\bnu{\boldsymbol{\nu}}

\newcommand{\wh}{\widehat}
\newcommand{\wt}{\widetilde}

\newcommand{\LL}{\mathrm{L}}
\newcommand{\HH}{\mathrm{H}}
\newcommand{\NN}{\mathrm{N}}

\newcommand{\VV}{\mathrm{V}}
\newcommand{\RR}{\mathrm{R}}
\newcommand{\WW}{\mathrm{W}}

\newcommand{\cJ}{\mathcal{J}}

\newcommand{\cU}{\mathcal{U}}
\newcommand{\cV}{\mathcal{V}}

\newcommand{\sL}{\mathscr{L}}
\newcommand{\sG}{\mathscr{G}}

\newcommand{\dd}[1][y]{\if#1y\,\fi{\mathrm d}}
\DeclareMathOperator{\vspan}{span}

\renewcommand{\div}{\operatorname{div}}

\DeclareMathOperator{\curl}{curl}
\DeclareMathOperator{\bcurl}{\boldsymbol{\curl}}
\newcommand{\tr}{\operatorname*{tr}}

\newcommand{\G}{\Gamma}
\newcommand{\nablaG}{\nabla_\G}
\newcommand{\nablaM}{\nabla_M}
\newcommand{\divG}{\div_\G}
\newcommand{\divM}{\div_M}
\newcommand{\DeltaG}{\Delta_\G}
\newcommand{\DeltaB}{\Delta_B}
\newcommand{\DeltaH}{\Delta_H}
\newcommand{\DeltaS}{\Delta_S}
\newcommand{\uD}{\underline{D}}

\theoremstyle{plain}

\newtheorem{thm}{Theorem}[section]
\newtheorem{proposition}[thm]{Proposition}
\newtheorem{lemma}[thm]{Lemma}
\newtheorem{corollary}[thm]{Corollary}
\newtheorem{definition}[thm]{Definition}
\newtheorem{remark}[thm]{Remark}

\theoremstyle{definition}
\newtheorem{appendix-proposition}{Proposition}[section]

\theoremstyle{remark}
\newtheorem{appendix-remark}{Remark}[section]

\def\itemautorefname~#1\null{%
(#1)\null
}

\newenvironment{proof-claim}[1]{\noindent\textit{Proof of \autoref#1.}}{\hfill$\square$}
\newenvironment{proof-thm}[1]{\vskip1em \noindent\textbf{Proof of \autoref#1.}\\}{\hfill$\square$\\}

\makeatletter
\renewcommand\subsection{\@startsection{subsection}{2}%
\z@{.5\linespacing\@plus.7\linespacing}{.1\linespacing}%
{\normalfont\centering\scshape}}
\makeatother

\makeatletter
\renewcommand\subsubsection{\@startsection{subsubsection}{3}%
\z@{.5\linespacing\@plus.7\linespacing}{.1\linespacing}%
{\normalfont\raggedright\scshape}}
\makeatother

\makeatletter
\def\@tocline#1#2#3#4#5#6#7{\relax
\ifnum #1>\c@tocdepth 
\else
\par \addpenalty\@secpenalty\addvspace{#2}%
\begingroup \hyphenpenalty\@M
\@ifempty{#4}{%
\@tempdima\csname r@tocindent\number#1\endcsname\relax
}{%
\@tempdima#4\relax
}%
\parindent\z@ \leftskip#3\relax \advance\leftskip\@tempdima\relax
\rightskip\@pnumwidth plus4em \parfillskip-\@pnumwidth
#5\leavevmode\hskip-\@tempdima
\ifcase #1
\or\or \hskip 1em \or \hskip 2em \else \hskip 3em \fi%
#6\nobreak\relax
\dotfill\hbox to\@pnumwidth{\@tocpagenum{#7}}\par
\nobreak
\endgroup
\fi}
\makeatother

\begin{document}

\title[$\LL^p$ theory on manifolds of minimal regularity]{$\LL^p$-based Sobolev theory on closed manifolds of minimal regularity: Vector-valued problems}

\author[G.A.~Benavides]{Gonzalo A.~Benavides}
\author[R.H.~Nochetto]{Ricardo H.~Nochetto}
\author[M.~Shakipov]{Mansur Shakipov}

\address{Department of Mathematics, University of Maryland College Park, MD 20742}
\email{gonzalob@umd.edu, rhn@umd.edu, shakipov@umd.edu}

\subjclass[2020]{35A15, 35B65, 35D30, 58J, 35J20, 35Q35, 35Q30, 76D07}
\keywords{PDEs on manifolds, Sobolev regularity, well-posedness, Navier--Stokes, Stokes operator, Leray projector, duality, Banach--Ne\v{c}as--Babu\v{s}ka, Babu\v{s}ka--Brezzi, inf-sup, Fredholm alternative, Galerkin method, fixed-point}

\begin{abstract}
This paper is the second part of a two-paper series, initiated in \cite{BenavidesNochettoShakipov2025-a} for scalar PDEs on hypersurfaces, and is concerned with the well-posedness and $\LL^p$-based Sobolev regularity of vector-valued PDEs of interest in fluid dynamics.
This family of PDEs includes the (stationary) Bochner Laplace, tangent Stokes and Oseen, and tangent Navier--Stokes equations.
We present several strong, weak and ultra-weak formulations of these problems on compact, connected $d$-dimensional manifolds without boundary embedded in $\RR^{d+1}$.
We prove $\WW^{m,p}$-regularity for any $p \in (1,\infty)$ for manifolds of minimal regularity $C^{m+1}$ or $C^{m,1}$ for $m\ge1$.
Building upon the $\LL^p$-based scalar elliptic theory from \cite{BenavidesNochettoShakipov2025-a}, we develop a parametrization-free and purely variational approach that resorts to classical results such as the Banach--Ne\v{c}as--Babu\v{s}ka theorem and the generalized Babu\v{s}ka--Brezzi theory in reflexive Banach spaces.
In particular, by exploiting the manifold closedness, we decouple the velocity and pressure variables in the tangent Stokes problem to establish their higher-regularity $\bW^{m,p} \times \WW^{m-1,p}$ ($m \geq 2$) as a consequence of the $\LL^p$-based well-posedness and regularity theory for the Laplace--Beltrami and Bochner--Laplace operators.
We study spectral and regularity properties of an appropriate Stokes operator, and apply them to show existence of solutions for the Navier--Stokes equations for $p=2$ and $d \leq 4$.
We next extend the existence to $p > 2$ and prove higher-order $\LL^p$-based regularity.
We finally examine alternative choices to the Bochner Laplace operator that are useful in fluid dynamics.
\end{abstract}

\maketitle

\tableofcontents

\section{Introduction}
Let $\G$ be a $d$-dimensional ($d \geq 2$) compact, connected and sufficiently regular (to be specified later) manifold without boundary embedded in $\RR^{d+1}$ and denote by $\bnu = (\nu_i)_{i=1}^{d+1}$ its outward unit normal.
Throughout this work, we are concerned with the following three vector-valued problems:

\noindent\begin{minipage}{0.9\textwidth}
\begin{center}
\begin{subequations} \leqnomode
\begin{align}
\label{eq:manifold-BochnerPoisson} \tag{\textbf{VL}} & \text{\qquad \qquad \qquad \textbf{V}ector-\textbf{L}aplace} &  -\DeltaB \bu & = \bf, & & \\
\label{eq:manifold-stokes}  \tag{\textbf{S}} & \text{\qquad \qquad \qquad Tangent \textbf{S}tokes} & - \DeltaB \bu + \nablaG \uppi & = \bf, & & \divG \bu = g, \\
\label{eq:manifold-Navier--Stokes} \tag{\textbf{NS}} & \text{\qquad \qquad \qquad Tangent \textbf{N}avier--\textbf{S}tokes} & - \DeltaB \bu + (\nablaG \bu) \bu + \nablaG \uppi & = \bf, & & \divG \bu = g,
\end{align}
\end{subequations}
\end{center}
\end{minipage}
\noindent\begin{minipage}{0.09\textwidth}
\setcounter{equation}{0}
\begin{equation} \label{our-problems-of-interest}
\quad
\end{equation}
\end{minipage} \vskip1em
\noindent
where $g$ is a given scalar field with vanishing mean on $\G$ and $\bf$ is a vector field everywhere tangential to $\G$.
The operators $\nablaG$ and $\divG$ respectively denote the \emph{covariant gradient} and \emph{covariant divergence};
see \autoref{sec:difgeo-Sobspace} below and, e.g., \cite[Sections 2-3]{BenavidesNochettoShakipov2025-a}.
The operator $\DeltaB := \bP \divG \nablaG$ is usually referred to as the \emph{Bochner Laplacian} \cite[\textsection 3.2]{JankuhnOlshanskiiReusken2018}, \cite{Rosenberg1997} or \emph{rough Laplacian} \cite{HansboLarsonLarsson2020}, where $\bP := \bI - \bnu \otimes \bnu$ is the projection operator onto the tangent plane to $\G$.
The unknowns in \eqref{our-problems-of-interest} are the tangential vector field $\bu: \G \rightarrow \RR^{d+1}$ and the scalar field $\uppi: \G \rightarrow \RR$ with vanishing mean.

The problems \eqref{our-problems-of-interest} are of great importance as prototypical constituents of the modeling of a great variety of physical phenomena on thin films.
For instance, \eqref{eq:manifold-stokes} and \eqref{eq:manifold-Navier--Stokes} are the basis in the study of fluid flows, with paramount theoretical \cite{EbinMarsden:1970,Miura-PartI,Miura-PartII,Miura-PartIII,Miura-PartIII-erratum} and practical importance, e.g.~\cite{LagerwallScalia2012,NeedlemanDogic2017,DoostmohammadiEtAl2018}.
The governing system of geometric partial differential equations may be derived either by a formal asymptotic expansion \cite{ElliottStinner2010,Nitschke2018,NitschkeReutherVoigt2019} or by a rigorous approach \cite{Miura-PartIII,Miura2025}.

This work is devoted to establishing $\LL^p$-based ($1<p<\infty$) well-posedness and Sobolev regularity results for appropriate weak formulations of problems \eqref{our-problems-of-interest} in terms of the regularity of the data ($\bf$ and $g$) and manifold $\G$.    
We stress that, unlike what is typically done in the differential geometry literature, we do not assume $\G$ to be necessarily of class $C^\infty$, but rather we specify in each statement the required regularity on $\G$.
As far as the authors are aware, our work is the first one to provide a fully-fledged well-posedness and regularity theory for problems \eqref{our-problems-of-interest}, valid for any integrability parameter $p \in (1,\infty)$ and manifolds of minimal regularity and arbitrary intrinsic dimension $d \geq 2$.
Additionally, we consider our exposition to be accessible to people without strong foundations on Riemannian geometry, hence serving as a bridge to applied fields.

We find it important to point out that second-order linear differential operators other than the Bochner Laplacian $\DeltaB$ are considered throughout the vast differential geometry literature.
For instance, the \emph{deformation Laplacian} $-2 \operatorname{Def}^*\operatorname{Def}$---or its analogue acting on tangent vector fields, the \emph{Surface Diffusion} operator $\DeltaS$---is considered to be the most physically relevant diffusion operator in the context of fluid flows \cite{EbinMarsden:1970, KobaLiuGiga2017, JankuhnOlshanskiiReusken2018}.
We focus our analysis on the Bochner Laplacian $\DeltaB$ primarily for theoretical convenience: the weak formulations associated with $\DeltaB$ do not require the velocity field to be sought in the orthogonal complement of the possibly non-trivial kernel of the $\DeltaS$ operator to ensure uniqueness.
Because these operators---along with others like the Hodge Laplacian $\DeltaH$---are closely related (differing either by zeroth-order curvature terms or, in the case of $\Delta_S$, an additional $\nabla_\Gamma \div_\Gamma \bu$ term that reduces to known data in fluid models, as detailed in \autoref{sec:background}), we can seamlessly transfer our regularity results for $\DeltaB$ to $\DeltaS$ and $\DeltaH$ in \autoref{sec:connection-other-laplacians}.

\subsubsection*{Preliminary notations}
We now introduce function spaces on the manifold $\G$; these spaces are succinctly discussed in \autoref{sec:difgeo-Sobspace} below (see \cite[Section 3]{BenavidesNochettoShakipov2025-a} for further details). We use the notation  $\LL^q(\G)$ and $\WW^{m,q}(\G)$ for the scalar Lebesgue and Sobolev spaces with integrability $q \in [1,\infty)$ and differentiability $m \geq 0$; we denote the corresponding norms respectively by $\|\cdot\|_{0,q;\G}$ and $\|\cdot\|_{m,q;\G}$. We also write $\WW^{0,q}(\G) = \LL^q(\G)$ and $\HH^m(\G) := \WW^{m,2}(\G)$ with corresponding norm $\|\cdot\|_{m,\G} := \|\cdot\|_{m,2;\G}$.
For any scalar function space $\VV$ defined on $\G$, we denote by $\bV$ and $\bbV$ its $(d+1)$-vectorial and $(d+1) \times (d+1)$ tensorial counterparts, and we use the same notation for their norms.
Moreover, we denote by $\WW^{m,q}_\#(\G)$ the closed subset of scalar-valued functions in $\WW^{m,q}(\G)$ that satisfy $\int_\G v =0$, whereas $\bW^{m,q}_t(\G)$ denotes the space of vector-valued functions $\bv: \G \rightarrow \RR^{d+1}$ that belong to $\bW^{m,q}(\G)$ and are a.e.~tangential to $\G$ (i.e.~$\bv \cdot \bnu = 0$).

Finally, for any normed space $V$ its dual space is denoted by $V'$, whose induced norm is given by $\|f\|_{V'} := \sup_{0 \neq v \in V} \frac{f(v)}{\|v\|_V}$.
We denote the duality pairing between $V'$ and $V$ by $\langle \cdot, \cdot \rangle_{V' \times V}$;
when there is no ambiguity, we will simply write $\langle \cdot,\cdot \rangle$.

Throughout this article, for any integrability parameter $p \in (1,\infty)$, we denote by $p^* := \frac{p}{p-1}$ its conjugate exponent.
Notice that $p^{**} = p$.
Also, $p \in (1,2)$ if and only if $p^* \in (2,\infty)$.

\subsection{Background}\label{sec:background}
The differential operators involved in \eqref{our-problems-of-interest} satisfy integration-by-parts formulas that highly resemble their flat-domain counterparts \cite[Ap.~B]{BouckNochettoYushutin2024}, \cite[\textsection 2]{DziukElliott2013} (see also \cite[Lemma 1]{BonitoDemlowNochetto2020}).
Said identities naturally induce a set of $\LL^2$-based weak formulations that we now review.

\subsubsection*{Vector--Laplace problem}

The associated weak formulation to problem \eqref{eq:manifold-BochnerPoisson} reads:
given $\bf \in (\bH^1_t(\G))'$, find $\bu \in \bH^1_t(\G)$ such that
\begin{equation}\label{eq:weak-manifold-BochnerPoisson}
\int_\G \nablaG \bu : \nablaG \bv = \langle \bf,\bv \rangle, \qquad \forall \bv \in \bH^1_t(\G).
\end{equation}
The well-posedness of \eqref{eq:weak-manifold-BochnerPoisson} follows from a Poincaré-type inequality for the covariant gradient $\nablaG$, which has only been established for $\G$ of class $C^\infty$ and of intrinsic dimension $d=2$ \cite{HansboLarsonLarsson2020} (although, with a small gap in its proof).
We fix this gap and extend the result to $\G$ of class $C^2$ and arbitrary dimension $d \geq 2$ in \autoref{sec:Lp-regularity-BochnerLaplacian}.

As mentioned before, our analysis relies on $\Delta_B$ rather than the physically relevant Surface Diffusion operator $\Delta_S$ or other prevalent operators like the Hodge Laplacian $\Delta_H$.
To be precise, the Surface Diffusion operator acting on tangent vector fields is defined as $\Delta_S := \mathbf{P} \div_\Gamma 2D_\Gamma$, where $D_\Gamma := \frac{1}{2}(\nabla_\Gamma + \nabla_\Gamma^T)$ is the symmetrized covariant gradient.
On the other hand, in the context of differential forms, the Hodge Laplacian $\Delta_H = -(\dd^* \dd + \dd \dd^*)$ is defined in terms of the exterior derivative $\dd$ and codifferential $\dd^*$;
we refer to \cite[p.~601]{Taylor2011} for its definition on general compact Riemannian manifolds.
By means of the musical isomorphism, we can also understand $\Delta_H$ as acting on tangent vector fields.
We stress that none of these operators coincide with the componentwise application of the scalar Laplace Beltrami operator $\DeltaG$ studied in \cite{BenavidesNochettoShakipov2025-a}, whence the theory of \eqref{eq:weak-manifold-BochnerPoisson} is not a simple consequence of \cite{BenavidesNochettoShakipov2025-a}.
For any intrinsic dimension $d \geq 2$, these three operators are related via the identities \cite[eq.~(2.10)]{JankuhnOlshanskiiReusken2018}, \cite[p.~583, eq.~(3.56)]{TaylorIII_2023}:
for each $\bv \in \bC^2_t(\G)$
\begin{equation}\label{all-Laplace-relationship}
\DeltaS \bv= \DeltaB \bv + \nablaG \divG \bv + ( \tr(\bB) \bB - \bB^2) \bv, \qquad
\DeltaH \bv = \DeltaB \bv - ( \tr(\bB) \bB - \bB^2) \bv.
\end{equation}
Here we have used that the Ricci curvature tensor reads $\tr(\bB) \bB - \bB^2$ in terms of the Weingarten map $\bB:=\nabla_\G\bnu$, which in turn encodes the principal curvatures of $\G$.
This follows from elementary manipulations of the Riemann Curvature tensor (see \autoref{sec:RCT} below), or alternatively, by comparing \cite[p.~347, eq.~(4.26)]{TaylorII_2023} with \cite[eq.~(2.10)]{JankuhnOlshanskiiReusken2018}.

From \eqref{all-Laplace-relationship} it is apparent that the Hodge Laplacian $\DeltaH$ and Bochner Laplacian $\DeltaB$ operators only differ on a zeroth-order term $( \tr(\bB) \bB - \bB^2) \bv$.
Moreover, such a relation also holds for the \emph{surface diffusion} operator $\DeltaS$
under the additional \emph{manifold incompressibility} assumption $\divG \bv= 0$.
For manifolds embedded in $\RR^3$, the Hodge Laplacian can be rewritten as $\DeltaH = - \bcurl_\G \curl_\G + \nablaG \divG$.
We refer to \cite[\textsection 3.2]{JankuhnOlshanskiiReusken2018} and \cite[\textsection 4.1]{Reusken2020} and their references therein for a thorough discussion of these operators, including their historical origin and their relation with each other.
Because these operators are closely related---differing only by zeroth-order curvature terms and, in the case of $\DeltaS$, the $\nablaG \divG$ operator which can be handled via the continuity equation constraint or a Helmholtz--Weyl decomposition---we successfully transfer all regularity results for $\DeltaB$ to $\DeltaS$ and $\DeltaH$ in \autoref{sec:connection-other-laplacians}.

The previously mentioned Hodge theory \cite[\textsection 7.4]{Morrey2008} also provides $\WW^{2,p}$-regularity ($1 < p < \infty$) for solutions of the Hodge Laplacian $\DeltaH$ operators (for $1$-forms) in manifolds of class $C^{2,1}$.
By virtue of the musical isomorphism, we expect the result on $\DeltaH$ to extend to tangent vector fields, although we are unable to pinpoint a specific reference for the proof.
Consequently, said regularity would also be enjoyed by the Bochner Laplacian $\DeltaB$, as it only differs from the Hodge Laplace $\DeltaH$ in lower-order terms (cf.~\eqref{all-Laplace-relationship}).

\subsubsection*{Tangent Stokes and Navier--Stokes problems}
The corresponding weak formulation for the tangent Stokes problem reads: 
given $\bf \in (\bH^1_t(\G))'$ and $g \in \LL_\#^2(\G)$, find $(\bu,\uppi) \in \bH^1_t(\G) \times \LL_\#^2(\G)$ such that
\begin{equation}\label{eq:weak-manifold-stokes}
\begin{aligned}
\int_\G \nablaG \bu : \nablaG \bv
- \int_\G \uppi \divG \bv & = \langle \bf , \bv \rangle
, \quad & \forall\, \bv \in \bH^1_t(\G),\\
- \int_\G q \divG \bu & = -\int_\G g \, q, \quad & \forall\, q \in \LL_\#^2(\G).
\end{aligned}
\end{equation}
If we consider the Hodge Laplacian $\DeltaH$ or the Surface Diffusion $\DeltaS$ in lieu of the Bochner Laplacian $\DeltaB$, different Stokes systems arise;
see, for instance, \cite[\textsection\textsection 4-5]{JankuhnOlshanskiiReusken2018} for $\DeltaS$, and
\cite[\textsection 17.A]{Taylor2011} for $\DeltaH$.

The weak formulation for the surface diffusion operator $\DeltaS$ \cite[eq.~(4.5)]{JankuhnOlshanskiiReusken2018} replaces the diffusion term $\int_\G \nablaG \bu : \nablaG \bv$ in \eqref{eq:weak-manifold-stokes} by $2\int_\G D_\G(\bu):D_\G(\bv)$.
Consequently, the flow velocity $\bu$ needs to be sought in the orthogonal complement of the possibly non-trivial kernel of the operator $D_\G$ so as to ensure uniqueness.
The well-posedness of said formulation follows from a straightforward application of the Babu\v{s}ka--Brezzi theory (see, e.g.~\cite[Th.~2.1]{Gatica2014}), which hinges on a \emph{Korn-type inequality} \cite[Lemma~4.1]{JankuhnOlshanskiiReusken2018} and an \emph{inf-sup condition} on the covariant divergence $\divG$ \cite[Lemma~4.2]{JankuhnOlshanskiiReusken2018}.
As expected, the well-posedness of \eqref{eq:weak-manifold-stokes} will depend on the same \emph{Poincaré inequality} (instead of a Korn inequality) for the covariant gradient $\nablaG$ in the space $\bH^1_t(\G)$ used to establish the well-posedness of \eqref{eq:weak-manifold-BochnerPoisson}.

For intrinsic dimensions $d \leq 4$, the $\LL^2$-based weak formulation for the tangent Navier--Stokes reads:
given $\bf \in (\bH^1_t(\G))'$ and $g \in \LL_\#^2(\G)$, find $(\bu,\uppi) \in \bH^1_t(\G) \times \LL_\#^2(\G)$ such that
\begin{equation}\label{eq:weak-L2based-manifold-NS}
\begin{aligned}
\int_\G \nablaG \bu : \nablaG \bv
+ \int_\G (\nablaG \bu) \bu \cdot \bv
- \int_\G \uppi \divG \bv & = \langle \bf , \bv \rangle
, \quad & \forall\, \bv \in \bH^1_t(\G),\\
- \int_\G q \divG \bu & = -\int_\G g \, q, \quad & \forall\, q \in \LL_\#^2(\G);
\end{aligned}
\end{equation}
the restriction $d \le 4$ ensures that the integral $\int_\G (\nablaG \bu)\bu \cdot \bv$ makes sense as a consequence of Sobolev embeddings.
We are unable to refer to any specific results concerning the existence of solutions of \eqref{eq:weak-L2based-manifold-NS} for data $\bf$ and $g$ and manifold $\G$ of minimal regularity, although we notice that a relatively straightforward adaptation of Temam's strategy \cite[Chapter 1, \textsection 2.5]{Temam2001}, that is, using a Faedo--Galerkin method with eigenfunctions of an appropriate (tangent) Stokes operator (see \autoref{sec:Stokes-operator} below) will suffice, as attested in \autoref{sec:existence-incompressibleNS-d234} below.

In flat bounded domains, estimates in $\WW^{m,p}$-norms for Stokes systems, provided such regular-enough solutions exist, are a consequence of the results by Agmon, Douglis and Nirenberg \cite{ADN1955,ADN1964}, whose proofs are built upon techniques of integral representation and potential theory.
The existence of such solutions was proved in \cite{Cattabriga1961} for $d=3$, while the case $d=2$ is much simpler as it reduces to a biharmonic problem \cite[Chapter 1, Proposition 2.3]{Temam2001}.
We refer to \cite[Chapter 1, \textsection 2.5]{Temam2001} for a thorough discussion on the regularity of the Stokes system and its early historical development.

The special case of $\LL^2$-based (i.e.~$p=2$) estimates allows alternative proofs that circumvent the use of potential theory;
among them we find \cite{Ghidaglia1984,AmroucheGirault1994,BeiraodaVeiga1997,GiaquintaModica1982}.
Ghidaglia \cite{Ghidaglia1984} (see also \cite[\textsection 3]{ConstantinFoias1988}) argues based on difference quotients and flattening of the boundary, while Amrouche and Girault \cite{AmroucheGirault1994} develop a Helmholtz decomposition approach of the space $\bW^{m,p}(\Omega) \cap \bW^{1,p}_0(\Omega)$.
Finally, we refer to Giaquinta and Modica \cite{GiaquintaModica1982} for a clear and self-contained discussion of the $\LL^2$-based Sobolev and H\"older regularity of nonlinear Stokes systems with general coefficients.

We highlight the particularly elegant argument by Beir\~{a}o da Veiga \cite{BeiraodaVeiga1997} (later extended by Berselli \cite{Berselli2014}) that, via the ``method of continuity'' (homotopy argument) \cite[\textsection 9, Theorem 5.2]{GilbargTrudinger2001}, transfers the $\bW^{2,p}(\Omega)$-well-posedness of the Poisson's equation to $\bW^{2,p}(\Omega) \times \WW^{1,p}(\Omega)$-well-posedness of the Stokes system.

The development of a quantitative Sobolev regularity theory for the Stokes equations (in flat domains) in terms of data and boundary regularity is not only of intrinsic interest, but also is a powerful (and in many occasions indispensable) ingredient when showing well-posedness of many nonlinear PDEs arising from fluid mechanics via the Faedo--Galerkin method, e.g.~\cite{ConstantinFoias1988,Temam2001,LariosLunasinTiti2013,BerchioFalocchiPatriarca2024,AbelsDolzmann2014,GaldiKyed2016}.
More precisely, upon approximating the solution by a finite linear combination of $N$ eigenfunctions of the Stokes operator, the standard idea is to obtain sufficiently strong a-priori estimates in time and space that are uniform in $N$ by handling the nonlinearities of the problem via Sobolev interpolation inequalities and the higher regularity of the Stokes eigenfunctions.
A weak solution of the continuous problem is recovered in the limit $N \to \infty$ by compactness arguments.

Regarding the Dirichlet problem for the tangent Stokes system \eqref{eq:manifold-stokes} (or any of its variants) on compact $C^\infty$ Riemannian manifolds $M$ with smooth boundary, we mention Taylor's book \cite[Ch.~17, Ap.~A]{Taylor2011} for the $\LL^2$-based regularity theory for the Hodge Laplacian $\DeltaH$.
For the case of the \emph{deformation Laplacian} $-2 \operatorname{Def}^*\operatorname{Def}$ (the analogue of the surface diffusion operator acting on $1$-forms) and smooth $M$ with $C^1$ or  Lipschitz boundary, existence of solutions of the Stokes and Navier--Stokes problems (the latter for intrinsic dimensions $d \leq 4$) in Sobolev--Besov spaces has been addressed by, among others, Dindo\v{s}, Mitrea and Taylor \cite{MitreaTaylor:2001,DindosMitrea:2004}.
Similar results on (unbounded) manifolds with cylindrical ends have been recently established by Kohr, Nistor and Wendland in \cite{KNW2024}.
The arguments in \cite{Taylor2011,MitreaTaylor:2001,DindosMitrea:2004,KNW2024} are built upon Schwartz kernel representation of differential operators and the representation of the velocity solution in the form of layer potentials (in particular, \cite{KNW2024} also rests on the theory of pseudodifferential operators developed in \cite{KNW2023}).
We refer to Part III of the very recent monograph \cite{KMNW2026} by Kohr, Mikhailov, Nistor and Wendland, for a systematic and thorough discussion of Boundary Value Problems for Stokes and Navier--Stokes-type problems on compact Riemannian manifolds, and to Part IV for akin results on manifolds with cylindrical ends.

Regarding manifolds (or hypersurfaces) of minimal regularity, Olshanskii, Reusken and Zhiliakov \cite[Lemma 2.1]{ORZ2021} provide $\bH^2\times \HH^1$-regularity for the case of the surface diffusion operator $\DeltaS$ and 2D compact $C^{2,1}$-regular surfaces without boundary embedded in $\RR^3$.
Its proof hinges on the ($d=2$)-only characterization $\DeltaH = -\bcurl_\G \curl_\G + \nablaG \divG$ of the Hodge Laplacian and its simple relation with the surface diffusion operator to decouple velocity and pressure, and regularity results in Hodge theory \cite[\textsection 7.4]{Morrey2008}.

There seem to be no readily available analogues for ``rougher'' manifolds of class $C^m$ or $C^{m-1,1}$ ($m \in \NN$) of arbitrary dimension $d \geq 2$.
On the one hand, the argument in \cite{ORZ2021} cannot be extended as is to dimensions $d \geq 3$.
On the other hand, adapting Dindo\v{s}, Mitrea and Taylor's approach to this less stringent regularity context is highly non-trivial.
We acknowledge the vast and thorough development of potential theory, well-posedness and regularity results on Lipschitz and $C^1$ domains in smooth Riemannian manifolds with metric tensors of limited regularity---see \cite{MitreaMitreaTaylor:2001} for a foundational monograph on the topic, and, e.g., \cite{MitreaTaylor-dinicoeff:2003,MitreaTaylor-holdercoeff:2005} for more specialized results valid for Dini and H\"older metric tensors, respectively.
However, the techniques therein assume infinite smoothness of the manifold (meaning the underlying coordinate charts are smooth), rendering them inapplicable to our context.
While extending their framework to manifolds of minimal regularity might be theoretically possible, such an endeavor falls outside the scope of this work, making our direct extrinsic calculus approach a natural and necessary alternative.

\subsection{Statements of the main results}\label{subsec:main-results}
We now collect the main results of this paper, and briefly comment on their nuances and connections.
The following result for the {\it Bochner-Laplace} operator \eqref{eq:manifold-BochnerPoisson} is derived in \autoref{sec:bl} and hinges on a Poincar\'e inequality for the covariant gradient $\nablaG$ proved in \autoref{lem:Poincare}.

\begin{thm}[Bochner--Laplace] \label{thm:bl}
Let $\G$ be of class $C^2$ and let $p \in (1,\infty)$.
Then, for every $\bf \in (\bW^{1,p^*}_t(\G))'$, there exists a unique $\bu \in \bW^{1,p}_t(\G)$ such that
\begin{equation} \label{eq:bl-weak}
\int_\G \nablaG \bu : \nablaG \bv = \langle \bf, \bv \rangle, \qquad \forall \bv \in \bW^{1,p^*}_t(\G).
\end{equation}
Moreover, there exists a positive constant $C$, depending only on $\G$ and $p$, such that
\[
\|\bu\|_{1,p;\G} \leq C \|\bf\|_{(\bW^{1,p^*}_t(\G))'}.
\]
Furthermore, if $\G$ is of class $C^{m+2,1}$ for some nonnegative integer $m$ and $\bf \in \bW^{m,p}_t(\G)$, then the solution $\bu \in \bW^{1,p}_t(\G)$ of \eqref{eq:bl-weak} is in fact in $\bW^{m+2,p}_t(\G)$ and strongly satisfies the PDE: $-\DeltaB \bu = \bf$ a.e.~on $\G$.
In addition, there is a constant $C = C(\G, m, p) > 0$ such that
\[
\|\bu\|_{m+2,p;\G} \leq C \|\bf\|_{m,p;\G}.
\]
\end{thm}

We point out that the required regularity of $\G$ in \autoref{thm:bl} is one unit higher than in \cite[Theorems 1.1 and 1.2]{BenavidesNochettoShakipov2025-a} for the scalar Laplace-Beltrami operator. This is due to the fact that the vector-valued solution $\bu\in
\bW^{1,p}_t(\G)$ is tangent a.e. to $\G$, whence $\bu$ is a linear combination of tangent vectors to $\G$ that already consume one derivative of the parametrization of $\G$. This observation applies to all statements below that rely on \autoref{thm:bl}.

We prove the $\LL^p$-based ($1<p<\infty$) well-posedness and higher regularity of the {\it tangent Stokes} problem \eqref{eq:manifold-stokes} in \autoref{sec:stokes} by generalizing a trick of Olshanskii, Reusken and Zhiliakov for the proof of \cite[Lemma 2.1]{ORZ2021}.

\begin{thm}[tangent Stokes] \label{thm:stokes}
Let $\G$ be of class $C^2$ and let $p \in (1,\infty)$.
Then, for every $(\bf, g) \in (\bW^{1,p^*}_t(\G))' \times \LL^p_\#(\G)$, there exists a unique pair $(\bu, \uppi) \in \bW^{1,p}_t(\G) \times \LL^p_\#(\G)$ such that
\begin{equation*}
\begin{aligned}
\int_\G \nablaG \bu : \nablaG \bv
- \int_\G \uppi \divG \bv & = \langle \bf , \bv \rangle
, \quad & \forall\, \bv \in \bW^{1,p^*}_t(\G),\\
- \int_\G q \divG \bu & = -\int_\G g \, q, \quad & \forall\, q \in \LL_\#^{p^*}(\G).
\end{aligned}
\end{equation*}
Moreover, there exists a positive constant $C$, depending only on $\G$ and $p$, such that
\[
\|\bu\|_{1,p;\G} + \|\uppi\|_{0,p;\G} \leq C \left(\|\bf\|_{(\bW^{1,p^*}_t(\G))'} + \|g\|_{0,p;\G} \right).
\]
Furthermore, if $\G$ is of class $C^{m+2,1}$ for some nonnegative integer $m$ and $(\bf, g) \in \bW^{m,p}_t(\G) \times \WW^{m+1,p}_\#(\G)$, then the solution $(\bu, \uppi) \in \bW^{1,p}_t(\G) \times \LL^p_\#(\G)$ is in fact in $\bW^{m+2,p}_t(\G) \times \WW^{m+1, p}_\#(\G)$ and satisfies the PDEs in strong form: $ -\DeltaB \bu + \nablaG \uppi = \bf, \quad \divG \bu = g$ a.e. on $\G$.
In addition, there is a constant $C = C(\G, m, p) > 0$ such that
\[
\|\bu\|_{m+2,p;\G} + \|\uppi\|_{m+1,p;\G} \leq C (\|\bf\|_{m,p;\G} + \|g\|_{m+1,p;\G}).
\]
\end{thm}
The main argument goes as follows. We begin by establishing the $\bH^1 \times \LL^2$-well-posedness as a consequence of the classical Babu\v{s}ka--Brezzi theory in Hilbert spaces---which hinges on \autoref{lem:Poincare} (Poincar\'e inequality).
In turn, by means of non-trivial commutator identities on tangential calculus and exploiting the lack of boundary of $\G$ (cf.~\autoref{app:calculus}), we obtain a collection of inf-sup conditions required to establish the $\bW^{1,p} \times \LL^p$-well-posedness as a direct application of the generalized Babu\v{s}ka--Brezzi theory in reflexive Banach spaces (see \autoref{app:var-prob}).
More precisely, this is achieved by decoupling the velocity and pressure unknowns into a Laplace--Beltrami problem for the pressure (with a right-hand side depending mildly on the velocity) and a Bochner--Laplace problem for the velocity, so as to be able to use the previously developed $\LL^p$-based ($1<p<\infty$) well-posedness and Sobolev regularity for the Laplace--Beltrami (cf.~\cite{BenavidesNochettoShakipov2025-a}) and Bochner--Laplace operators (cf.~\autoref{thm:bl}), and the $\bH^1 \times \LL^2$-well-posedness of the tangent Stokes problem.
Finally, the same decoupling procedure, combined with the aforementioned properties of the \eqref{eq:Laplace--Beltrami-Poisson} and \eqref{eq:manifold-BochnerPoisson} problems and the $\bW^{1,p} \times \LL^p$-well-posedness of the tangent Stokes problem yields the higher $\LL^p$-based Sobolev regularity results.

Following ideas from \cite[\textsection\textsection 6]{DindosMitrea:2004} and \cite[Chapter 2, \textsection 1]{Temam2001}, we are able to show existence of solutions to the {\it tangent incompressible Navier--Stokes} equations \eqref{eq:manifold-Navier--Stokes} for dimensions $d \in \{2,3,4\}$ and integrability parameter $p \in [2,\infty)$.
To that end, we first focus our attention to a ``linearized'' version of \eqref{eq:manifold-Navier--Stokes}, the so-called {\it tangent Oseen} equations.
By means of the Fredholm alternative in reflexive Banach spaces \cite[Theorem 6.6]{Brezis2011}, we establish the well-posedness of the Oseen equations.
This is the subject of \autoref{sec:Oseen}.

On the other hand, the existence of solutions to \eqref{eq:manifold-Navier--Stokes} for $p=2$ follows from a standard Galerkin method in terms of eigenfunctions of the Stokes operators (cf.~\autoref{sec:Stokes-operator}).
For $p \in (2,\infty)$, we increase the integrability of solutions as a consequence of the well-posedness of the Oseen equations.
This is the content of \autoref{sec:existence-incompressibleNS-d234}.

\begin{thm}[existence of solutions of the incompressible tangent Navier--Stokes equations]\label{thm:ns1}
Let $d \in \{2,3,4\}$ and $p \in [2,\infty)$.
Assume $\G$ is of class $C^2$.
Then, for each $\bf \in (\bW^{1,p^*}_t(\G))'$ there exists a solution $(\bu,\uppi) \in \bW^{1,p}_t(\G) \times \LL^p_\#(\G)$ of
\begin{equation*}
\begin{aligned}
\int_\G \nablaG \bu : \nablaG \bv
+ \int_\G (\nablaG \bu)\bu \cdot \bv - \int_\G \uppi \divG \bv & = \langle \bf , \bv \rangle, \quad & \forall\, \bv \in \bW^{1,p^*}_t(\G),\\
- \int_\G q \divG \bu & = 0, \quad & \forall\, q \in \LL^{p^*}_\#(\G).
\end{aligned}
\end{equation*}
\end{thm}
For arbitrary dimensions $d \geq 2$, the well-posedness of \eqref{eq:manifold-Navier--Stokes} is obtained in \autoref{sec:ns-smalldata} by using the Banach fixed-point theorem under a restricted range of $p$ and a smallness assumption on data.
\begin{thm}[tangent Navier--Stokes with small data] \label{thm:ns2}
Let $\G$ be of class $C^2$ and 
\[
p \in
\begin{cases}
[\frac{4}{3},\infty) & \text{if $d=2$},\\
(\frac{3}{2}, \infty) & \text{if $d=3$},\\
[\frac{d}{2},\infty) & \text{if $d \geq 4$}.
\end{cases}
\]
Then, for every $(\bf, g) \in (\bW^{1,p^*}_t(\G))' \times \LL^{p}_\#(\G)$ such that $\|\bf\|_{(\bW^{1,p^*}_t(\G))'} + \|g\|_{0,p;\G} $ is sufficiently small, there exists a unique pair $(\bu, \uppi) \in \bW^{1,p}_t(\G) \times \LL^p_\#(\G)$ with $\|\bu\|_{1,p;\G}$ sufficiently small such that
\begin{equation}\label{eq:s-weak}
\begin{aligned}
\int_\G \nablaG \bu : \nablaG \bv
+ \int_\G (\nablaG \bu)\bu \cdot \bv - \int_\G \uppi \divG \bv & = \langle \bf , \bv \rangle
, \quad & \forall\, \bv \in \bW^{1,p^*}_t(\G),\\
- \int_\G q \divG \bu & = -\int_\G g \, q, \quad & \forall\, q \in \LL^{p^*}_\#(\G).
\end{aligned}
\end{equation}
Moreover, there exists a constant $C = C(\G, p, \bf, g) > 0$ such that
\[
\|\bu\|_{1,p;\G} + \|\uppi\|_{0,p;\G} \leq C.
\]
\end{thm}
We bootstrap the aforementioned well-posedness and higher-regularity theory for the tangent Stokes problem (cf.~\autoref{thm:stokes}) to deal with the nonlinearity $(\nablaG\bu) \bu$ of \eqref{eq:manifold-Navier--Stokes}, resulting in the following higher regularity result for the tangent Navier--Stokes equations (cf.~\autoref{thm:Lpbased-regularity-NS}).
\begin{thm}[higher regularity of the tangent Navier--Stokes equations]\label{thm:ns3}
Let
\[
\begin{cases}
p \in (\frac43, \infty), &\qquad \text{if $d = 2$}, \\
p \in (\frac{d}{2}, \infty), &\qquad \text{if $d \geq 3$}.
\end{cases}
\]
Assume $\G$ is of class $C^{m+2,1}$ for some nonnegative integer $m$.
Let $(\bu, \uppi) \in \bW^{1,p}_t(\G) \times \LL^p_\#(\G)$ be a solution of \eqref{eq:s-weak} with right-hand side $(\bf, g) \in \bW^{m,p}_t(\G) \times \WW^{m+1,p}_\#(\G)$.
Then, $(\bu, \uppi)$ actually lies in $\bW^{m+2, p}_t(\G) \times \WW^{m+1, p}_\#(\G)$ and satisfies the PDE in strong form: $ -\DeltaB \bu + (\nablaG \bu) \bu + \nablaG \uppi = \bf$, $\divG \bu = g$ a.e.~on $\G$.
Moreover, the following estimate holds
\[
\|\bu\|_{m+2,p;\G} + \|\uppi\|_{m+1,p;\G} \leq F_{m,p}\big(\|\bf\|_{m,p;\G},\|g\|_{m+1,p;\G}, \|\bu\|_{1,p;\G}\big),
\]
where $F_{m,p}: [0,\infty)^3 \rightarrow [0,\infty)$ is a smooth nonlinear function, strictly increasing in each of its three components, and satisfies $F_{m,p}(0,0,0) = 0$.
\end{thm}
\subsubsection*{Connection to other Laplacians}
As it has already been mentioned, other operators can be chosen in lieu of the Bochner--Laplace $\DeltaB$ operator, such as the Surface Diffusion operator $\DeltaS= \bP \divG 2D_\G$ and the Hodge Laplacian $\DeltaH$, as the leading second-order differential operator in problems \eqref{eq:manifold-BochnerPoisson}, \eqref{eq:manifold-stokes} and \eqref{eq:manifold-Navier--Stokes};
their difference from the Bochner--Laplace operator is explicitly stated in \eqref{all-Laplace-relationship}.
Our preference for $\DeltaB$ is mostly for theoretical convenience, as the resulting weak formulations do not need to be analyzed ``away from'' kernel spaces, e.g.~\cite[Section 4]{JankuhnOlshanskiiReusken2018}.
Using identities \eqref{all-Laplace-relationship}, we show in \autoref{sec:connection-other-laplacians} that solutions (provided they exist) to the aforementioned problems with $\DeltaS$ or $\DeltaH$ in place of $\DeltaB$ also enjoy $\LL^p$-based higher Sobolev regularity similar to $\DeltaB$.

\subsubsection*{Summary of the main contributions}
Since some of the results that we derived appear in various forms in the literature, we would like to emphasize which results, to the best of our knowledge, can \textit{only} be found in this paper.
\emph{For $d$-dimensional compact manifolds $\G$ without boundary, we make the following contributions}:
\begin{enumerate}[$\bullet$]
\item $\LL^p$-based well-posedness of weak formulations for problems \eqref{our-problems-of-interest} with minimal regularity of data, namely, $\G$ of class $C^2$, $\bf \in (\bW^{1,p^*}_t(\G))'$ and $g \in \LL^p_\#(\G)$.
Notice that this is not an immediate consequence of the $\LL^2$-based well-posedness or $\LL^p$-based higher regularity.
\item $\bW^{m+2,p}_t(\G) \times \WW^{m+1,p}_\#(\G)$ regularity ($m \geq 0$, $1 < p < \infty$) for the tangent Stokes problem for dimensions $d \geq 2$ provided $\G$ is of class $C^{m+2,1}$.
The only similar result that we know of is for $C^{2,1}$-surfaces of dimension $d=2$, $m=0$ and $p=2$ (for the Hodge Laplacian $\DeltaH$ and surface diffusion operator $\DeltaS$) \cite[Lemma~2.1]{ORZ2021}.
\item $\bW^{m+2, p}_t(\G) \times \WW^{m+1, p}_\#(\G)$ regularity $(m \geq 0)$ for the Navier--Stokes problem for  $p \in (p_0,\infty)$ with $p_0 = p_0(d)$ sufficiently far from $1$, provided $\G$ is of class $C^{m+2,1}$.
\end{enumerate}
Additionally, we would like to emphasize that because we are in a setting of manifolds with limited regularity, our proof techniques differ from those for $C^\infty$-manifolds in, e.g., \cite{Taylor2011, MitreaTaylor:2001, DindosMitrea:2004}. Our approach is more transparent from the perspective of extrinsic calculus, which for the most part hinges on elementary algebraic manipulations combined with functional analytic arguments.

The rest of the paper is structured as follows.
We start in \autoref{sec:difgeo-Sobspace} with basic concepts from differential geometry and Sobolev spaces on manifolds.
We summarize in \autoref{sec:LBeltrami} some of our results from \cite{BenavidesNochettoShakipov2025-a} regarding the Laplace--Beltrami problem that will prove useful throughout the rest of the manuscript.
\autoref{sec:bl} is devoted to proving \autoref{thm:bl} (Bochner--Laplace), while \autoref{sec:stokes} is mostly devoted to proving \autoref{thm:stokes} (tangent Stokes).
\autoref{sec:ns} contains the proofs of \autoref{thm:ns1}, \autoref{thm:ns2} and \autoref{thm:ns3} (tangent Navier--Stokes).
Finally, in \autoref{sec:connection-other-laplacians} we discuss the applicability of our regularity results to other choices of Laplace operators on $\G$.

\section{Preliminaries}

In this section, we review results from differential geometry on hypersurfaces, Sobolev spaces on $\G$, prove a crucial geometric identity (cf.~\autoref{lem:exploiting-symmetry-covarianthessian}), and review elliptic theory results for the \eqref{eq:Laplace--Beltrami-Poisson} problem.
We also reconcile several definitions of Sobolev spaces, which turn out to be useful depending on the context and regularity of $\G$.

\subsection{Differential geometry and Sobolev spaces on \texorpdfstring{$\G$}{Γ}}\label{sec:difgeo-Sobspace}
In this section we collect classical results on the calculus on manifolds (see, e.g.~\cite[\textsection 2]{DziukElliott2013}, \cite[\textsection 1.2]{BonitoDemlowNochetto2020} and \cite{Shahshahani2016}) and introduce and discuss properties of Sobolev spaces on $\G$.
We shall be brief, as we have discussed calculus on hypersurfaces to a great extent in \cite[Section 2]{BenavidesNochettoShakipov2025-a}.
We assume connectedness for convenience, otherwise we could argue independently for each connected component.
Let $\G$ be a $d$-dimensional ($d \geq 2$) compact and connected manifold without boundary embedded in $\RR^{d+1}$ of class $C^{0,1}$ (see \cite[p.~102]{Shahshahani2016} for the definition of $C^m(\G)$ for $m \in \NN$, which naturally extends to $C^{m-1,1}$).
Since $\G$ is compact, there exists a finite atlas $\{(\cV_i,\cU_i,\wt\bchi_i)\}_{i=1}^N$, where each of the charts $\wt\bchi_i: \cV_i \rightarrow \cU_i \cap \G$ are isomorphisms of class $C^{0,1}$ compatible with the orientation of $\G$.
The sets $\cV_i$ are open connected subsets of $\RR^d$ and the $\cU_i$ are open subsets of $\RR^{d+1}$.
An a.e.~defined unit normal vector field $\bnu = (\nu_i)_{i=1}^{d+1}: \G \rightarrow \RR^{d+1}$ to $\G$ is given by $\bnu := \bN_i/|\bN_i|$ (the normalization makes $\bnu$ parametrization-independent), where $\bN_i \circ \wt\bchi_i := \sum_{j=1}^{d+1} \det(\be_j, \nabla \wt\bchi_i) \be_j$ for each $i = 1,\dotsc,N$; here $\{\be_j\}_{j=1}^{d+1}$ is the canonical basis of $\RR^{d+1}$.
We denote by $\bP := \bI - \bnu \otimes \bnu$ the projection operator onto the tangent plane to $\G$.

\begin{definition}[Sobolev spaces on $\G$]\label{def:Sobolev-spaces-manifold}
Assume $\G$ of class $C^{m-1,1}$, for some $m \in \NN$, and let $\{(\cV_i,\cU_i,\wt\bchi_i)\}_{i=1}^N$ be a finite atlas of $\G$.
For $p \in [1,\infty)$, we define the Sobolev space $\WW^{m,p}(\G)$ by
\begin{equation*}
\WW^{m,p}(\G) := \{ u:\Gamma \rightarrow \RR \mid u \circ \wt\bchi_i \in \WW^{m,p}(\cV_i), \quad \forall i=1,\dotsc,N\},
\end{equation*}
endowed with the norm
\begin{equation}\label{scalar-Sobolev-norm}
\|u\|_{m,p;\G} := \left( \sum_{i=1}^N \|u \circ \wt\bchi_i\|_{m,p;\cV_i}^p \right)^{1/p}.
\end{equation}
\end{definition}
We know (e.g.~\cite[Proposition 3.2]{BenavidesNochettoShakipov2025-a}) that the definition of $\WW^{m,p}(\G)$ is independent of the chosen atlas.
Moreover, given two finite atlases of $\G$, their induced norms \eqref{scalar-Sobolev-norm} are topologically equivalent.

It is easy to see that $C^{m-1,1}(\G) \subseteq \WW^{m,p}(\G)$.
We write $\bW^{m,p}(\G) := [\WW^{m,p}(\G)]^{d+1}$ and endow it with its natural product-space norm.
For $p = 2$, we also write $\bH^{m}(\G) := \bW^{m,2}(\G)$.

For $\G$ of class $C^{0,1}$ and $p \in [1,\infty)$, it is possible to define differential operators \emph{exterior gradient} $\nablaM$, \emph{covariant gradient} $\nablaG := \bP \nablaM$, \emph{exterior divergence} $\div_M$ and \emph{covariant divergence} $\divG$ (the latter two coincide) acting on elements of $\WW^{1,p}(\G)$ (and its vector- and tensor-valued versions) by extending the a.e.~valid formulas for classically differentiable functions on $\G$.
Similarly, if $\G$ is of class $C^{m-1,1}$ for some $m \in \NN$, the appropriate composition of the aforementioned operators allows us to define differential operators of order $m$ acting on elements of $\WW^{m,p}(\G)$;
in particular, the second order Laplace--Beltrami operator acting on scalar functions is defined by $\DeltaG := \divG \nablaG$.
We omit further details and simply refer to \cite{BenavidesNochettoShakipov2025-a} for their explicit formulas and further discussion.

It is sometimes convenient to write $\nablaG v = (\uD_1 v, \dotsc, \uD_{d+1} v)$ for the $d+1$ components of $\nablaG v = \nablaM v$, where $v$ is a scalar field.
This notation allows us, in particular, to express the action of $\divG$ and $\DeltaG$ as
\begin{equation*}
\divG \bv = \sum_{k=1}^{d+1} \uD_k v_k, \qquad \qquad \DeltaG v = \sum_{k=1}^{d+1} \uD_k \uD_k v.
\end{equation*}
For $\G$ of class $C^{1,1}$, we write $\bB := \nablaM \bnu = \nablaG \bnu \in \bbL^\infty(\G)$ called the \emph{shape operator} of $\G$ (aka \emph{Weingarten map}).
Since $|\bnu|^2 = 1$, it follows that $\bB^T \bnu = \mathbf{0}$ a.e.~on $\G$.
Notice that for $\G$ of class $C^2$, it can be proved that $\bB$ coincides with the Hessian matrix of the distance function to $\G$, and so it is symmetric.

If $\G$ is of class $C^2$, we can express $\nablaM$ and $\nablaG$ (and all other aforementioned differential operators on $\G$) in a purely extrinsic---that is, independent of parametrizations---manner:
\begin{gather*}
\nablaG v = \nablaM v = \bP \left(\nabla v^e\right)\rvert_\G,\\
\nablaM \bv = \left(\nabla \bv^e\right)\rvert_\G \bP, \qquad \nablaG \bv = \bP \left(\nabla \bv^e\right)\rvert_\G \bP,
\end{gather*}
where $v^e$ (resp.~$\bv^e$) denotes any extension of $v$ (resp., $\bv$) to a $(d+1)$-dimensional tubular neighborhood $\Omega_\delta$ of $\G$ with the following property:
if $v:\G \rightarrow \RR$ is of class $C^2$, then so is $v^e: \Omega_\delta \rightarrow \RR$.
In this case, it holds that
\begin{equation}\label{symmetry-cov-Hes}
\nablaG \nablaG v = \bP \nablaM \nablaM v = \bP (\nabla\nabla v^e)\rvert_\G \bP - ((\nabla v^e)\rvert_\G \cdot \bnu) \bB,
\end{equation}
which shows that the \emph{covariant Hessian} $\nablaG\nablaG v$ is symmetric.
An instance of such extension is given by \cite[eq.~(1.62)]{BonitoDemlowNochetto2020}.
Notice that, if $\G$ is not better than of class $C^2$, then the \emph{constant normal} extension is not a valid choice---if $v: \G \rightarrow \RR$ is of class $C^2$, then $v^e$ is not necessarily of class $C^2$.
Formula \eqref{symmetry-cov-Hes} seems to be well-known within the community of numerical methods for surface PDEs for $\G$ of class $C^3$ and the constant normal extension, for which the second term in the right-hand side vanishes.
We are, however, unaware of a specific reference for the proof.
For completeness, we prove \eqref{symmetry-cov-Hes} in all its generality in \autoref{prop:sym-cov-Hess}.

Notice that in general $\uD_i \uD_j v \neq \uD_j \uD_i v$ for each $i \neq j$.
In fact, if $\G$ and $v: \G \rightarrow \RR$ are of class $C^2$, we have that
\begin{equation}\label{commutators:second-tangential-derivatives}
\uD_i \uD_j v - \uD_j \uD_i v = (\bB \nablaM v)_j \nu_i - (\bB \nablaM v)_i \nu_j,
\end{equation}
for all $i,j=1,\dotsc,d+1$;
we refer to \cite[Lemma~2.6]{DziukElliott2013}, which states the result without a proof.
For completeness, we provide a short proof in \autoref{prop:commutator-identity}.

The following result states that Sobolev spaces on $\G$ behave well under multiplication by smooth-enough functions.
\begin{lemma}[product rule on $\WW^{m,p}(\G)$]\label{lem:product-by-smooth}
Assume $\G$ is of class $C^{m-1,1}$ for some $m \in \NN$, and let $p \in [1,\infty)$.
Then for each $\theta \in C^{m-1,1}(\G)$ and $u \in \WW^{m,p}(\G)$, the product $\theta u$ belongs to $\WW^{m,p}(\G)$,
\begin{equation*}
\nablaM (\theta u) = \theta \nablaM u + u \nablaM \theta, \quad \text{a.e.~on $\G$},
\end{equation*}
and there exists a constant $C_\theta>0$, depending only on $\theta$, such that
\begin{equation*}
\|\theta u\|_{m,p;\G} \leq C_\theta \|u\|_{m,p;\G}.
\end{equation*}
\end{lemma}
\begin{thm}[density]\label{thm:density-C^m}
Let $p \in [1,\infty)$ and $m \in \NN$.
If $\G$ is of class $C^{m-1,1}$ (resp.~$C^m$), then $C^{m-1,1}(\G)$ (resp.~$C^m(\G)$) is dense in $\WW^{m,p}(\G)$.
\end{thm}
It is helpful to endow $\WW^{m,p}(\G)$ with an equivalent norm that is parametrization-independent and resembles its flat-domain counterpart.
This is the content of the following result.

\begin{proposition}[equivalent norm in $\WW^{m,p}(\G)$]\label{prop:param_ind-Sobolev-norm}
Assume $\G$ is of class $C^{m-1,1}$, for some $m \in \NN$, and let $p \in [1,\infty)$.
Then, the mapping
\begin{equation}\label{param_ind-scalar-Sobolev-norm}
u \mapsto \left(\sum_{k=0}^m \|\nablaM^k u \|_{0,p;\G} ^p\right)^{1/p},
\end{equation}
defines an equivalent norm in $\WW^{m,p}(\G)$.
Here, $\nablaM^k$ denotes the \emph{$k$-fold weak extrinsic gradient}---i.e., $\nablaM^k u$ is the $k$-dimensional tensor defined by $(\nablaM^k u)_\bj =
\uD_{j_{k}} \cdots \uD_{j_1} u$ for each $\bj = (j_i)_{i=1}^k \in \{1,\dotsc,d+1\}^k$;
we also write $\nablaM^0 u := u$.
\end{proposition}
Throughout the rest of the paper, it will be usually convenient to utilize \eqref{param_ind-scalar-Sobolev-norm} as the ``main'' norm on $\WW^{m,p}(\G)$ and consequently we will also denote it by $\|\cdot\|_{m,p;\G}$ when no confusion arises.
Notice that this equivalent norm naturally induces a notion of Sobolev norms on subsets of $\G$:
let $\gamma \subseteq \G$ be a measurable set with respect to the Lebesgue measure on $\G$, we write
\begin{equation*}
\|v\|_{m,p;\gamma} := \left(\sum_{k=0}^m \|1_\gamma \, \nablaM^k v \|_{0,p;\G} ^p\right)^{1/p}, \qquad \forall v \in \WW^{m,p}(\G),
\end{equation*}
where $1_\gamma$ denotes the characteristic function of $\gamma$.

We now compare the adopted definition \autoref{def:Sobolev-spaces-manifold} of Sobolev spaces and a few others found in the literature.

\begin{remark}[$\WW^{m,p}(\G)$ coincides with $H^p_m(\G)$ of Hebey and Robert \cite{HebeyRobert2008}]
\autoref{thm:density-C^m} and \autoref{prop:param_ind-Sobolev-norm} show that $\WW^{m,p}(\G)$ (cf.~\autoref{def:Sobolev-spaces-manifold}) coincides with the space $H^p_m(\G)$ of \cite[Definition 2.1, Proposition 2.2]{HebeyRobert2008}.
\end{remark}
\begin{proposition}[integration by parts on $C^{1,1}$-manifolds]
If $\G$ is of class $C^{1,1}$, then for each $u \in \WW^{1,1}(\G)$ and $\bvarphi \in \bC^1(\G)$ it holds that
\begin{equation}\label{eq:int-by-parts}
\int_\Gamma u  \, \divG \bvarphi = - \int_\Gamma \nablaG u \cdot \bvarphi + \int_\Gamma \tr(\bB) u \bvarphi \cdot \bnu.
\end{equation}
\end{proposition}
This is consistent with the definition of Sobolev spaces of Dziuk and Elliott \cite{DziukElliott2013} for $\G$ of class $C^2$.
We proved \eqref{eq:int-by-parts} for $\G$ of class $C^{1,1}$ in \cite[Proposition 3.9]{BenavidesNochettoShakipov2025-a}.

We introduce the following spaces
\begin{align*}
\WW^{m,p}_\#(\G) & := \{\textstyle  q \in \WW^{m,p}(\G): \int_\G q = 0 \},\\
\bW^{m,p}_t(\G) & := \{ \bv \in \bW^{m,p}(\G): \bv \cdot \bnu = 0~\text{a.e.~on $\G$} \},
\end{align*}
which are respectively closed subspaces of $\WW^{m,p}(\G)$ and $\bW^{m,p}(\G)$.
Notice that, if $\G$ is of class $C^{m,1}$ for some nonnegative integer $m$, $\WW^{m+1,p}(\G) \subseteq \WW^{m,p}(\G)$ and $\nablaM u \in \bW^{m,p}_t(\G)$ for each $u \in \WW^{m+1,p}(\G)$.
The following density result, which is a slight improvement over \cite[Lemma~3.7]{Miura-PartI}, will be instrumental in this work.
\begin{lemma}[density]\label{lem:density-tangential}
Let $p \in [1,\infty)$, $m \in \NN$ and an integer $0 \leq k \leq m-1$.
If $\G$ is of class $C^m$, then the space of everywhere tangential $C^{m-1}$-vector fields $\bC^{m-1}_t(\G)$ is dense in $\bW^{k,p}_t(\G)$.
Similarly, if $m \geq 2$ and $\G$ is of class $C^{m-1,1}$, then the space of a.e.~tangential $C^{m-2,1}$-vector fields $\bC^{m-2,1}_t(\G)$ is dense in $\bW^{k,p}_t(\G)$.
\end{lemma}

If $\G$ is of class $C^2$, from \cite[eq.~(2.17)]{JankuhnOlshanskiiReusken2018} and \autoref{lem:density-tangential} we know that
\begin{equation}\label{WeingartenExtr-for-tangent}
\bB \bv = - (\nablaM^T \bv) \bnu, \qquad \forall \bv \in \bC^1_t(\G).
\end{equation}
Hence, recalling that $\nablaG \bv = \bP \nablaM \bv$ and $\bP = \bI - \bnu \otimes \bnu$, we obtain that
\begin{equation}\label{ExtCov-for-tangent}
\nablaM \bv = \nablaG \bv - \bnu \bv^T \bB, \qquad \forall \bv \in \bC^1_t(\G).
\end{equation}
This shows that for tangential vector fields, $\nablaM$ and $\nablaG$ only differ in lower order terms.

For the rest of this section $p \in [1,\infty)$ and $m \in \NN$ are arbitrary unless stated otherwise, and whenever we mention the space $\WW^{m,p}(\G)$ we implicitly assume $\G$ is of class $C^{m-1,1}$.

Since $\G$ is compact, it follows that the Sobolev embeddings of the first type (Gagliardo--Nirenberg--Sobolev) on $\G$ behave just like their counterparts for a bounded open subset of $\RR^d$ with $C^1$ boundary.
\begin{proposition}[Gagliardo–Nirenberg–Sobolev inequality {\cite[\textsection 4.3]{HebeyRobert2008}}]\label{prop:GNG-embeddings}
Let $1 \leq p < q < \infty$, and any integers $0 \leq m < k$ such that
\begin{equation*}
k - \frac{d}{p} \geq m - \frac{d}{q}.
\end{equation*}
Then, $\WW^{k,p}(\G)$ is continuously embedded in $\WW^{m,q}(\G)$.

As a consequence, if $kp < d$ and $1 \leq r \leq \frac{dp}{d-kp}$, then $\WW^{k,p}(\G)$ is continuously embedded in $\LL^{r}(\G)$.
\end{proposition}

Morrey's inequality is also valid on $\G$.
We refer to \cite[p.~389]{HebeyRobert2008} for the precise definition of the H\"{o}lder space $C^{0,\alpha}(\G) \subseteq C(\G)$ ($0 < \alpha < 1$).
\begin{proposition}[Morrey's inequality {\cite[Theorem 6.2]{HebeyRobert2008}}]\label{prop:Morrey-inequality}
Let $p > d$.
Then, $\WW^{1,p}(\G)$ is continuously embedded in $C^{0,\alpha}(\G)$, where $\alpha := 1 - \frac{d}{p}$.
\end{proposition}
The Rellich--Kondrachov compactness theorem is valid on $\G$.

\begin{proposition}[compact Sobolev embedding {\cite[Theorem 8.1]{HebeyRobert2008}}]\label{prop:Rellich-Kondrachov}
Let $1 \leq p < \infty$ and integers $0 \leq m < k$ such that  $(k-m)p < d$.
Then, $\WW^{k,p}(\G)$ is compactly embedded in $\WW^{m,q}(\G)$ for each $1 \leq q < \frac{dp}{d-(k-m)p}$.
In particular, $\WW^{1,p}(\G)$ is compactly embedded in $\LL^q(\G)$ for $1 \leq p < d$ and $1 \leq q < \frac{dp}{d-p}$.
\end{proposition}

\begin{proposition}[compact embedding in H\"older spaces {\cite[Theorem 8.2]{HebeyRobert2008}}]\label{prop:Holder-compactly-embedded}
Let $d < p < \infty$.
Then, $\WW^{1,p}(\G)$ is compactly embedded in $C^{0,\alpha}(\G)$ for all $\alpha \in (0,1-\frac{d}{p})$.
\end{proposition}

\begin{remark}[compact Sobolev embedding]\label{rem:W1p-compactly-embedded-in-L^p}
Since $\frac{dp}{d-p} \to \infty$ as $p \to d$, we have from \autoref{prop:Rellich-Kondrachov}, \autoref{prop:Holder-compactly-embedded}, and the Arzel\`a--Ascoli theorem \cite[Theorem 4.25]{Brezis2011}, that $\WW^{1,p}(\G)$ is compactly embedded in $\LL^p(\G)$ for all $1 \leq p < \infty$.
\end{remark}

Recall the following identity valid for $\G$ of class $C^3$ and for each $\bv \in \bC^2_t(\G)$ \cite[Lemma 2.1]{JankuhnOlshanskiiReusken2018}:
\begin{equation*}
\bP \divG \nabla^T_\G \bv = \nablaG \divG \bv + \left(\tr(\bB)\bB - \bB^2\right)\bv.
\end{equation*}
We will show that, loosely speaking, said identity is still valid in a weak sense for $\G$ of class $C^2$ if we replace $\nablaG^T \bv$ by $\nablaG \bv$ on the left-hand side (hence obtaining the Bochner Laplacian $\DeltaB \bv = \bP \divG \nablaG \bv$) as long as we test against gradients of scalar functions.
This elementary observation, which relies on $\G$ not having a boundary, will play a crucial role throughout this paper.
We refer to \autoref{app:calculus} for its proof.
\begin{lemma}[weak commutator identity]\label{lem:exploiting-symmetry-covarianthessian}
Assume $\G$ is of class $C^2$.
Then for each $\phi \in C^2(\G)$ and $\bv \in \bC^1_t(\G)$ there holds that
\begin{equation}\label{exploiting-symmetry-covarianthessian}
- \int_\G \nablaG \bv : \nablaG\nablaG \phi = - \int_\G \divG \bv \, \DeltaG \phi + \int_\G \left(\tr(\bB)\bB - \bB^2\right) \bv \cdot \nablaG \phi.
\end{equation}
\end{lemma}

\subsection{The Laplace-Beltrami problem}\label{sec:LBeltrami}
In this section, we summarize some of our main results concerning the Laplace--Beltrami problem posed on $\G$ from the companion paper \cite{BenavidesNochettoShakipov2025-a}.
We begin with the strong formulation:
\begin{equation}
\label{eq:Laplace--Beltrami-Poisson} \tag{LB} -\DeltaG u = f,
\end{equation}
where $\DeltaG$ is the Laplace--Beltrami operator $\DeltaG = \divG \nablaG$. Before we begin, the following result will be of use in the forthcoming analysis.
\begin{proposition}[surjectivity of the divergence operator]\label{prop:surj-div}
Let $p \in (1,\infty)$ and assume $\G$ is of class $C^{0,1}$.
Then, the linear and continuous operator $\divG : \bL^p_t(\G) \rightarrow (\WW^{1,p^*}_\#(\G))'$ defined by
\begin{equation*}
\langle \divG \bu, v \rangle = -\int_\G \bu \cdot \nablaG v, \qquad \forall \bu \in \bL^p_t(\G), \ \forall v \in \WW^{1,p^*}_\#(\G),
\end{equation*}
is surjective.
\end{proposition}
We now state the main result from \cite{BenavidesNochettoShakipov2025-a} concerning well-posedness and higher regularity of solutions to the Laplace--Beltrami problem.
\begin{thm}[Laplace--Beltrami] \label{thm:lb}
Let $p \in (1,\infty)$ and assume that $\G$ is of class $C^1$ or of class $C^{0,1}$ if $p \in (2-\varepsilon,2+\varepsilon)$ for a certain $\varepsilon>0$ depending only on $\G$.
Then, for every $f \in (\WW^{1,p^*}_\#(\G))'$, there exists a unique $u \in \WW^{1,p}_\#(\G)$ such that
\begin{equation} \label{eq:lb-weak}
\int_\G \nablaG u \cdot \nablaG v = \langle f, v \rangle, \qquad \forall v \in \WW^{1,p^*}_\#(\G).
\end{equation}
Moreover, there exists a constant $C>0$, depending only on $\G$ and $p$, such that
\[
\|u\|_{1,p;\G} \leq C \|f\|_{(\WW^{1,p^*}_\#(\G))'}.
\]
Furthermore, if $\G$ is of class $C^{m+1,1}$ for some nonnegative integer $m$ and $f \in \WW^{m,p}_\#(\G)$, then the solution $u \in \WW^{1,p}_\#(\G)$ of \eqref{eq:lb-weak} is in fact in $\WW^{m+2, p}_\#(\G)$
and satisfies the PDE in strong form: $ -\DeltaG u = f$ a.e. on $\G$.
Moreover, there is $C = C(\G, m, p) > 0$ such that
\[
\|u\|_{m+2,p;\G} \leq C \|f\|_{m,p;\G}.
\]
\end{thm}
The next corollary will be of use when analyzing the Bochner--Laplace problem \eqref{eq:manifold-BochnerPoisson}.
\begin{corollary}\label{cor:LB-W1p-special-RHS}
Let $p \in (1,\infty)$ and $q \in (1,d)$, and define $s := \min\{p,\frac{dq}{d-q}\}$.
Assume $\G$ is of class $C^1$ or of class $C^{0,1}$ if $s \in (2-\varepsilon,2+\varepsilon)$ for a certain $\varepsilon>0$ depending only on $\G$.
Then for each $\bf \in \bL^p_t(\G)$ and $f \in \LL^q_\#(\G)$, there is a unique $u \in \WW^{1,s}_\#(\G)$ such that
\begin{equation*}
\int_\G \nablaG u \cdot \nablaG v = - \int_\G \bf \cdot \nablaG v + \int_\G f \, v, \qquad \forall v \in \WW^{1,s^*}_\#(\G).
\end{equation*}
Moreover, there exists a positive constant $C$, depending only on $\G$, $p$ and $q$ such that
\begin{equation*}
\|u\|_{1,s;\G} \leq C \left( \|\bf\|_{0,p;\G} + \|f\|_{0,q;\G} \right).
\end{equation*}
\end{corollary}
We close this section with the discussion of ``ultra-weak" solutions to \eqref{eq:Laplace--Beltrami-Poisson}, which will be utilized in the analyses of \eqref{eq:manifold-BochnerPoisson} and \eqref{eq:manifold-stokes}.
\begin{lemma}[ultra-weak solutions of the Laplace--Beltrami operator on $\G$]\label{lem:ultraweak-allp-LB}
Let $p \in (1,\infty)$ and assume $\G$ is of class $C^{1,1}$.
Then, for each $f \in (\WW^{2,p^*}_\#(\G))'$ there exists a unique $u \in \LL^p_\#(\G)$ such that
\begin{equation}\label{eq:dual-problem}
-\int_\G u \, \DeltaG v = \langle f, v \rangle, \qquad \forall v \in \WW^{2,p^*}_\#(\G).
\end{equation}
Moreover, there exists a constant $C>0$ depending only on $\G$ and $p$ such that
\begin{equation*}
\|u\|_{0,p;\G} \leq C \|f\|_{(\WW^{2,p^*}_\#(\G))'}.
\end{equation*}
\end{lemma}

\section{The Bochner--Laplace problem} \label{sec:bl}
As was stated in the introduction, the nonlinear nature of the tangent Navier--Stokes equations necessitates studying the tangent Stokes equations in the $\LL^p$-based setting. 
It turns out that to develop the $\LL^p$-based theory for the latter, it suffices to have the $\LL^p$-based theory for the Laplace--Beltrami and the Bochner--Laplace problems at hand. 
We illustrate this idea by the formal calculation below.

Consider the weak formulation of the tangent Stokes problem with $p \in (1,\infty)$: given $(\bf, g) \in (\bW^{1,p^*}_t(\G))' \times \LL^p_\#(\G)$, find $(\bu, \uppi) \in \bW^{1,p}_t(\G) \times \LL^p_\#(\G)$ such that
\begin{equation}\label{motivation:stokes}
\begin{aligned}
\int_\G \nablaG \bu : \nablaG \bv
- \int_\G \uppi \divG \bv & = \langle \bf , \bv \rangle
, \quad & \forall\, \bv \in \bW^{1,p^*}_t(\G),\\
- \int_\G q \divG \bu & = -\int_\G g \, q, \quad & \forall\, q \in \LL_\#^{p^*}(\G).
\end{aligned}
\end{equation}
\emph{Suppose that it is well-posed with solution $(\bu, \uppi)$ depending continuously on the data $(\bf, g)$.} Now, suppose that $(\bf, g) \in \bL^p_t(\G) \times \WW^{1,p}_\#(\G)$. By testing the momentum equation with $\bv = \nablaG \psi$ for some $\psi \in \WW^{2,p^*}_\#(\G)$, we obtain
\[
\int_\G \nablaG \bu : \nablaG \nablaG \psi - \int_\G \uppi \, \DeltaG \psi = \int_\G \bf \cdot \nablaG \psi.
\]
\emph{Using the symmetry of the covariant Hessian (cf.~\eqref{symmetry-cov-Hes}) $\nablaG \nablaG \psi$ and $\divG \bu = g$}, in conjunction with \eqref{exploiting-symmetry-covarianthessian}, one can rewrite the above as
\[
- \int_\G \uppi\,\DeltaG \psi = \int_\G (\bf + \nablaG g + (\tr(\bB)\bB - \bB^2) \bu) \cdot \nablaG \psi, \qquad \forall \psi \in \WW^{2,p^*}_\#(\G).
\]
\emph{If we had the $\LL^p$-based regularity and well-posedness in $\WW^{1,p}(\G)$ for the Laplace--Beltrami problem for $\G$}, then we could infer that
\begin{equation*}
\|\uppi\|_{1,p;\G} \leq C(\G, p) (\|\bf\|_{0,p;\G} + \|g\|_{1,p;\G}),
\end{equation*}
where we have used the well-posedness of problem \eqref{motivation:stokes} to bound $\|\bu\|_{0,p;\G}$ by data. 
With this at hand, we can now go back to the momentum equation to obtain a Bochner--Laplace equation for the velocity:
\[
\int_\G \nablaG \bu : \nablaG \bv = \int_\G (\bf - \nablaG \uppi) \cdot \bv, \quad \forall \bv \in \bW^{1,p^*}_t(\G).
\]
\emph{If we had the $\bW^{2,p}(\G)$-regularity for the Bochner--Laplace problem}, we could conclude
\[
\|\bu\|_{2,p;\G} \leq C(\G, p) (\|\bf\|_{0,p;\G} + \|g\|_{1,p;\G}).
\]
This formal calculation, inspired by \cite[Lemma 2.1]{ORZ2021}, shows how to decouple the equations for pressure and velocity in the tangent Stokes system \eqref{eq:manifold-stokes} in order to first derive $\WW^{1,p}(\G)$-regularity for pressure and next to deduce $\bW^{2,p}(\G)$-regularity for velocity.
Circumventing the saddle-point structure of \eqref{eq:manifold-stokes} is one of the most significant contributions of this paper.
We thus claim that to study \eqref{eq:manifold-stokes}, it is enough to study the Laplace--Beltrami \eqref{eq:Laplace--Beltrami-Poisson} and the Bochner--Laplace \eqref{eq:manifold-BochnerPoisson} problems. 
This motivates the $\LL^p$-based theory developed in this section. 
Of course, all of these calculations will be made rigorous in the current and the forthcoming sections.

\subsection{\texorpdfstring{$\LL^p$}{Lᵖ}-based Sobolev regularity}\label{sec:Lp-regularity-BochnerLaplacian}

Recall that the Bochner--Laplace problem formally reads:
find a velocity field $\bu$ everywhere tangential to $\G$ (i.e.~$\bu \cdot \bnu = 0$ on $\G$) such that
\begin{equation}
-\bP \divG \nablaG \bu = \bf \quad \text{on $\G$}, \tag{\ref{eq:manifold-BochnerPoisson}}
\end{equation}
where $\bf$ is a given everywhere tangential force and $\bP := \bI - \bnu \otimes \bnu$ is the projection operator onto the tangent plane to $\G$.
We stress that the well-posedness and regularity results for the Laplace--Beltrami problem from the previous section do not immediately translate to \eqref{eq:manifold-BochnerPoisson} due to the fact that $\DeltaB = \bP \divG \nablaG = \bP \divM (\bP \nablaM)$ does not coincide with the component-wise application of the Laplace--Beltrami operator $\divM \nablaM$.
In this section we are concerned with the well-posedness of the $\LL^p$-based ($1<p<\infty$) weak formulation for \eqref{eq:manifold-BochnerPoisson}:
find $\bu \in \bW^{1,p}_t(\G)$ such that
\begin{equation}\label{eq:weak-manifold-BochnerPoisson-nonhilbertian}
\int_\G \nablaG \bu : \nablaG \bv = \langle \bf , \bv \rangle, \quad \forall\, \bv \in \bW^{1,p^*}_t(\G).
\end{equation}
We start by showing that replacing $\nablaM$ by the (in principle) weaker $\nablaG = \bP \nablaM$ yields an equivalent norm in $\bW^{1,p}_t(\G)$ for $p \in (1,\infty)$ (and in particular, in $\bH^1_t(\G) = \bW^{1,2}_t(\G)$).

\begin{proposition}[norm equivalence]\label{prop:equivalence-of-norms}
Let $p \in (1,\infty)$ and assume $\G$ is of class $C^2$.
Then,
\begin{equation}\label{eq:equivalence-of-norms}
\|\bv\|_{1,p;\G} \leq C (\|\nablaG \bv\|_{0,p;\G} + \|\bv\|_{0,p;\G}), \qquad \forall \bv \in \bW^{1,p}_t(\G),
\end{equation}
where the constant $C$ depends only on $p$ and linearly on $\|\bB\|_{0,\infty;\G}$.
\begin{proof}
The proof is a simplified version of \cite[Lemma~4.1]{JankuhnOlshanskiiReusken2018}.
From \eqref{ExtCov-for-tangent} and \autoref{lem:density-tangential} (density) we know that
\begin{equation*}
\nablaM \bv = \nablaG \bv - \bnu \bv^T \bB, \qquad \forall \bv \in \bW^{1,p}_t(\G).
\end{equation*}
This, together with the continuity of $\bB = \nablaM \bnu$ on $\G$, yields \eqref{eq:equivalence-of-norms}.
\end{proof}
\end{proposition}

We will show that the lower order term $\|\bv\|_{0,p;\G}$ can actually be removed from the right-hand side of \eqref{eq:equivalence-of-norms}.
This hinges on the injectivity of the operator $\nablaG$ on $\bW^{1,p}_t(\G)$, which has been previously stated in \cite[Lemma 2.1]{HansboLarsonLarsson2020} for $p=2$ and a $C^\infty$ surface embedded in $\RR^3$ (i.e.~$d=2$).

The proof in \cite[Lemma 2.1]{HansboLarsonLarsson2020} relies on subtle differential geometry manipulations to first prove the injectivity of $\nablaG$ on $\bC^\infty_t(\G)$;
then the proof asserts, without further justification, that said injectivity is transferred via density to $\bH^1_t(\G)$.
This, however, is not straightforward and it does not follow from abstract functional analysis;
there are non-injective linear continuous operators $T: X \rightarrow Y$ between two Hilbert spaces such that the restriction of $T$ to a dense subset of $X$ is injective (e.g.~\cite{MathStack-counterexample-injective-1}).
In the following lemma, we fill in this theoretical gap (see Step 1 within the proof below) and extend the result for less-regular manifolds of dimension $d \geq 2$ and $p \in (1,\infty)$.
Moreover, we note that the ``weak'' Riemann curvature tensor (cf.~\eqref{weak-R-definition}) is instrumental to the proof.

\begin{lemma}[injectivity of $\nablaG$]\label{lem:covariant-gradient-injective}
Let $p \in (1,\infty)$ and assume $\G$ is of class $C^2$.
Then, $\nablaG$ is injective on $\bW^{1,p}_t(\G)$.

\begin{proof}

\textbf{Step 1:}
We start by showing that if $\bv \in \bW^{1,p}_t(\G)$ is such that $\nablaG \bv = \textbf{0}$ a.e.~on $\G$, then $\bv \in \bC^1_t(\G)$.
Let $\bv$ as above.
Then, by \eqref{ExtCov-for-tangent} we know that 
\begin{equation}\label{ExtCov-for-tangent&nullcov}
\nablaM \bv = -\bnu \bv^T \bB.
\end{equation}
It suffices to prove the case $p<d$.
If $\bv \in \bW^{1,p}_t(\G)$, then by \autoref{prop:GNG-embeddings} (Gagliardo--Nirenberg--Sobolev), it transpires that $\bv \in \bL^{p_0}_t(\G)$ for $p_0 = \frac{pd}{d-p}$.
Notice that the $C^2$-regularity of $\G$ implies that $\bnu$ and $\bB$ are respectively $C^1$ and $C^0$ functions and so from \eqref{ExtCov-for-tangent&nullcov} we deduce that $\bv \in \bL^{p_0}(\G)$  implies $\nablaM \bv \in \bbL^{p_0}(\G)$.
Iterating this process leads to sufficiently high integrability of $\bv$ so as to use \autoref{prop:Morrey-inequality} (Sobolev embedding in H\"older spaces) to obtain $\bv \in \bC^0_t(\G)$, and so, again, by \eqref{ExtCov-for-tangent&nullcov}, $\bv \in \bC^1_t(\G)$.
This completes the proof of Step 1.
\\ \\
\textbf{Step 2:}
We claim that:
if $\bv \in \bC^1_t(\G)$ is such that $\nablaG \bv = \textbf{0}$ on $\G$, then $\bv = \textbf{0}$ on a non-empty open set (relative to $\G$).

Fix $\bx \in \G$.
Let $\{\bt_i\}_{i=1}^d$ be a $C^1$ orthonormal basis of the tangent plane near $\bx$;
such basis can be obtained by applying the Gram--Schmidt method to the columns of $(\nabla \wt\bchi) \circ \wt\bchi^{-1}$, where $\wt\bchi$ is a local parametrization of $\G$ near $\bx$.
In what follows, for every (scalar, vector or tensor) function $\bf$ defined near $\bx$, we let $\wt\bf := \bf \circ \wt \bchi$.
Denote by $\bT \in \RR^{(d+1)\times d}$ the matrix field whose $i$-th column is $\bt_i$.
Notice that we cannot select $\{\bt_i\}_{i=1}^d$ to be the eigenframe of $\bB$, as such eigenframe may be discontinuous at $\bx$ if not all of the principal curvatures are simple \cite{Matthew2014}.
For our purposes, we circumvent this technicality as follows:
recall that the action of $\bB$ on the tangent plane of $\G$ can be represented by the $d\times d$ tensor $\bA := \bT^T \bB \bT$.
Let $\{\bp_i\}_{i=1}^d \subseteq \RR^d$ be the eigenframe of $\wt\bA$ at $\by := \wt\bchi^{-1}(\bx)$;
i.e.~$\wt\bA \bp_i  \big\rvert_\by= \wt\kappa_i \bp_i \big\rvert_\by$, where $\{\kappa_i\}_{i=1}^d$ are the principal curvatures of $\G$.
Moreover, since $\bT \bT^T = \bI - \bnu \otimes \bnu$ is the projection onto the tangent plane of $\G$, we have that $\wt\bB (\wt\bT \bp_i) \big\rvert_\by = \wt\bT (\wt\bT^T \wt\bB \wt\bT) \bp_i \big\rvert_\by = \wt\kappa_i (\wt\bT \bp_i) \big\rvert_\by$.
Hence, $\{\br_i := \bT \bp_i \}_{i=1}^d$ is also a $C^1$ orthonormal basis of the tangent plane near $\bx$ and $\bB \br_i \big\rvert_\bx = \kappa_i \br_i \big\rvert_\bx$ for each $i \in \{1,\dotsc,d\}$.

Now, let $\bv \in \bC^1_t(\G)$ be such that $\nablaG \bv = \textbf{0}$ on $\G$.
From \eqref{weak-R-definition} and \eqref{weak-RCT-simplified} (Riemann Curvature Tensor) and a localization argument, we deduce that for each $\ba,\bb \in \bC^1_t(\G)$ it holds that
\begin{equation}\label{eq:RCT-zero}
(\bv^T \bB \bb) \bB \ba - (\bv^T \bB \ba) \bB \bb = 0, \qquad \text{on $\G$}.
\end{equation}
By expressing $\bv$ locally near $\bx$ by $\bv = \sum_{k=1}^d \alpha_k \br_k$, it follows from \eqref{eq:RCT-zero} that for each $i,j = 1,\dotsc,d$:
\begin{equation*}
0
= \left[(\bv^T \bB \br_j) \bB \br_i - (\bv^T \bB \br_i) \bB \br_j\right]\big\rvert_{\bx}
= \left[\kappa_i \kappa_j \left((\bv \cdot \br_j) \br_i - (\bv \cdot \br_i) \br_j\right)\right]\big\rvert_{\bx}
= \left[\kappa_i \kappa_j \left( \alpha_j \br_i - \alpha_i \br_j\right)\right]\big\rvert_{\bx};
\end{equation*}
hence,
\begin{equation}\label{prod-eig-on}
\alpha_j(\bx) \kappa_i(\bx) \kappa_j(\bx) = 0, \qquad \forall i,j=1,\dotsc,d, \, i \neq j.
\end{equation}
Since $\G$ is a $d$-dimensional compact manifold without boundary embedded in $\RR^{d+1}$, we have by \cite[Theorem 4, p.88]{Thorpe1994} (still valid for $C^2$ manifolds \cite{Shiffrin2024private}) and the continuity of the Gauss curvature $K := \prod_{i=1}^d \kappa_i$ on $\G$, that the set $\{\bz \in \G: K(\bz) \neq 0\} = \{\bz \in \G: \kappa_i(\bz) \neq 0, \quad \forall i \in \{1,\dotsc,d\}\}$ is a non-empty open set (relative to $\G$).
For each point $\bx$ in that set, identity \eqref{prod-eig-on} yields $\alpha_j(\bx) = 0$ for each $j=1,\dotsc,d$, whence $\bv(\bx) = \sum_{k=1}^d \alpha_k(\bx) \br_k(\bx) = \textbf{0}$.
\\ \\
\textbf{Step 3:}
Let $\bv \in \bC^1_t(\G)$ be such that $\nablaG \bv = \textbf{0}$ on $\G$ and there is $\bx \in \G$ such that $\bv(\bx) = \textbf{0}$.
We now aim to prove that $\bv = \mathbf{0}$ on $\G$.
We give an alternative and simpler proof than \cite[Lemma 2.1]{HansboLarsonLarsson2020}.
By the chain rule and the fact that $\bP\bv = \bv$, we have that
\begin{equation*}
\nablaM |\bv|^2
= 2 \nablaM^T \bv \, \bv
= 2 \nablaM^T \bv \, \bP \bv
= 2 \nablaG^T \bv \, \bv
= \mathbf{0}, \qquad \text{on $\G$.}
\end{equation*}
Because the kernel of the covariant gradient acting on scalar fields consists of globally constant functions (which easily follows from its parametric formula \cite[eq.~(2.3)]{BenavidesNochettoShakipov2025-a} and the fact that $\G$ is connected), we deduce that $|\bv|$ is constant on $\G$.
Because $\bv(\bx) = 0$, it follows that $|\bv| = 0$ on $\G$.
This finishes the proof.
\end{proof}
\end{lemma}
In the proof above, we needed $\G$ to be of class $C^2$.
This requirement shall propagate to all well-posedness and regularity results in this paper, even in the $\LL^2$-based setting.
It would be interesting to determine if such requirement can be relaxed to $\G$ of class $C^{1,1}$.
We leave this question open for future research.

We are in position to prove a Poincar\'e-type inequality in $\bW^{1,p}_t(\G)$.
\begin{thm}[Poincaré inequality]\label{lem:Poincare}
Let $p \in (1,\infty)$ and assume $\G$ is of class $C^2$.
Then, there is a constant $C > 0$ depending only on $\G$ and $p$ such that
\begin{equation}\label{eq:Poincare}
\|\nablaG \bv\|_{0,p;\G}  \geq C \|\bv\|_{1,p;\G}, \qquad \forall\, \bv \in \bW^{1,p}_t(\G).
\end{equation}
\begin{proof}
This is a consequence of the Peetre--Tartar's Lemma \cite[Lemma~A.38]{ErnGuermond2004} applied to operators $A:= \nablaG \in \sL(\bW^{1,p}_t(\G), \bbL^p(\G))$ and $T:= \textrm{id}:\bW^{1,p}_t(\G) \rightarrow \bL^p(\G)$---which are respectively injective and compact by \autoref{lem:covariant-gradient-injective} and \autoref{rem:W1p-compactly-embedded-in-L^p}---and inequality \eqref{eq:equivalence-of-norms}.
\end{proof}
\end{thm}
A direct application of the Lax--Milgram theorem \cite[Theorem 1.1]{Gatica2014} yields the well-posedness of \eqref{eq:weak-manifold-BochnerPoisson} for $p=2$.
\begin{thm}[well-posedness of the Bochner Laplacian in $\bH^1_t(\G)$]\label{thm:eq:weak-manifold-BochnerPoisson-H1-wellposedness}
Assume $\G$ is of class $C^2$.
For each $\bf \in (\bH^1_t(\G))'$, there exists a unique solution $\bu \in \bH^1_t(\G)$ of \eqref{eq:weak-manifold-BochnerPoisson-nonhilbertian} for $p=2$.
Moreover, there is a constant $C > 0$, depending only on $\G$, such that
\begin{equation}\label{L2based-BochnerLaplacian-apriori}
\|\bu\|_{1,\G} \leq C \|\bf \|_{(\bH^1_t(\G))'}.
\end{equation}
\end{thm}

With the well-posedness of \eqref{eq:weak-manifold-BochnerPoisson} for $p=2$ at hand, we now proceed to establish its well-posedness for any $p \in (1,\infty)$.
In striking contrast to the development of the scalar elliptic theory in \cite{BenavidesNochettoShakipov2025-a},
the arguments presented in this section do not require going through ``localization".
Instead, we exploit the nontrivial relation between the Bochner--Laplace and componentwise Laplace--Beltrami operators induced by the identities in \autoref{prop:weak_BL-LB-relation} in order to keep all arguments ``intrinsic", that is, on $\G$.

Before establishing the well-posedness of problem \eqref{eq:weak-manifold-BochnerPoisson-nonhilbertian}, we point out the following vector-valued analogue of \autoref{prop:surj-div} (surjectivity of the divergence operator), which is an immediate consequence of i) of \cite[Theorem~A.1]{BenavidesNochettoShakipov2025-a} (characterization of injectivity and surjectivity), the duality identification $(\bbL^p_t(\G))' \equiv \bbL^{p^*}_t(\G)$, and \autoref{lem:Poincare} (Poincaré inequality).
\begin{proposition}[surjectivity of the divergence operator]\label{prop:surj-div-vectorfields}
Let $p \in (1,\infty)$ and assume $\G$ is of class $C^2$.
Let
\begin{equation*}
\bbL^p_t(\G) := \{\bA \in \bbL^p(\G): \quad \bP \bA \bP = \bA, \quad \text{on $\G$}\},
\end{equation*}
which is a closed subspace of $\bbL^p(\G)$.
Then, the linear and continuous operator $\divG : \bbL^p_t(\G) \rightarrow (\bW^{1,p^*}_t(\G))'$ given by
\begin{equation*}
\langle \divG \bA, \bv \rangle = -\int_\G \bA : \nablaG \bv, \qquad \forall \bA \in \bbL^p_t(\G), \ \forall \bv \in \bW^{1,p^*}_t(\G),
\end{equation*}
is surjective.
\end{proposition}
\begin{thm}[well-posedness of the Bochner Laplacian in $\bW^{1,p}_t(\G)$]\label{thm:eq:weak-manifold-BochnerPoisson-nonhilbertian-wellposedness}
Let $p \in (1,\infty)$ and assume $\G$ is of class $C^2$.
Then, for each $\bf \in (\bW^{1,p^*}_t(\G))'$ there exists a unique solution $\bu \in \bW^{1,p}_t(\G)$ of \eqref{eq:weak-manifold-BochnerPoisson-nonhilbertian}.
Moreover, there exists a constant $C>0$, depending only on $\G$ and $p$, such that
\begin{equation}\label{Lpbased-BochnerLaplacian-apriori}
\|\bu\|_{1,p;\G} \leq C \|\bf\|_{(\bW^{1,p^*}_t(\G))'}.
\end{equation}
\begin{proof}
We first argue for $p \in (2,\infty)$.
Since $\bf \in (\bW^{1,p^*}_t(\G))' \subseteq (\bH^1_t(\G))'$, we know from \autoref{thm:eq:weak-manifold-BochnerPoisson-H1-wellposedness} that there exists a unique $\bu \in \bH^1_t(\G)$ such that
\begin{equation}\label{H1-VL-rhs-better}
\int_\G \nablaG \bu : \nablaG \bv = \langle \bf , \bv \rangle, \qquad \forall \bv \in \bH^1_t(\G).
\end{equation}
Because of $\bf \in (\bW^{1,p^*}_t(\G))'$, we know from \autoref{prop:surj-div-vectorfields} (surjectivity of the divergence operator) and \cite[Remark~A.3]{BenavidesNochettoShakipov2025-a}
that there exists $\bA_\bf \in \bbL^p_t(\G)$ such that $\langle \bf,\bv \rangle = -\int_\G \bA_\bf : \nablaG \bv$, for each $\bv \in \bW^{1,p^*}_t(\G)$, and
\begin{equation}\label{inv-div_vector-bound}
\|\bA_\bf\|_{0,p;\G} \lesssim \|\bf\|_{(\bW^{1,p^*}_t(\G))'}.  
\end{equation}
Consequently, we can equivalently state \eqref{H1-VL-rhs-better} as
\begin{equation*}
\int_\G \nablaG \bu : \nablaG (\bP\bv) = - \int_\G \bA_\bf : \nablaG (\bP\bv), \quad \forall\, \bv \in \bH^1(\G),
\end{equation*}
or, by means of \autoref{prop:weak_BL-LB-relation} and a density argument, also as
\begin{equation*}
\int_\G \nablaM \bu : \nablaM \bv = \int_\G \bh \cdot \bv + \int_\G \bH : \nablaM \bv, \qquad \forall \bv \in \bH^1(\G);
\end{equation*}
where $\bh := ((\bA_\bf + \nablaM \bu) : \bB) \bnu$ and $\bH := - (\bA_\bf + \bnu \bu^T \bB)$.
In particular, choosing $\bv = \alpha \be_k$ for fixed $k=1,\dotsc,d+1$ and arbitrary $\alpha \in \HH^1(\G)$ yields
\begin{equation*}
\int_\G \nablaG u_k \cdot \nablaG \alpha = \int_\G h_k \, \alpha + \int_\G \bH_k \cdot \nablaG \alpha, \qquad \forall \alpha \in \HH^1(\G),
\end{equation*}
where $\bH_k$ denotes the $k$-th row of $\bH$.
That is, $\overline u_k := u_k - \frac{1}{|\G|} \int_\G u_k \in \HH^1_\#(\G)$ satisfies \eqref{eq:lb-weak}
with right-hand side datum $f_k \in (\HH^1_\#(\G))'$ given by $\langle f_k,\alpha \rangle := \int_\G h_k \, \alpha + \int_\G \bH_k \cdot \nablaG \alpha$.
Since $\bB$ and $\bnu$ are respectively of class $C^0$ and $C^1$ and $\bA_\bf \in \bbL^p(\G)$, and $\bh$ depends affinely on $\nablaM \bu$ and $\bH$ linearly on $\bu$, we can proceed via a bootstrapping argument based upon Sobolev embeddings (cf.~\autoref{sec:difgeo-Sobspace}) and \autoref{cor:LB-W1p-special-RHS} to deduce that for each $k \in \{1,\dotsc,d+1\}$, $u_k \in \WW^{1,p}(\G)$ and 
\begin{equation*}
\|u_k\|_{1,p;\G}
\lesssim \|\bu\|_{1,\G} + \|\bA_\bf\|_{0,p;\G}
\stackrel{\eqref{L2based-BochnerLaplacian-apriori},\eqref{inv-div_vector-bound}}{\lesssim} \|\bf \|_{(\bH^1_t(\G))'} +  \|\bf\|_{(\bW^{1,p^*}_t(\G))'}
\lesssim \|\bf\|_{(\bW^{1,p^*}_t(\G))'},
\end{equation*}
and so $\bu \in \bW^{1,p}_t(\G)$ with the same bound.
Furthermore, by a density argument it is easy to see that $\bu$ satisfies \eqref{eq:weak-manifold-BochnerPoisson-nonhilbertian}.
This finishes the proof for $p \in (2,\infty)$.

Akin to Step 3 in the proof of \cite[Theorem 1.1]{BenavidesNochettoShakipov2025-a}, the proof for $p\in(1,2)$ follows by duality, noting that the operator induced by the left-hand side of \eqref{H1-VL-rhs-better} is formally self-adjoint.
\end{proof}
\end{thm}
We now prove the higher $\LL^p$-based regularity of solutions to \eqref{eq:weak-manifold-BochnerPoisson-nonhilbertian}.

\begin{thm}[$\bW^{m+2,p}$-regularity of the Bochner Laplacian on $\G$]\label{thm:allp-based-regularity-BochnerLaplacian}
Let $p \in (1,\infty)$ and assume $\G$ is of class $C^{m+2,1}$ for some nonnegative integer $m$.
If $\bf \in \bW^{m,p}_t(\G)$, then the solution $\bu \in \bW^{1,p}_t(\G)$ of \eqref{eq:weak-manifold-BochnerPoisson-nonhilbertian} belongs to the space $\bW^{m+2,p}_t(\G)$ and there exists a constant $C>0$, depending only on $\G$, $m$ and $p$, such that
\begin{equation}\label{eq:allp-Lpbased-BochnerLaplacian_higher-apriori}
\|\bu\|_{m+2,p;\G} \leq C \|\bf\|_{m,p;\G}.
\end{equation}
Moreover, $\bu$ strongly satisfies \eqref{eq:manifold-BochnerPoisson} in the sense that
\begin{equation}\label{eq:strong-manifold-BochnerPoisson}
-\bP \divG \nablaG \bu = \bf \quad \text{a.e.~on $\G$}.
\end{equation}
\begin{proof}
We proceed similarly to the proof of \autoref{thm:eq:weak-manifold-BochnerPoisson-nonhilbertian-wellposedness} above but without using \autoref{prop:surj-div-vectorfields}. The combination of \autoref{prop:weak_BL-LB-relation}, a density argument, and integration-by-parts formula \cite[Lemma B.7]{BouckNochettoYushutin2024} shows that $\bu$ satisfies $\int_\G \nablaM \bu : \nablaM \bv = \int_\G \bh \cdot \bv$ for all $\bv \in \bW^{1,p^*}(\G),$ where $\bh := \bf + \divG(\bnu \bu^T \bB) + (\nablaM \bu : \bB)\bnu \in \bL^p(\G)$. By choosing $\bv = \alpha \be_k$ for fixed $k \in \{1,\dotsc,d+1\}$ and arbitrary $\alpha \in \WW^{1,p^*}(\G)$, we have that
$\int_\G \nablaG u_k \cdot \nablaG \alpha = \int_\G h_k \, \alpha$ for all $\alpha \in \WW^{1,p^*}(\G)$. Successive applications of the higher regularity result from \autoref{thm:lb} (Laplace--Beltrami) conclude the proof of \eqref{eq:allp-Lpbased-BochnerLaplacian_higher-apriori}. Finally, since $\bu \in \bW^{2,p}_t(\G)$, we can integrate \eqref{eq:weak-manifold-BochnerPoisson-nonhilbertian} by parts using \cite[Ap.~B]{BouckNochettoYushutin2024} (valid by density) to deduce \eqref{eq:strong-manifold-BochnerPoisson}.
\end{proof}
\end{thm}
Akin to \autoref{lem:ultraweak-allp-LB} (ultra-weak solutions of the Laplace--Beltrami operator on $\G$), we finish this section by establishing the existence and uniqueness of ``ultra-weak'' solutions for the Bochner Laplacian.
\begin{lemma}[ultra-weak solutions of the Bochner Laplacian on $\G$]\label{lem:ultraweak-allp-VL}
Let $p \in (1,\infty)$ and assume $\G$ is of class $C^{2,1}$.
Then, for each $\bf \in (\bW^{2,p^*}_t(\G))'$ there exists a unique $\bu \in \bL^p_t(\G)$ such that
\begin{equation*}
-\int_\G \bu \cdot (\bP \divG \nablaG \bv) = \langle \bf, \bv \rangle, \qquad \forall \bv \in \bW^{2,p^*}_t(\G).
\end{equation*}
Moreover, there exists a constant $C>0$ depending only on $\G$ and $p$ such that
\begin{equation*}
\|\bu\|_{0,p;\G} \leq C \|\bf\|_{(\bW^{2,p^*}_t(\G))'}.
\end{equation*}
\begin{proof}
Since $-\DeltaB := -\bP \divG \nablaG$ is a linear continuous bijection from $\bW^{2,p^*}_t(\G)$ onto $\bL^{p^*}_t(\G) \equiv (\bL^{p}_t(\G))'$ (cf.~\autoref{thm:allp-based-regularity-BochnerLaplacian}), the proof follows by duality similarly to \autoref{lem:ultraweak-allp-LB} (cf.~\cite[Lemma 4.3]{BenavidesNochettoShakipov2025-a}).
\end{proof}
\end{lemma}

\section{Tangent Stokes equations} \label{sec:stokes}

In this section we apply the well-posedness and regularity theory developed in \autoref{sec:bl} for the \eqref{eq:manifold-BochnerPoisson} together with the Laplace--Beltrami theory from \cite{BenavidesNochettoShakipov2025-a} to derive a corresponding theory for the tangent Stokes equations \eqref{eq:manifold-stokes}.
We refer the reader to the introduction of \autoref{sec:bl} for the sketch of the arguments used in this section.
We are concerned with the well-posedness of a general variational formulation for the tangent Stokes equations \eqref{eq:manifold-stokes}.
Let $p \in (1,\infty)$.
For each $\bf \in (\bW^{1,p^*}_t(\G))'$ and $g \in \LL^p_\#(\G)$ consider the problem:
find $(\bu,\uppi) \in \bW^{1,p}_t(\G) \times \LL^p_\#(\G)$ such that
\begin{equation}\label{eq:weak-manifold-stokes-nonHilbert}
\begin{aligned}
\int_\G \nablaG \bu : \nablaG \bv
- \int_\G \uppi \divG \bv & = \langle \bf , \bv \rangle
, \quad & \forall\, \bv \in \bW^{1,p^*}_t(\G),\\
- \int_\G q \divG \bu & = -\int_\G g \, q, \quad & \forall\, q \in \LL_\#^{p^*}(\G).
\end{aligned}
\end{equation}
Problem \eqref{eq:weak-manifold-stokes-nonHilbert} is of the generalized Babu\v{s}ka--Brezzi-type (see \autoref{thm:gen-BB}), whose well-posedness hinges on a collection of inf-sup conditions, which we prove below.
Notice that, for $p=2$, \eqref{eq:weak-manifold-stokes-nonHilbert} reduces to \eqref{eq:weak-manifold-stokes}, for which the classical (and simpler) Babu\v{s}ka--Brezzi theory in Hilbert spaces \cite[Theorem 2.1]{Gatica2014} applies.

We start by establishing the following inf-sup condition that encodes, for arbitrary $p \in (1,\infty)$, the surjectivity of the covariant divergence $\divG$ as an operator acting from $\bW^{1,p^*}_t(\G)$ onto $(\LL^p_\#(\G))' \equiv \LL^{p^*}_\#(\G)$;
compare with \autoref{prop:surj-div} (surjectivity of the divergence operator).

\begin{proposition}[inf-sup condition for $\divG: \bW^{1,p^*}_t(\G) \to \LL^{p^*}_\#(\G)$]\label{prop:inf-sup-divergence-nonHilbert}
Let $p \in (1,\infty)$ and assume $\G$ is of class $C^{1,1}$.
Then, there is a constant $C > 0$, depending on $\G$ and $p$, such that
\begin{equation}\label{eq:inf-sup-divergence-nonHilbert}
\sup_{\substack{\bv \in \bW^{1,p^*}_t(\G)\\\bv \neq \mathbf{0}}} \frac{ \left| \int_\G \uppi \divG \bv \right| }{\|\bv\|_{1,p^*;\G}} \geq C \|\uppi\|_{0,p;\G}, \qquad \forall \uppi \in \LL_\#^p(\G).
\end{equation}
\begin{proof}
Fix $\uppi \in \LL_\#^p(\G)$.
By the higher regularity result within \autoref{thm:lb} (Laplace--Beltrami) we know there exists a unique $\phi \in \WW^{2,p^*}_\#(\G)$ such that $\DeltaG \phi = |\uppi|^{p-2} \uppi - \alpha_\uppi$ a.e.~on $\G$, where $\alpha_\uppi := \frac{1}{|\G|} \int_\G |\uppi|^{p-2} \uppi$.
Furthermore, $\|\phi\|_{2,p^*;\G} \lesssim \||\uppi|^{p-2} \uppi - \alpha_\uppi\|_{0,p^*;\G} \lesssim \|\uppi\|_{0,p;\G}^{p-1}$.
We use the function $\bw := \nablaG \phi \in \bW^{1,p^*}_t(\G)$, that satisfies $\divG \bw = |\uppi|^{p-2} \uppi - \alpha_\uppi$ on $\G$ and $\|\bw\|_{1,p^*;\G} \lesssim \|\phi\|_{2,p^*;\G} \lesssim \|\uppi\|_{0,p;\G}^{p-1}$, to bound the supremum from below.
This yields
\begin{equation*}
\sup_{\substack{\bv \in \bW^{1,p^*}_t(\G)\\\bv \neq \mathbf{0}}} \frac{ \left| \int_\G \uppi \divG \bv \right| }{\|\bv\|_{1,p^*;\G}}
\geq \frac{ \int_\G \uppi \divG \bw }{\|\bw\|_{1,p^*;\G}}
\gtrsim \frac{\|\uppi\|_{0,p;\G}^p - \alpha_\uppi \int_\G \uppi}{\|\uppi\|_{0,p;\G}^{p-1}}
= \|\uppi\|_{0,p;\G},
\end{equation*}
which finishes the proof.
\end{proof}
\end{proposition}

Recalling \autoref{lem:Poincare} (Poincaré inequality), we are now in position of establishing the well-posedness of formulation \eqref{eq:weak-manifold-stokes} (equivalently, \eqref{eq:weak-manifold-stokes-nonHilbert} for $p=2$), which follows from the classical Babu\v{s}ka--Brezzi theory.

\begin{thm}[well-posedness of tangent Stokes in $\bH^1_t(\G) \times \LL^2_\#(\G)$]\label{thm:well-posedness}
Assume $\G$ is of class $C^2$.
For each $\bf \in (\bH^1_t(\G))'$ and $g \in \LL_\#^2(\G)$, there exists a unique $(\bu,\uppi) \in \bH^1_t(\G) \times \LL_\#^2(\G)$ solution of \eqref{eq:weak-manifold-stokes-nonHilbert} for $p=2$.
Moreover, there is a constant $C > 0$, depending only on $\G$, such that
\begin{equation*}
\|\bu\|_{1,\G} + \|\uppi\|_{0,\G} \leq C \left(  \|\bf \|_{(\bH^1_t(\G))'} + \|g\|_{0,\G} \right) .
\end{equation*}
\begin{proof}
Recalling the Poincaré inequality \eqref{eq:Poincare} for $p=2$, the bilinear form $(\bv,\bw) \mapsto \int_\G \nablaG \bv : \nablaG \bw$ is coercive in $\bH^1_t(\G)$.
Hence, in view of the inf-sup condition \eqref{eq:inf-sup-divergence-nonHilbert} for $p=2$, the assertion follows from a straightforward application of the Babu\v{s}ka--Brezzi theory in Hilbert spaces \cite[Theorem 2.1]{Gatica2014} to problem \eqref{eq:weak-manifold-stokes}.
\end{proof}
\end{thm}

The well-posedness of \eqref{eq:weak-manifold-stokes-nonHilbert} for general $p \in (1,\infty)$ is reserved to \autoref{sec:Stokes-wp-generalp}, for which the tools developed in the forthcoming \autoref{sec:Helmholtz} are essential.

\subsection{Helmholtz--Weyl decomposition of \texorpdfstring{$\bW^{m,p}_t(\G)$}{Wᵐᵖ}}\label{sec:Helmholtz}

In this section we show that any tangent vector field in $\bW^{m,p}_t(\G)$ can be uniquely decomposed as the sum of a tangent vector field with zero covariant divergence and the covariant gradient of a scalar field in $\WW^{m+1,p}_\#(\G)$.

Assume $\G$ is of class $C^{m,1}$ for some nonnegative integer $m$.
We define
\begin{align*}
\bW^{m,p}_{t,\sigma}(\G) &:= \{ \bv \in \bW_t^{m,p}(\G): \divG \bv = 0\},
\end{align*}
where for $m = 0$, $\divG \bv$ must be understood as an element of $(\WW^{1,p^*}(\G))'$ induced by the integration-by-parts formula \cite[Lemma B.5]{BouckNochettoYushutin2024}; that is
\begin{equation}\label{eq:def-distr-divergence}
\langle \divG \bv, q \rangle_{(\WW^{1,p^*}(\G))' \times \WW^{1,p^*}(\G)} = - \int_\G \bv \cdot \nablaG q, \qquad \forall q \in \WW^{1,p^*}(\G).
\end{equation}
Notice that for $m \geq 1$, $\divG \bv$ as an element of $(\WW^{1,p^*}(\G))'$ is identified with the function $\divG \bv \in \WW^{m-1,p}_\#(\G)$ via
\begin{equation*}
\langle \divG \bv, q \rangle_{(\WW^{1,p^*}(\G))' \times \WW^{1,p^*}(\G)} 
= \int_\G q \, \divG \bv, \qquad \forall q \in \WW^{1,p^*}(\G),
\end{equation*}
because of integration-by-parts formula \eqref{eq:int-by-parts} and a density argument.
Consequently, the condition $\divG \bv = 0$ in the definition of $\bW^{m,p}_{t,\sigma}(\G)$ for $m \geq 1$ can be understood to be pointwise a.e.

\begin{proposition}[Helmholtz--Weyl decomposition on $\G$]\label{prop:Helmholtzdecomposition}

Let $p \in (1,\infty)$.
Let $m$ be a nonnegative integer and assume $\G$ is of class $C^1$ if $m=0$ or $C^{m,1}$ if $m \geq 1$.
Then $\bW^{m,p}_t(\G)$ can be decomposed as
\begin{equation}\label{H^mp1-decomposition}
\bW^{m,p}_t(\G) = \bW^{m,p}_{t,\sigma}(\G) \oplus \nablaG \WW^{m+1,p}_\#(\G).
\end{equation}
Furthermore, this decomposition is stable in the sense that there exists a positive constant $C > 0$, depending only on $m$, $p$ and $\G$, such that for each $\bv = \bv_1 + \bv_2 \in \bW^{m,p}_t(\G)$ with $\bv_1 \in \bW^{m,p}_{t,\sigma}(\G)$ and $\bv_2 \in \nablaG \WW^{m+1,p}_\#(\G)$ it holds that
\begin{equation*}
\|\bv_1\|_{m,p;\G} + \|\bv_2\|_{m,p;\G} \leq C \|\bv\|_{m,p;\G}.
\end{equation*}
For $p=2$, this decomposition is $\LL^2$-orthogonal.

Finally, if $m = 0$ and $\G$ is only of class $C^{0,1}$, then the stable decomposition \eqref{H^mp1-decomposition} still holds as long as $p \in (2-\varepsilon,2+\varepsilon)$, for some $\varepsilon>0$ depending only on $\G$.

\begin{proof}
We start by proving the case $m = 0$.
Assume $\G$ is of class $C^1$ or of class $C^{0,1}$ if $p \in (2-\varepsilon,2+\varepsilon)$, where $\varepsilon>0$ is given by \autoref{thm:lb}
(Laplace--Beltrami).
Let $\bv \in \bL^p_t(\G)$ and let $\uppi \in \WW^{1,p}_\#(\G)$ be the unique solution of
\begin{equation}\label{auxiliar-problem-explicit-rewritten}
\int_\G \nablaG \uppi \cdot \nablaG q = \int_\G \bv \cdot \nablaG q, \qquad \forall q \in \WW^{1,p^*}_\#(\G),
\end{equation}
guaranteed by \autoref{thm:lb}.
We know that $\|\uppi\|_{1,p;\G} \lesssim \|\bv\|_{0,p;\G}$.

Then, define $\bg := \nablaG \uppi \in \nablaG \WW^{1,p}_\#(\G)$ and $\bh := \bv - \bg \in \bL^p_t(\G)$, the latter satisfying $\divG \bh = \divG \bv - \DeltaG \uppi = 0$ in $(\WW^{1,p^*}(\G))'$.
This gives the desired decomposition $\bv = \bh + \nablaG \uppi$, which is stable because
\begin{equation*}
\|\bh\|_{0,p;\G} + \|\nablaG \uppi\|_{0,p;\G}
\leq \|\bv\|_{0,p;\G} + 2\|\nablaG \uppi\|_{0,p;\G}
\lesssim \|\bv\|_{0,p;\G}.
\end{equation*}
The case $m \geq 1$ follows by integrating the right-hand side of \eqref{auxiliar-problem-explicit-rewritten} by parts and using the higher regularity result within \autoref{thm:lb}.
We omit further details.

For any $m \geq 0$, the uniqueness of the decomposition is proved as follows.
Suppose $\bv \in \bW^{m,p}_{t,\sigma}(\G) \cap \nablaG \WW^{m+1,p}_\#(\G)$.
Then, $\bv = \nablaG \phi$ for some $\phi \in \WW^{m+1,p}_\#(\G)$ and $\divG \bv = \DeltaG \phi = 0 \in (\WW^{1,p^*}(\G))'$;
that is, $\int_\G \nablaG \phi \cdot \nablaG q = 0$ for each $q \in \WW^{1,p^*}(\G)$.
Hence, by \autoref{thm:lb}, $\phi = 0$ and so $\bv = \textbf{0}$.

For $p=2$ the orthogonality of the decomposition follows from
\begin{equation*}
\int_\G \bh \cdot \bg = \int_\G \bh \cdot \nablaG \uppi = - \langle \divG \bh , \uppi \rangle_{(\HH^1(\G))' \times \HH^1(\G)} = 0.
\end{equation*}
This finishes the proof.
\end{proof}
\end{proposition}

\begin{definition}[Leray projection in $\bL^p_t(\G)$]\label{def:Leray-projection}
Under the hypotheses of \autoref{prop:Helmholtzdecomposition}, we define the Leray projection $P_\sigma: \bL^p_t(\G) \rightarrow \bL^p_{t,\sigma}(\G)$ by $P_\sigma \bv = \bh$ for each $\bv = \bh + \bg \in \bL^p_t(\G)$, where $\bh \in \bL^p_{t,\sigma}(\G)$ and $\bg \in \nablaG \WW_\#^{1,p}(\G)$ are the unique components of $\bv$ given by \autoref{prop:Helmholtzdecomposition} (Helmholtz--Weyl decomposition on $\G$).
It shall also prove useful to denote $P_\sigma^\perp \bv = \bg$.

\end{definition}

Notice that
\begin{equation}\label{Leray-Wmp-stability}
\|P_\sigma \bv\|_{m,p;\G} + \|P_\sigma^\perp \bv\|_{m,p;\G} \leq C\|\bv\|_{m,p;\G}, \qquad \bv \in \bW^{m,p}_t(\G).
\end{equation}
Also, for each $\bv \in \bL^p_t(\G)$ and $\bw \in \bL^{p^*}_t(\G)$, we have
\begin{equation*}
\int_\G \bv \cdot \bw = \int_\G P_\sigma \bv \cdot P_\sigma \bw + \int_\G P_\sigma^\perp \bv \cdot P_\sigma^\perp \bw.
\end{equation*}
\subsection{Well-posedness in \texorpdfstring{$\bW^{1,p}_t(\G) \times \LL^p_\#(\G)$}{W¹ᵖ x Lᵖ}}\label{sec:Stokes-wp-generalp}
In this section we verify that problem \eqref{eq:weak-manifold-stokes-nonHilbert} verifies the conditions of \autoref{thm:gen-BB} (generalized Babu\v{s}ka--Brezzi).
The following two results establish the bijectivity of $-\DeltaB := -\bP \divG \nablaG$ regarded as an operator acting from $\bW^{1,p}_{t,\sigma}(\G)$ to $(\bW^{1,p^*}_{t,\sigma}(\G))'$ (see \cite[Theorem~A.2]{BenavidesNochettoShakipov2025-a}).
\begin{proposition}[injectivity of the adjoint operator]\label{prop:non-Hilbert-BL-divergencefree-uniqueness}
Let $p \in (1,\infty)$ and assume $\G$ is of class $C^2$.
If $\bv \in \bW^{1,p^*}_{t,\sigma}(\G)$ satisfies
\begin{equation}\label{eq:injectivity-leaidng-adjoint-in-kernel}
\int_\G \nablaG \bu : \nablaG \bv = 0, \qquad \forall\, \bu \in \bW^{1,p}_{t,\sigma}(\G),
\end{equation}
then $\bv = \mathbf{0}$ a.e.~on $\G$.
\begin{proof}
Let $\bv$ be as above.
First, we claim that there exists unique $q \in \LL^{p^*}_\#(\G)$ such that
\begin{equation}\label{eq:stokes-injectivity-0}
\int_\G q \divG \bu = \int_\G \nablaG \bv : \nablaG \bu, \qquad \forall\, \bu \in \bW^{1,p}_t(\G).
\end{equation}
By virtue of \autoref{prop:Helmholtzdecomposition} (Helmholtz--Weyl decomposition on $\G$), each $\bu \in \bW^{1,p}_t(\G)$ in \eqref{eq:stokes-injectivity-0} can be uniquely decomposed as $\bu = \bw + \nablaG \phi$, where $\bw \in \bW^{1,p}_{t,\sigma}(\G)$ and $\phi \in \WW^{2,p}_\#(\G)$.
Consequently, employing the assumption \eqref{eq:injectivity-leaidng-adjoint-in-kernel} for $\bu$ replaced by $\bw \in \bW^{1,p}_{t,\sigma}(\G)$, problem \eqref{eq:stokes-injectivity-0} is equivalent to finding $q \in \LL^{p^*}_\#(\G)$ such that
\begin{equation} \label{eq:stokes-injectivity-1}
\int_\G q \, \DeltaG \phi = \int_\G \nablaG \bv : \nablaG \nablaG \phi, \qquad \forall \phi \in \WW^{2,p}_\#(\G).
\end{equation}
Problem \eqref{eq:stokes-injectivity-1} is an instance of an ultraweak formulation \eqref{eq:dual-problem} for the Laplace--Beltrami problem, whose well-posedness is ensured by \autoref{lem:ultraweak-allp-LB} (ultra-weak solutions of the Laplace--Beltrami operator on $\G$).
Therefore, there is a unique $q \in \LL^{p^*}_\#(\G)$ that satisfies \eqref{eq:stokes-injectivity-1}, whence also \eqref{eq:stokes-injectivity-0}.
If $p^* \in [2,\infty)$, then \autoref{thm:well-posedness} (well-posedness of tangent Stokes in $\bH^1_t(\G) \times \LL^2_\#(\G)$) applied to the context given by \eqref{eq:stokes-injectivity-0} and $\divG \bv = 0$, yields $\bv = \textbf{0}$.

Otherwise if $p^* \in (1,2)$, we proceed by a bootstrapping argument based on Sobolev embeddings (cf.~\autoref{sec:difgeo-Sobspace}) to increase the Sobolev integrability parameter of the pair $(\bv,q)$.
First, by utilizing \autoref{lem:exploiting-symmetry-covarianthessian} (weak commutator identity), which is valid by density for $\phi \in \WW^{2,p}_\#(\G)$, and the fact that $\divG \bv = 0$, we can rewrite \eqref{eq:stokes-injectivity-1} as
\begin{equation}\label{eq:stokes-injectivity-2}
-\int_\G q \, \DeltaG \phi = \int_\G \bh[\bv] \cdot \nablaG \phi, \qquad \forall \phi \in \WW^{2,p}_\#(\G),
\end{equation}
where $\bh[\bv] := (\tr(\bB)\bB - \bB^2) \bv \in \bL^{p^*}_t(\G)$.
Therefore, we deduce from the $\LL^p$-based Sobolev regularity theory of the Laplace--Beltrami operator for $p \in (1,\infty)$ (see \autoref{lem:ultraweak-allp-LB} and \autoref{thm:lb}) that $q \in \LL^{p^*}_\#(\G)$ is in fact in $\WW^{1,p^*}_\#(\G)$.
From \autoref{prop:GNG-embeddings} (Gagliardo--Nirenberg--Sobolev inequality), we infer that $q \in \LL^{p^*_0}_\#(\G)$, where $p_0^* := \frac{dp^*}{d-p^*} > p^*$. With this, we can go back to \eqref{eq:stokes-injectivity-0} and use \autoref{thm:eq:weak-manifold-BochnerPoisson-nonhilbertian-wellposedness} (well-posedness of the Bochner Laplacian in $\bW^{1,p}_t(\G)$) to deduce that $\bv \in \bW^{1,p^*_0}_t(\G)$.
Using the improved integrability of $\bv$, we consequently obtain the improved integrability for $q \in \WW^{1,p^*_0}_\#(\G)$ from \eqref{eq:stokes-injectivity-2}.
Iterating this process eventually yields $(\bv, q) \in \bH^1_t(\G) \times \LL^2_\#(\G)$, which brings us back to the first case.
This ends the proof.
\end{proof}
\end{proposition}

\begin{proposition}[surjectivity of the adjoint operator] \label{prop:infsup-BochnerLaplacian-nonHilbert}
Let $p \in (1,\infty)$ and assume that $\G$ is of class $C^2$.
Then, there exists a positive constant $\alpha > 0$ such that
\begin{equation}\label{eq:infsup-BochnerLaplacian-nonHilbert}
\sup_{\substack{\bv \in \bW^{1,p^*}_{t,\sigma}(\G)\\ \bv \neq \mathbf{0}}} \frac{\left| \int_\G \nablaG \bu : \nablaG \bv \right|}{\|\bv\|_{1,p^*;\G}} \geq \alpha \| \bu\|_{1,p;\G}, \qquad \forall \bu \in \bW^{1,p}_{t,\sigma}(\G).
\end{equation}

\begin{proof}
\textbf{Step 1:}
We first establish the (in principle) weaker estimate
\begin{equation}\label{eq:WORSE-infsup-BochnerLaplacian-nonHilbert}
\|\bu\|_{0,p;\G} + \sup_{\substack{\bv \in \bW^{1,p^*}_{t,\sigma}(\G)\\ \bv \neq \mathbf{0}}} \frac{\left| \int_\G \nablaG \bu : \nablaG \bv \right|}{\|\bv\|_{1,p^*;\G}} \geq \alpha \|\bu\|_{1,p;\G}, \qquad \forall \bu \in \bW^{1,p}_{t,\sigma}(\G),
\end{equation}
which will be later used in Step 2.

Indeed, by using properties of the Leray projection (cf.~\autoref{def:Leray-projection}) we write in first instance,
\begin{equation*}
\begin{multlined}
\sup_{\substack{\bv \in \bW^{1,p^*}_{t,\sigma}(\G)\\ \bv \neq \mathbf{0}}} \frac{\left| \int_\G \nablaG \bu : \nablaG \bv \right|}{\|\bv\|_{1,p^*;\G}}
= \sup_{\substack{\bv \in \bW^{1,p^*}_t(\G)\\ P_\sigma \bv \neq \mathbf{0}}} \frac{\left| \int_\G \nablaG \bu : \nablaG (P_\sigma \bv) \right|}{\|P_\sigma \bv\|_{1,p^*;\G}}
\stackrel{\eqref{Leray-Wmp-stability} }{\gtrsim} \sup_{\substack{\bv \in \bW^{1,p^*}_t(\G)\\ \bv \neq \mathbf{0}}} \frac{\left| \int_\G \nablaG \bu : \nablaG (P_\sigma \bv) \right|}{\|\bv\|_{1,p^*;\G}}\\
= \sup_{\substack{\bv \in \bW^{1,p^*}_t(\G)\\ \bv \neq \mathbf{0}}} \frac{\left| \int_\G \nablaG \bu : \nablaG \bv - \int_\G \nablaG \bu : \nablaG P_\sigma^\perp \bv \right|}{\|\bv\|_{1,p^*;\G}}.
\end{multlined}
\end{equation*}
Recalling that $P_\sigma^\perp \bv \in \nablaG \WW^{2,p^*}_\#(\G)$ for each $\bv \in \bW^{1,p^*}_t(\G)$, and applying \autoref{lem:exploiting-symmetry-covarianthessian} (weak commutator identity) with $\nablaG \phi \gets P_\sigma^\perp \bv$ and $\bv \gets \bu$, which is valid by density, yields
\begin{equation*}
\int_\G \nablaG \bu : \nablaG P_\sigma^\perp \bv
= \int_\G \divG \bu \, \divG P_\sigma^\perp \bv - \int_\G \left[\left(\tr(\bB)\bB - \bB^2\right)\bu\right] \cdot P_\sigma^\perp \bv
= - \int_\G \left[\left(\tr(\bB)\bB - \bB^2\right)\bu\right] \cdot P_\sigma^\perp \bv,
\end{equation*}
for each $\bv \in \bW^{1,p^*}_t(\G)$.
By utilizing this last identity, the $\bW^{1,p^*}(\G)$-stability of $P_\sigma^\perp$ (cf.~\eqref{Leray-Wmp-stability}) and the corresponding inf-sup condition for the Bochner Laplacian given by \autoref{thm:eq:weak-manifold-BochnerPoisson-nonhilbertian-wellposedness} (well-posedness of the Bochner Laplacian in $\bW^{1,p}_t(\G)$) and \cite[Theorem~A.1]{BenavidesNochettoShakipov2025-a} (characterization of injectivity and surjectivity), we arrive at
\begin{equation*}
\sup_{\substack{\bv \in \bW^{1,p^*}_{t,\sigma}(\G)\\ \bv \neq \mathbf{0}}} \frac{\left| \int_\G \nablaG \bu : \nablaG \bv \right|}{\|\bv\|_{1,p^*;\G}}
\gtrsim \sup_{\substack{\bv \in \bW^{1,p^*}_t(\G)\\ \bv \neq \mathbf{0}}} \frac{\left| \int_\G \nablaG \bu : \nablaG \bv \right|}{\|\bv\|_{1,p^*;\G}} - \| \bu\|_{0,p;\G}
\gtrsim \|\bu\|_{1,p;\G} - \| \bu\|_{0,p;\G}.
\end{equation*}
This establishes \eqref{eq:WORSE-infsup-BochnerLaplacian-nonHilbert} and finishes the proof of Step 1.

\textbf{Step 2:}
With \eqref{eq:WORSE-infsup-BochnerLaplacian-nonHilbert} at hand, we now proceed to prove \eqref{eq:infsup-BochnerLaplacian-nonHilbert} by contradiction in the spirit of Peetre--Tartar's Lemma \cite[Lemma~A.38]{ErnGuermond2004}.

In fact, assume there exists a sequence $\{\bu_n\}_{n \in \NN} \subseteq \bW^{1,p}_{t,\sigma}(\G)$ such that $\|\bu_n\|_{1,p;\G} = 1$ for each $n \in \NN$ and
\begin{equation*}
\sup_{0 \neq \bv \in \bW^{1,p^*}_{t,\sigma}(\G)} \frac{\left| \int_\G \nablaG \bu_n : \nablaG \bv \right|}{\|\bv\|_{1,p^*;\G}} \xrightarrow[]{n \to \infty} 0.
\end{equation*}
In particular,
\begin{equation}\label{aux-consequence-sup-go-to-zero}
\lim_{n\to\infty} \int_\G \nablaG \bu_n : \nablaG \bv = 0, \qquad \forall \bv \in \bW^{1,p^*}_{t,\sigma}(\G).
\end{equation}
Since $\{\bu_n\}_{n \in \NN}$ is bounded in $\bW^{1,p}_{t,\sigma}(\G)$ (which is reflexive), the Banach--Alaoglu Theorem \cite[Theorem 3.18]{Brezis2011} and \autoref{rem:W1p-compactly-embedded-in-L^p} (compact Sobolev embedding) implies the existence of $\bu \in \bW^{1,p}_{t,\sigma}(\G)$ such that (up to a subsequence)
\begin{equation*}
\bu_n \rightharpoonup \bu \quad \text{in $\bW^{1,p}_{t,\sigma}(\G)$}, \qquad 
\bu_n \to \bu \quad \text{in $\bL^p_{t,\sigma}(\G)$}.
\end{equation*}
Thus, for each $\bv \in \bW^{1,p^*}_{t,\sigma}(\G)$,
\begin{equation*}
\int_\G \nablaG \bu : \nablaG \bv
= \lim_{n\to\infty} \int_\G \nablaG \bu_n : \nablaG \bv
\stackrel{\eqref{aux-consequence-sup-go-to-zero}}{=} 0.
\end{equation*}
By \autoref{prop:non-Hilbert-BL-divergencefree-uniqueness} (injectivity of the adjoint operator), we deduce that $\bu = 0$ a.e.~on $\G$.
Hence, passing to the limit $n \to \infty$ in inequality \eqref{eq:WORSE-infsup-BochnerLaplacian-nonHilbert} for $\bu=\bu_n$ yields $\lim_{n \to \infty} \|\bu_n\|_{1,p;\G} = 0$, which contradicts the fact that $\|\bu_n\|_{1,p;\G} = 1$ for each $n \in \NN$.
This concludes the proof.
\end{proof}
\end{proposition}

With \autoref{prop:non-Hilbert-BL-divergencefree-uniqueness} and \autoref{prop:infsup-BochnerLaplacian-nonHilbert} at hand, we now establish the well-posedness of problem \eqref{eq:weak-manifold-stokes-nonHilbert}.

\begin{thm}[well-posedness of the tangent Stokes equations in $\bW^{1,p}_t(\G) \times \LL^p_\#(\G)$] \label{thm:stokes-lp-well-posedness}
Let $p \in (1,\infty)$ and assume $\G$ is of class $C^2$.
Then, for each $\bf \in (\bW^{1,p^*}_t(\G))'$ and $g \in \LL^p_\#(\G)$ there exists a unique $(\bu,\uppi) \in \bW^{1,p}_t(\G) \times \LL^p_\#(\G)$ solution of \eqref{eq:weak-manifold-stokes-nonHilbert}.
Moreover, there exists a constant $C>0$, depending only on $\G$ and $p$, such that
\begin{equation*}
\|\bu\|_{1,p;\G} + \|\uppi\|_{0,p;\G} \leq C \left( \|\bf\|_{(\bW^{1,p^*}_t(\G))'} + \|g\|_{0,p;\G} \right).
\end{equation*}
\begin{proof}
We first notice that all bilinear forms defining problem \eqref{eq:weak-manifold-stokes-nonHilbert} are bounded in their respective domains, and that for each $r \in (1,\infty)$, the set $V := \left\{\bw \in \bW^{1,r}_t(\G): \int_\G q \divG \bw = 0, \quad \forall q \in \LL^{r^*}_\#(\G)\right\}$ reduces to $V = \bW^{1,r}_{t,\sigma}(\G)$.
Then, thanks to \autoref{prop:non-Hilbert-BL-divergencefree-uniqueness}, \autoref{prop:infsup-BochnerLaplacian-nonHilbert} and \autoref{prop:inf-sup-divergence-nonHilbert}, the proof follows from a straightforward application of \autoref{thm:gen-BB} (generalized Babu\v{s}ka--Brezzi theory in reflexive Banach spaces).
\end{proof}
\end{thm}

\subsection{Higher \texorpdfstring{$\LL^p$}{Lᵖ}-based Sobolev regularity}\label{sec:higher-reg-Stokes}

Having established the well-posedness of \eqref{eq:weak-manifold-stokes-nonHilbert}, we are now ready to derive higher-regularity results directly for any $p \in (1,\infty)$.
The proof strategy is a generalization of the trick used in \cite{ORZ2021}.
\begin{thm}[higher $\LL^p$-based regularity of the tangent Stokes equations]\label{thm:Lp-based-regularity-tangentStokes}
Let $p \in (1,\infty)$ and assume $\G$ is of class $C^{m+2,1}$ for some nonnegative integer $m$.
If $\bf \in \bW^{m,p}_t(\G)$ and $g \in \WW^{m+1,p}_\#(\G)$, then the solution $(\bu,\uppi) \in \bW^{1,p}_t(\G) \times \LL^p_\#(\G)$ of \eqref{eq:weak-manifold-stokes-nonHilbert} (provided by \autoref{thm:stokes-lp-well-posedness}) belongs to $\bW^{m+2,p}_t(\G) \times \WW^{m+1,p}_\#(\G)$.
Moreover, there exists a positive constant $C$, depending only on $\G$, $m$ and $p$, such that
\begin{equation}\label{eq:allp-Lpbased-Stokes_higher-apriori}
\|\bu\|_{m+2,p;\G} + \|\uppi\|_{m+1,p;\G} \leq C \left( \|\bf\|_{m,p;\G} + \|g\|_{m+1,p;\G}\right).
\end{equation}
\end{thm}

\begin{proof}
We first prove the case $m=0$.
For general $m$ the proof will follow by induction.

\noindent \textbf{Step 1: Case $m=0$.}
Recalling \autoref{lem:exploiting-symmetry-covarianthessian} (weak commutator identity), and using a density argument, we start by choosing $\bv = \nablaG q$ with arbitrary $q \in \WW^{2,p^*}_\#(\G)$ in the first equation of \eqref{eq:weak-manifold-stokes-nonHilbert}, thereby arriving at
\begin{multline*}
- \int_\G \uppi \, \DeltaG q
= \int_\G \bf \cdot \nablaG q - \int_\G \nablaG \bu : \nablaG \nablaG q
\stackrel{\eqref{exploiting-symmetry-covarianthessian}}{=} \int_\G \bf \cdot \nabla_\G q - \int_\G \divG \bu \cdot \DeltaG q + \int_\G \left(\tr(\bB)\bB - \bB^2\right)\bu \cdot \nablaG q.
\end{multline*}
Furthermore, upon using the fact that $\div_\G \bu = g$ a.e.~on $\G$, and integrating-by-parts (cf.~\eqref{eq:int-by-parts}) the second term in the right-hand side, we further obtain
\begin{equation*}
- \int_\G \uppi \, \DeltaG q = \int_\G \bh[\bu] \cdot \nabla_\G q, \qquad \forall q \in\WW^{2,p^*}_\#(\G),
\end{equation*}
where $\bh[\bu] := \bf + \nablaG g + (\tr(\bB)\bB - \bB^2)\bu \in \bL^p_t(\G)$.
That is, $\uppi \in \LL^p_\#(\G)$ satisfies the ultra-weak formulation of the Laplace--Beltrami \eqref{eq:dual-problem} with right-hand side given by $-\divG \bh[\bu]$, understood as an element of $(\WW^{2,p^*}_\#(\G))'$ by integration by parts.
Since $-\divG \bh[\bu]$ can also be seen as an element of $(\WW^{1,p^*}_\#(\G))'$, we deduce from
\autoref{thm:lb}
(Laplace--Beltrami) that $\uppi \in \WW^{1,p}_\#(\G)$ and
\begin{equation}\label{aux:estimate-pi}
\|\uppi\|_{1,p;\G}
\lesssim \|\divG \bh[\bu]\|_{(\WW^{1,p^*}_\#(\G))'}
\lesssim \|\bh[\bu]\|_{0,p;\G}
\lesssim \|\bf\|_{0,p;\G} + \|g\|_{1,p;\G} + \|\bu\|_{0,p;\G}.
\end{equation}
With this further regularity on $\uppi$, we can now integrate by parts in the first equation of \eqref{eq:weak-manifold-stokes-nonHilbert}, yielding
\begin{equation*}
\int_\G \nablaG \bu: \nablaG \bv = \int_\G \bj[\uppi] \cdot \bv, \quad \forall\, \bv \in \bW^{1,p^*}_t(\G),
\end{equation*}
where $\bj[\uppi] := \bf - \nablaG \uppi \in \bL^p_t(\G)$.
Therefore, by \autoref{thm:allp-based-regularity-BochnerLaplacian} ($\bW^{m+2,p}$-regularity of the Bochner Laplacian on $\G$), we infer that $\bu \in \bW^{2,p}_t(\G)$ and
\begin{equation}\label{aux:estimate-u}
\|\bu\|_{2,p;\G}
\lesssim \|\bj[\uppi]\|_{0,p;\G}
\lesssim \|\bf\|_{0,p;\G} + \|\uppi\|_{1,p;\G}
\lesssim \|\bf\|_{0,p;\G} + \|g\|_{1,p;\G} + \|\bu\|_{0,p;\G}.
\end{equation}
In this way, combining \eqref{aux:estimate-pi} and \eqref{aux:estimate-u}, and using \autoref{thm:stokes-lp-well-posedness} (well-posedness of tangent Stokes in $\bW^{1,p}_t(\G) \times \LL^p_\#(\G)$) to bound $\|\bu\|_{0,p;\G}$ in terms of $\bf$ and $g$, we finally arrive at
\begin{equation*}
\|\bu\|_{2,p;\G} + \|\uppi\|_{1,p;\G}
\lesssim \|\bf\|_{0,p;\G} + \|g\|_{1,p;\G}.
\end{equation*}
This concludes the proof for $m = 0$.
\\ \\
\textbf{Step 2: Case $m \geq 1$.}
Assume the assertion holds for a fixed $m \geq 0$.
Assume $\G$ is of class $C^{m+3,1}$ and $\bf \in \bW^{m+1,p}_t(\G)$ and $g \in \WW^{m+2,p}_\#(\G)$.
By the induction hypothesis we deduce that $(\bu,\uppi) \in \bW^{m+2,p}_t(\G) \times \WW^{m+1,p}_\#(\G)$ and
\begin{equation}\label{aux:ind_hyp}
\|\bu\|_{m+2,p;\G} + \|\uppi\|_{m+1,p;\G} \leq C_{m,p} \left( \|\bf\|_{m,p;\G} + \|g\|_{m+1,p;\G}\right).
\end{equation}
Now, proceeding as in the proof of the case $m=0$, we observe that $\bh[\bu] := \bf + \nablaG g + (\tr(\bB)\bB - \bB^2)\bu \in \bW^{m+1,p}_t(\G)$.
Hence, by the higher regularity result within \autoref{thm:lb} (Laplace--Beltrami) and using \eqref{aux:ind_hyp} we deduce that $\uppi \in \WW^{m+2,p}_\#(\G)$ and
\begin{equation*}
\|\uppi\|_{m+2,p;\G}
\lesssim \|\divG \bh[\bu]\|_{m,p;\G}
\lesssim \|\bf\|_{m+1,p;\G} + \|g\|_{m+2,p;\G} + \|\bu\|_{m+1,p;\G}
\lesssim \|\bf\|_{m+1,p;\G} + \|g\|_{m+2,p;\G}.
\end{equation*}
In turn, $\bj[\uppi] := \bf - \nablaG \uppi \in \bW^{m+1,p}_t(\G)$, whence by \autoref{thm:allp-based-regularity-BochnerLaplacian} ($\bW^{m+2,p}$-regularity of the Bochner Laplacian on $\G$) we obtain that $\bu \in \bW^{m+3,p}_t(\G)$ and
\begin{equation*}
\|\bu\|_{m+3,p;\G}
\lesssim \|\bj[\uppi]\|_{m+1,p;\G}
\lesssim \|\bf\|_{m+1,p;\G} + \|\uppi\|_{m+2,p;\G}
\lesssim \|\bf\|_{m+1,p;\G} + \|g\|_{m+2,p;\G}.
\end{equation*}
This concludes the induction step, and hence, the proof of \autoref{thm:Lp-based-regularity-tangentStokes}.
\end{proof}

\subsection{The Stokes operator}\label{sec:Stokes-operator}

In this section, we study spectral properties of the tangent Stokes operator on $\G$ induced by \eqref{eq:manifold-stokes}, thereby mimicking spectral properties of the classical Stokes operator in flat domains \cite[\textsection 4]{ConstantinFoias1988} \cite[Part I, \textsection 2]{RobinsonRodrigoSadowski2016}.
We will later make use of this theory in \autoref{sec:existence-incompressibleNS-d234} to establish the existence of solutions for the incompressible tangent Navier--Stokes equations via the Faedo--Galerkin method. We start by introducing the Stokes operator under the ``minimal regularity'' regime.
\begin{definition}[Stokes operator on $\G$]\label{def:Stokes-operator}
Let $p \in (1,\infty)$ and assume $\G$ is of class $C^2$.
We define the Stokes operator on $\G$ as the linear operator $\wt A: \bW^{1,p}_{t,\sigma}(\G) \rightarrow (\bW^{1,p^*}_{t,\sigma}(\G))'$ such that
\begin{equation}\label{eq:Stokes-min-reg}
\langle \wt A\bu, \bv \rangle_{(\bW^{1,p^*}_{t,\sigma}(\G))' \times \bW^{1,p^*}_{t,\sigma}(\G)} := \int_\G \nablaG \bu : \nablaG \bv, \qquad \forall \bu \in \bW^{1,p}_{t,\sigma}(\G), \quad \bv \in \bW^{1,p^*}_{t,\sigma}(\G).
\end{equation}
\end{definition}
In agreement with our philosophy of requiring minimal regularity of $\G$, we favor \autoref{def:Stokes-operator} in lieu of the classical approach of defining a ``strong'' Stokes operator $A = - P_\sigma (\bP \divG \nablaG \cdot)$ because the latter would require $\G$ to be of class $C^{2,1}$. We reveal the connection between $\widetilde A$ and $A$ in \eqref{eq:weak-strong-stokes-operator} below.

From \autoref{prop:non-Hilbert-BL-divergencefree-uniqueness} (injectivity of the adjoint operator), \autoref{prop:infsup-BochnerLaplacian-nonHilbert} (surjectivity of the adjoint operator) and \cite[Theorem~A.2]{BenavidesNochettoShakipov2025-a} (Banach--Ne\v{c}as--Babu\v{s}ka) we know that $\wt A: \bW^{1,p}_{t,\sigma}(\G) \rightarrow (\bW^{1,p^*}_{t,\sigma}(\G))'$ is a continuous bijection.
For $p=2$, we notice that endowing the Hilbert space $H := \bH^1_{t,\sigma}(\G)$ and its dual $H'$ with the equivalent inner products (cf.~\autoref{lem:Poincare} (Poincaré inequality))
\begin{equation}\label{equiv:H1tsig-inner-prod}
(\bu,\bv)_H := \int_\G \nablaG \bu : \nablaG \bv, \qquad \forall \bu,\bv \in H.
\end{equation}
and
\begin{equation}\label{equiv:dualH1tsig-inner-prod}
(\bf,\bg)_{H'}
:= (\wt A^{-1}\bf,\wt A^{-1}\bg)_H
= \langle \bf, \wt A^{-1}\bg \rangle_{H' \times H}
= \langle \bg, \wt A^{-1}\bf \rangle_{H' \times H}
,
\qquad \forall\, \bf,\bg \in H',
\end{equation}
trivially turns $\wt A: H \rightarrow H'$ into an isometry.
Let $\iota: H \rightarrow H'$ be the compact canonical embedding of $H$ into $H'$ given by (cf.~\autoref{sec:difgeo-Sobspace})
\begin{equation*}
\langle \iota(\bu), \bv \rangle_{H' \times H} := \int_\G \bu \cdot \bv, \qquad \forall \bu,\bv \in H.
\end{equation*}
We next introduce the linear operator $T : H' \rightarrow H'$ defined by $T := \iota \circ \wt A^{-1}$.
This operator is compact, self-adjoint and positive because
\begin{equation*}
(T\bf,\bg)_{H'}
= \langle \iota (\wt A^{-1} \bf) , \wt A^{-1}\bg \rangle_{H' \times H}
= \int_\G \wt A^{-1} \bf \cdot \wt A^{-1}\bg
= (\bf,T\bg)_{H'},
\end{equation*}
and $(T\bf,\bf)_{H'} = \|\wt A^{-1} \bf\|_{0,\G}^2 \geq 0$ for each $\bf, \bg \in H'$.
From the latter, we also deduce that $T\bf = \mathbf{0}$ if and only if $\bf = \mathbf{0}$.

We are now ready to establish the spectral properties of $\wt A: H \rightarrow H'$.

\begin{thm}[eigendecomposition of the Stokes operator]\label{thm:Stokes-eigen-minreg}
Assume $\G$ is of class $C^2$.
Then, there exists a family of functions $\{\bv_n\}_{n \in \NN} \subseteq \bigcap_{2 \leq p<\infty} \bW^{1,p}_{t,\sigma}(\G)$ and a nondecreasing sequence $\{\omega_n\}_{n \in \NN}$ of positive numbers with $\lim_{n \to \infty} \omega_n = \infty$ such that
\begin{enumerate}
\item\label{it:eigen-equation-lowreg} For each $n \in \NN$, $(\bv_n,\omega_n)$ is an eigenpair of the Stokes operator, i.e.: $\wt A \bv_n = \omega_n \iota(\bv_n)$ in $(\bH^1_{t,\sigma}(\G))'$, or equivalently,
\begin{equation}\label{eq:weak-Stokes-eigfun}
\int_\G \nablaG \bv_n : \nablaG \bw = \omega_n \int_\G \bv_n \cdot \bw, \qquad \forall \bw \in \bH^1_{t,\sigma}(\G).
\end{equation}
\item\label{it:dualH1-basis} $\{\omega_n^{1/2}\iota(\bv_n)\}_{n \in \NN}$ is an orthonormal basis of $H'= (\bH^1_{t,\sigma}(\G))'$ with respect to the equivalent inner product \eqref{equiv:dualH1tsig-inner-prod}.
\item\label{it:L^2-basis} $\{\bv_n\}_{n \in \NN}$ is an orthonormal basis of $\bL^2_{t,\sigma}(\G)$.
\item\label{it:H^1-basis} $\{\omega_n^{-1/2}\bv_n\}_{n \in \NN}$ is an orthonormal basis of $H = \bH^1_{t,\sigma}(\G)$ with respect to the equivalent inner product \eqref{equiv:H1tsig-inner-prod}.
\end{enumerate}
\begin{proof}
We already know that the linear operator $T$ is injective, compact, self-adjoint and positive with respect to the inner product \eqref{equiv:dualH1tsig-inner-prod}.
Hence, since $(\bH^1_{t,\sigma}(\G))'$ is separable, it follows from \cite[\textsection D.6]{Evans2010} that there exist an orthogonal basis $\{\bz_n\}_{n \in \NN}$ of $(\bH^1_{t,\sigma}(\G))'$ and a sequence of (positive) nonincreasing numbers $\{\lambda_n \}_{n \in \NN}$ with $\lim_{n \to \infty} \lambda_n = 0$ such that $T \bz_n = \lambda_n \bz_n$ in $(\bH^1_{t,\sigma}(\G))'$.
Equivalently, upon defining $\bv_n := \wt A^{-1} \bz_n \in \bH^1_{t,\sigma}(\G)$ and $\omega_n := \lambda_n^{-1}$, we see that $\wt A \bv_n = \omega_n \iota(\bv_n)$ in $(\bH^1_{t,\sigma}(\G))'$ as well as
\begin{equation}\label{weak-eigen-problem}
\int_\G \nablaG \bv_n : \nablaG \bw 
\stackrel{\eqref{eq:Stokes-min-reg}}{=} \langle \wt A\bv_n, \bw \rangle_{H' \times H}
= \omega_n \langle \iota(\bv_n), \bw \rangle_{H' \times H}
= \omega_n \int_\G \bv_n \cdot \bw, \qquad \forall \bw \in H,
\end{equation}
which proves \eqref{eq:weak-Stokes-eigfun}.

Without loss of generality assume $(\bz_n,\bz_k)_{(\bH^1_{t,\sigma}(\G))'} = \delta_{nk} \omega_k$ for each $n,k \in \NN$.
To show that $\{\omega_n^{1/2} \bv_n\}_{n \in \NN} \equiv \{\omega_n^{1/2} \iota \bv_n\}_{n \in \NN}$ is orthonormal in $H'$, it is enough to notice that
\begin{equation*}
(\iota \bv_n, \iota \bv_k)_{H'}
= (\iota \wt A^{-1} \bz_n, \iota \wt A^{-1}\bz_k)_{H'}
= (T \bz_n, T\bz_k)_{H'}
= \omega_n^{-1} \omega_k^{-1} (\bz_n, \bz_k)_{H'}
= \omega_n^{-1} \delta_{nk},
\end{equation*}
for each $n,k \in \NN$.
In turn, for the $H$-orthonormality of $\{\omega_n^{-1/2}\bv_n\}_{n \in \NN}$ and the $\bL^2_{t,\sigma}(\G)$-orthonormality of $\{\bv_n\}_{n \in \NN}$ it suffices to see that
\begin{equation}\label{eigen-H1-relation}
\omega_k \int_\G \bv_n \cdot \bv_k
\stackrel{\eqref{weak-eigen-problem}}{=} 
(\bv_n,\bv_k)_{H}
\stackrel{\eqref{equiv:dualH1tsig-inner-prod}}{=} (\wt A \bv_n,\wt A \bv_k)_{H'}
= (\bz_n,\bz_k)_{H'}
= \omega_k \delta_{nk}.
\end{equation}
To show that $\{\iota \bv_n\}_{n \in \NN}$ is a basis of $H'$, let $\bf \in H'$ such that $(\bf,\iota\bv_n)_{H'} = 0$ for each $n \in \NN$.
That is, $(\bf,T\bz_n)_{H'} = \lambda_n (\bf,\bz_n)_{H'} = 0$ for each $n \in \NN$.
Since $\lambda_n \neq 0$ and $\{\bz_n\}_{n \in \NN}$ is an orthogonal basis of $H'$, we deduce that $\bf = \mathbf{0}$.
Hence, $\vspan\{\omega_n^{1/2} \iota \bv_n\}_{n \in \NN}$ is dense in $H'$.
Items \autoref{it:L^2-basis} and \autoref{it:H^1-basis} are shown in a similar fashion, and hence we omit their proofs.

Finally, for each $p \in (2,\infty)$, the $\bW^{1,p}_{t,\sigma}(\G)$-regularity of each $\bv_n$ follows from a bootstrapping argument resulting from combining the invertibility of the Stokes operator \eqref{eq:Stokes-min-reg} applied to problem \eqref{weak-eigen-problem} with Sobolev embeddings (cf.~\autoref{sec:difgeo-Sobspace}).
\end{proof}
\end{thm}
As expected, if $\G$ is more regular, say of class $C^{2,1}$, then the family $\{\bv_n\}_{n \in \NN}$ provided by \autoref{thm:Stokes-eigen-minreg} are actually a.e.~eigenfunctions of the ``strong'' Stokes operator $A := - P_\sigma (\bP \divG \nablaG \cdot): \bW^{2,p}_{t,\sigma}(\G) \rightarrow \bL^p_{t,\sigma}(\G)$;
see \autoref{thm:Stokes-eigen-highreg} below.
Before stating the result notice that integration-by-parts formula \eqref{eq:int-by-parts} and the definition of the Leray projection (cf.~\autoref{def:Leray-projection}) let us write
\begin{equation} \label{eq:weak-strong-stokes-operator}
\langle \wt A\bu, \bv \rangle_{(\bW^{1,p^*}_{t,\sigma}(\G))' \times \bW^{1,p^*}_{t,\sigma}(\G)}
= \int_\G \nablaG \bu : \nablaG \bv
= \int_\G A\bu \cdot \bv, \qquad \forall \bu \in \bW^{2,p}_{t,\sigma}(\G), \bv \in \bW^{1,p^*}_{t,\sigma}(\G),
\end{equation}
Consequently, if $\G$ is of class $C^{m+2,1}$ for some $m \in \NN \cup \{0\}$, for each $\bu \in \bW^{m+2,p}_{t,\sigma}(\G)$ and $\bf \in \bW^{m,p}_t(\G)$, it follows that: $\wt A\bu = P_\sigma \bf$ in $(\bW^{1,p^*}_{t,\sigma}(\G))'$ if and only if $A\bu = P_\sigma \bf$ a.e.~on $\G$, or equivalently $P_\sigma(\bf + \bP \divG \nablaG \bu) = \mathbf{0}$.
Invoking \autoref{def:Leray-projection} (Leray projection in $\bL^p_t(\G)$) along with \eqref{Leray-Wmp-stability}, there exists a unique $\uppi \in \WW^{m+1,p}_\#(\G)$ such that
\begin{align*}
- \bP \divG \nablaG \bu + \nablaG \uppi & = \bf, \quad \text{on $\G$},\\
\divG \bu & = 0, \quad \text{on $\G$}.
\end{align*}
Moreover, if $\bf \in \bW^{m,p}_{t,\sigma}(\G)$, then by \autoref{thm:Lp-based-regularity-tangentStokes} (higher $\LL^p$-based regularity of the tangent Stokes equations) we deduce that $A$ is a continuous bijection from $\bW^{m+2,p}_{t,\sigma}(\G)$ to $\bW^{m,p}_{t,\sigma}(\G)$ and
\begin{equation}\label{invStokesOp-contdep}
\|\bu\|_{m+2,p;\G} + \|\uppi\|_{m+1,p;\G} \lesssim \|\bf\|_{m,p;\G} = \|A \bu\|_{m,p;\G}.
\end{equation}
Conversely, the Bounded Inverse Theorem \cite[Corollary 2.7]{Brezis2011} (see also \cite[Teorema 3.4]{Gatica2024}) implies that
\begin{equation}\label{StokesOp-contdep}
\|A \bu\|_{m,p;\G} \lesssim \|\bu\|_{m+2,p;\G}.
\end{equation}
Hence, from \eqref{invStokesOp-contdep} and \eqref{StokesOp-contdep} we have that in particular for $p=2$:
if $m+2 = 2\ell$ is even, then
\begin{equation}\label{equivalent-mp2-even}
(\bu, \bv) \mapsto \int_\G (\underbrace{A \dotsc A}_{\text{$\ell$ times}} \bu) \cdot (\underbrace{A \dotsc A}_{\text{$\ell$ times}} \bv)
\end{equation}
defines an equivalent inner product in $\bH^{m+2}_{t,\sigma}(\G)$.
On the other hand, if $m+2 = 2\ell+1$ is odd we also use \autoref{lem:Poincare} (Poincaré inequality) to deduce that
\begin{equation}\label{equivalent-mp2-odd}
(\bu, \bv) \mapsto \int_\G \nablaG (\underbrace{A \dotsc A}_{\text{$\ell$ times}} \bu) : \nablaG (\underbrace{A \dotsc A}_{\text{$\ell$ times}} \bv)
\end{equation}
defines an equivalent inner product in $\bH^{m+2}_{t,\sigma}(\G)$;
we refer to \cite[Section 7.3]{Brezis2011} for similar constructions.
The following result shows that, provided $\G$ is sufficiently regular, the eigenpairs $\{(\bv_n,\omega_n)\}_{n \in \NN}$ provided by \autoref{thm:Stokes-eigen-minreg} enjoy higher regularity and orthogonality with respect to the inner products \eqref{equivalent-mp2-even} and \eqref{equivalent-mp2-odd}.

\begin{thm}[higher regularity of the eigenfunctions of the Stokes operator]\label{thm:Stokes-eigen-highreg}
Assume $\G$ is of class $C^{m+2,1}$ for some nonnegative integer $m$.
Then, the family of eigenpairs $\{(\bv_n,\omega_n)\}_{n \in \NN}$ provided by \autoref{thm:Stokes-eigen-minreg} satisfy
\begin{enumerate}
\item\label{it:higherSobolevreg} $\{\bv_n\}_{n \in \NN} \subseteq \bigcap_{2 \leq p<\infty} \bW^{m+2,p}_{t,\sigma}(\G)$.
\item\label{it:eigen-equation} For each $n \in \NN$, $(\bv_n,\omega_n)$ is an eigenpair of $A$:
\begin{equation*}
A \bv_n = \omega_n \bv_n, \qquad \text{a.e.~on $\G$}.
\end{equation*}
\item\label{it:higherSobolev-basis}
$\{\omega_n^{-(m+2)/2} \bv_n\}_{n \in \NN}$ is an orthonormal basis of $\bH^{m+2}_{t,\sigma}(\G)$ with respect to the equivalent inner product \eqref{equivalent-mp2-even} or \eqref{equivalent-mp2-odd}, depending on the parity of $m+2$.
\end{enumerate}
\begin{proof}
Fix $p \in (2,\infty)$ and notice from \autoref{thm:Stokes-eigen-minreg} that each $\bv_n \in \bW^{1,p}_{t,\sigma}(\G)$ satisfies \eqref{eq:weak-Stokes-eigfun}, namely
\begin{equation*}
\int_\G \nablaG \bv_n : \nablaG \bw = \omega_n \int_\G \bv_n \cdot \bw, \qquad \forall \bw \in \bW^{1,p^*}_{t,\sigma}(\G).
\end{equation*}
Hence, proceeding similarly to the proof of \autoref{prop:non-Hilbert-BL-divergencefree-uniqueness} (injectivity of the adjoint operator) we can utilize \autoref{prop:Helmholtzdecomposition} (Helmholtz--Weyl decomposition on $\G$) to show that there exists a unique $\uppi_n \in \LL^p_\#(\G)$ such that
\begin{equation*}
\int_\G \nablaG \bv_n : \nablaG \bw - \int_\G \uppi_n \divG \bw = \omega_n \int_\G \bv_n \cdot \bw, \qquad \forall \bw \in \bW^{1,p^*}_t(\G).
\end{equation*}
Consequently, by iterating \autoref{thm:Lp-based-regularity-tangentStokes} (higher $\LL^p$-based regularity of the tangent Stokes equations) as many times as the regularity of $\G$ permits, we deduce that $\bv_n \in \bW^{m+2,p}_{t,\sigma}(\G)$.
Moreover, integration-by-parts formula \eqref{eq:int-by-parts} and properties of the Leray projection (cf.~\autoref{def:Leray-projection}) yield $A\bv_n := - P_\sigma (\bP \divG \nablaG \bv_n) = \omega_n \bv_n$ a.e.~on $\G$.
This proves \autoref{it:higherSobolevreg} and \autoref{it:eigen-equation}.

The proof of \autoref{it:higherSobolev-basis} follows similarly to the proof of \autoref{it:H^1-basis} in \autoref{thm:Stokes-eigen-minreg}, by also using the bijectivity of $A$ and the first equality of \eqref{eigen-H1-relation} if $m+2$ is odd.
We omit further details.
\end{proof}
\end{thm}

\section{Tangent Navier--Stokes equations} \label{sec:ns}
In this section, we utilize the well-posedness, regularity and spectral theory for the \eqref{eq:manifold-stokes} developed in \autoref{sec:stokes}, in conjunction with the Fredholm Alternative, to derive a corresponding well-posedness and regularity theory for the tangent Navier--Stokes equations \eqref{eq:manifold-Navier--Stokes}.
More precisely, given $p \in [p_0,\infty)$ ($p_0$ to be determined later), we are interested in the following variational formulation associated to \eqref{eq:manifold-Navier--Stokes}:
Given $\bf \in (\bW^{1,p^*}_t(\G))'$ and $g \in \LL^p_\#(\G)$, find $(\bu,\uppi) \in \bW^{1,p}_t(\G) \times \LL_\#^p(\G)$ such that
\begin{equation}\label{eq:weak-manifold-Navier--Stokes}
\begin{aligned}
\int_\G \nablaG \bu : \nablaG \bv
+ \int_\G (\nablaG \bu)\bu \cdot \bv - \int_\G \uppi \divG \bv & = \langle \bf , \bv \rangle
, \quad & \forall\, \bv \in \bW^{1,p^*}_t(\G),\\
- \int_\G q \divG \bu & = -\int_\G g \, q, \quad & \forall\, q \in \LL^{p^*}_\#(\G).
\end{aligned}
\end{equation}
In contrast with \eqref{eq:manifold-stokes}, the presence of the nonlinear term $(\nablaG \bu)\bu$ in \eqref{eq:manifold-Navier--Stokes} makes it necessary to restrict the range of $p$ to ensure that the integral $\int_\G (\nablaG \bu)\bu \cdot \bv$ makes sense, as it also happens in the flat case; see, e.g., \cite{BjorlandEtAl:2011}.
We now determine such range by first invoking H\"older's inequality, yielding $\int_\G |(\nablaG \bu)\overline\bu \cdot \bv| \leq \|\nablaG \bu\|_{0,p;\G} \|\overline\bu \, \bv^T\|_{0,p^*;\G}$. 
It remains to make sense of the latter for each $\overline{\bu} \in \bW^{1,p}_t(\G)$ and $\bv \in \bW^{1,p^*}_t(\G)$.
We argue first for $d=2$.
H\"older's inequality and Sobolev embeddings (cf.~\autoref{sec:difgeo-Sobspace}) imply
\begin{align*}
\|\overline{\bu} \, \bv^T\|_{0,p^*;\G}
& \leq
\begin{cases}
\|\overline{\bu}\|_{0,4;\G} \|\bv\|_{0,4;\G}, & \text{if $p=2=p^*$},\\
\|\overline{\bu}\|_{0,\G} \|\bv\|_{0,\frac{2p^*}{2-p^*};\G}, & \text{if $p>2>p^*$},\\
\|\overline{\bu}\|_{0,\frac{2p}{2-p};\G} \|\bv\|_{0,\frac{2p}{3p-4};\G}, & \text{if $\frac{4}{3}<p<2<p^*$},\\
\|\overline{\bu}\|_{0,4;\G} \|\bv\|_{0,\infty;\G}, & \text{if $p=\frac{4}{3}$ (and so $p^*=4$)}
\end{cases}\\
& \lesssim \|\overline\bu\|_{1,p;\G} \|\bv\|_{1,p^*;\G},
\qquad \forall \overline{\bu} \in \bW^{1,p}_t(\G), \bv \in \bW^{1,p^*}_t(\G),
\end{align*}
as long as $p\geq\frac{4}{3}$.
Conversely, the limiting bound $\|\overline{\bu} \, \bv^T\|_{0,4;\G} \leq \|\overline{\bu}\|_{0,4;\G} \|\bv\|_{0,\infty;\G}$ for $p=\frac{4}{3}$ above makes it apparent that $\|\overline{\bu} \, \bv^T\|_{0,p^*;\G}$ cannot be controlled if $p<\frac{4}{3}$.

For $d \geq 3$, we proceed differently.
If $d >p^*$, then H\"older's inequality yields $\|\overline\bu \, \bv^T\|_{0,p^*;\G} \leq \|\overline\bu\|_{0,d;\G} \|\bv\|_{0,\frac{d p^*}{d-p^*};\G}$.
If we further assume that $p \geq \frac{d}{2}$ (which actually implies $d \geq p^*$, with equality when $d=3$ and $p = \frac{3}{2}$), then a Sobolev embedding lets us write
\begin{equation*}
\|\overline\bu \, \bv^T\|_{0,p^*;\G}
\lesssim \|\overline\bu\|_{1,p;\G} \|\bv\|_{1,p^*;\G},
\qquad \forall \overline{\bu} \in \bW^{1,p}_t(\G), \bv \in \bW^{1,p^*}_t(\G),
\end{equation*}
unless $p = \frac{3}{2}$ and $d = 3$, for which $p^* = 3 = d$.

On the other hand, if $3 \leq d \leq p^*$ (which implies $d>p$), then it is easy to see that $\frac{dp}{d-p} \leq p^*$, with equality only for $d=3$, $p=\frac{3}{2}$, and so it is impossible to exploit H\"older's inequality in conjunction with the embedding $\bW^{1,p}(\G) \hookrightarrow \bL^{\frac{dp}{d-p}}$ to control $\|\overline\bu \, \bv^T\|_{0,p^*;\G}$.

Consequently, we have deduced that if
\begin{equation}\label{eq:range-for-p}
p \in
\begin{cases}
[\frac{4}{3},\infty) & \text{if $d=2$},\\
(\frac{3}{2}, \infty) & \text{if $d=3$},\\
[\frac{d}{2},\infty) & \text{if $d \geq 4$},
\end{cases}
\end{equation}
then
\begin{equation}\label{convection-estimate}
\int_\G |(\nablaG \bu)\overline{\bu} \cdot \bv|
\leq \|\nablaG \bu\|_{0,p;\G} \|\overline\bu \, \bv^T\|_{0,p^*;\G}
\leq C_{\operatorname{conv}} \|\bu\|_{1,p;\G} \|\overline\bu\|_{1,p;\G} \|\bv\|_{1,p^*;\G},
\end{equation}
for each $\bu, \overline{\bu} \in \bW^{1,p}_t(\G)$ and $\bv \in \bW^{1,p^*}_t(\G)$.

\subsection{Tangent Oseen equations}\label{sec:Oseen}

In this section, we establish the well-posedness of the tangent Oseen equations following ideas from \cite[\textsection 6]{DindosMitrea:2004}.
Formally speaking, the tangent Oseen equations read: given an incompressible velocity field $\overline\bu$, we seek a tangential velocity field $\bu: \G \rightarrow \RR^{d+1}$ and a pressure field $\uppi: \G \rightarrow \RR$ such that
\begin{equation}\label{eq:manifold-Oseen}
\begin{aligned}
- \DeltaB \bu + (\nablaG \bu) \overline\bu + \nablaG \uppi &= \bf,\\
\divG \bu & = 0.
\end{aligned}
\end{equation}
System \eqref{eq:manifold-Oseen} is interesting as it can be regarded as a more physically accurate linearization of the incompressible tangent Navier--Stokes equations \eqref{eq:manifold-Navier--Stokes} than the incompressible tangent Stokes equations \eqref{eq:manifold-stokes} by introducing partial convection in terms of an a-priori approximation $\overline\bu$ of the unknown velocity field $\bu$.
It is also worth mentioning that in the context of numerical methods for Navier--Stokes in flat domains, Oseen-type equations naturally arise as a result of performing semi-implicit discretizations in time of the convective term $(\nabla \bu)\bu$.

As in flat domains, the analysis of the tangent Oseen equations hinges on the ``skew-symmetry'' property enjoyed by the convective term $(\nablaG \bu)\overline\bu$ provided $\overline\bu$ is $\divG$-free.
We now make this statement more precise.
Suppose $\G$ is of class $C^2$, let $r \in (1,\infty)$ and $\overline\bu \in \bL^r_{t,\sigma}(\G)$.
Then, for each $\bu,\bv \in \bC^1_t(\G)$, since $\divG \overline\bu = 0$, definition \eqref{eq:def-distr-divergence} with $\bv$ replaced by $\overline\bu$ and $q$ replaced by $\bu \cdot \bv$ yields
\begin{equation}\label{eq:PROTO-skew-symmetry-convective-smoothfunctions}
0
= \int_\G \overline\bu \cdot \nablaG(\bu \cdot \bv)
= \int_\G \overline\bu \cdot ((\nablaG\bu)^T \bv + (\nablaG\bv)^T \bu),
\end{equation}
and therefore
\begin{equation}\label{eq:skew-symmetry-convective-smoothfunctions}
\int_\G (\nablaG \bu)\overline\bu \cdot \bv = - \int_\G (\nablaG \bv)\overline\bu \cdot \bu, \qquad \bu,\bv \in \bC^1_t(\G).
\end{equation}
That is, by fixing $\overline\bu$, the bilinear form $(\bu,\bv) \mapsto \int_\G (\nablaG \bu)\overline\bu \cdot \bv$ is skew-symmetric.
In general, identity \eqref{eq:skew-symmetry-convective-smoothfunctions} can be extended to $\bu$ and $\bv$ lying in appropriate Sobolev spaces such that $\bu \cdot \bv \in \WW^{1,r^*}(\G)$.
The following preliminary result makes this last observation rigorous and shall prove essential in establishing the well-posedness of an appropriate weak formulation for problem \eqref{eq:manifold-Oseen}.

\begin{proposition}[linearized convective operator]\label{prop:convective-compact}
Let $p \in (1,\infty)$ and assume $\G$ is of class $C^2$.
For each $\varepsilon>0$ define
\begin{equation}\label{def:integ-ubar}
r=r(\varepsilon,p,d) := \begin{cases}
2+\varepsilon, & \text{if $p=d=2$},\\
d, & \text{otherwise}.
\end{cases}
\end{equation}
Then, for each $\overline\bu \in \bL^{r}_{t,\sigma}(\G)$ the linear operator $K_{\overline\bu}: \bW^{1,p}_t(\G) \rightarrow (\bW^{1,p^*}_t(\G))'$ given by
\begin{equation}\label{def:K-operator}
\langle K_{\overline\bu}(\bu),\bv \rangle
:=
\begin{dcases}
\int_\G (\nablaG \bu)\overline\bu \cdot \bv, & \text{if $p^*<d$ or $p=d=2$},\\
- \int_\G (\nablaG \bv)\overline\bu \cdot \bu, & \text{if $p<d$ or $p=d=2$},
\end{dcases}
\end{equation}
is well-defined, continuous and compact.
Moreover, the mapping $\overline\bu \mapsto K_{\overline\bu}$ is linear and continuous from $\bL^r_{t,\sigma}(\G)$ to $\sL(\bW^{1,p}_t(\G),(\bW^{1,p^*}_t(\G))')$ with bound
\begin{equation}\label{K-normbound}
\|K_{\overline\bu}\|_{\sL(\bW^{1,p}_t(\G),(\bW^{1,p^*}_t(\G))')} \leq C \|\overline\bu\|_{0,r;\G},
\end{equation}
where $C>0$ depends only on $\G$, $p$ and, if $p=d=2$, also on $\varepsilon$.

\begin{proof}
In view of the density of $\bC^1_t(\G)$ in both $\bW^{1,p}_t(\G)$ and $\bW^{1,p^*}_t(\G)$ (cf.~\autoref{lem:density-tangential}) it is enough to first consider $\bu,\bv \in \bC^1_t(\G)$.
For this pair of functions the integrals in \eqref{def:K-operator} are clearly well-defined and, on account of \eqref{eq:skew-symmetry-convective-smoothfunctions}, coincide.
Consequently, H\"older's inequality and \autoref{prop:GNG-embeddings} (Gagliardo--Nirenberg--Sobolev)  yields
\begin{equation}\label{estimate-I}
\begin{aligned}
|\langle K_{\overline\bu}(\bu),\bv \rangle|
& \leq \begin{cases}
\|\nablaG\bv\|_{0,p^*;\G} \|\overline\bu\|_{0,d;\G} \|\bu\|_{0,\frac{dp}{d-p};\G}, & \text{if $p<d$},\\
\|\nablaG\bu\|_{0,p;\G} \|\overline\bu\|_{0,d;\G} \|\bv\|_{0,\frac{dp^*}{d-p^*};\G}, & \text{if $p^*<d$},\\
\|\nablaG\bv\|_{0,\G} \|\overline\bu\|_{0,2+\varepsilon;\G} \|\bu\|_{0,q(\varepsilon);\G} , & \text{if $p=p^*=d=2$},
\end{cases}\\
& \leq C \|\bu\|_{1,p;\G} \|\overline\bu\|_{0,r;\G} \|\bv\|_{1,p^*;\G},
\end{aligned}
\end{equation}
where $q =q(\varepsilon) \in (1,\infty)$ is such that $\frac{1}{2} + \frac{1}{2+\varepsilon} + \frac{1}{q} = 1$.
Notice that if both conditions $p<d$ and $p^*<d$ are met, then H\"older's inequality and Sobolev embeddings $\bW^{1,p}(\G) \hookrightarrow \bL^{\frac{dp}{d-p}}(\G)$ and $\bW^{1,p^*}(\G) \hookrightarrow \bL^{\frac{dp^*}{d-p^*}}(\G)$ imply that $\bu \cdot \bv \in \WW^{1,d^*}(\G)$ and so
the aforementioned density of $\bC^1_t(\G)$ lets us extend the validity of identity \eqref{eq:skew-symmetry-convective-smoothfunctions} for $r=d \geq 3$ to $\bu \in \bW^{1,p}_t(\G)$ and $\bv \in \bW^{1,p^*}_t(\G)$.
Similarly, if $p=d=2$, we can extend identity \eqref{eq:skew-symmetry-convective-smoothfunctions} for $r=2+\varepsilon$ to $\bu,\bv \in \bH^1_t(\G)$ upon utilizing the Sobolev embedding $\bH^1(\G) \hookrightarrow \bL^q(\G)$ for each $q \in (1,\infty)$.

This fact, in conjunction with the aforementioned density result, the linearity of $\langle K_{\overline\bu}(\bu),\bv \rangle$ in each of its arguments and estimate \eqref{estimate-I} imply that $K_{\overline\bu}$ is well-defined and satisfies
\begin{equation*}
|\langle K_{\overline\bu}(\bu),\bv \rangle| \leq C \|\bu\|_{1,p;\G} \|\overline\bu\|_{0,r;\G} \|\bv\|_{1,p^*;\G}, \qquad \forall \bu \in \bW^{1,p}_t(\G), \, \bv \in \bW^{1,p^*}_t(\G),
\end{equation*}
that is, $K_{\overline\bu}$ is continuous and
\begin{equation*}
\|K_{\overline\bu}\|_{\sL(\bW^{1,p}_t(\G),(\bW^{1,p^*}_t(\G))')} \leq C \|\overline\bu\|_{0,r;\G},
\end{equation*}
which is precisely \eqref{K-normbound}.

Regarding the compactness of $K_{\overline\bu}$, we first observe that in the case $p=p^*=d=2$, estimate \eqref{estimate-I} let us write
\begin{equation*}
\|K_{\overline\bu}(\bu)\|_{(\bH^1_t(\G))'} \leq \|\overline\bu\|_{0,2+\varepsilon;\G} \|\bu\|_{0,q(\varepsilon);\G}.
\end{equation*}
This, together with the compact embedding $\bH^1(\G) \hookrightarrow \bL^{q(\varepsilon)}(\G)$, implies that $K_{\overline\bu}$ is indeed compact.

For the other cases, we proceed via an approximation argument.
Using the density of $\bC^1_t(\G)$ in $\bL^d_t(\G)$ we first construct a sequence $\{\bvarphi_n\} \subseteq \bC^1_t(\G)$ such that $\lim_{n\to\infty}\|\bvarphi_n-\overline\bu\|_{0,d;\G} = 0$.
Hence, by \autoref{prop:Helmholtzdecomposition} (Helmholtz--Weyl decomposition on $\G$) we have that $\overline\bu_n := P_\sigma \bvarphi_n = \bvarphi_n - \nablaG \DeltaG^{-1} \divG \bvarphi_n \in \bW^{1,p}_{t,\sigma}(\G)$ for each $p \in (1,\infty)$ and
\begin{equation*}
\|\overline\bu_n - \overline\bu\|_{0,d;\G}
= \|P_\sigma\bvarphi_n - \overline\bu\|_{0,d;\G}
= \|P_\sigma(\bvarphi_n - \overline\bu)\|_{0,d;\G}
\lesssim \|\bvarphi_n - \overline\bu\|_{0,d;\G}
\xrightarrow{n\to\infty} 0.
\end{equation*}
Consequently, for each $n \in \NN$, we have that
\begin{equation*}
\|K_{\overline\bu_n} - K_{\overline\bu}\|_{\sL(\bW^{1,p}_t(\G),(\bW^{1,p^*}_t(\G))')} \leq C \|\overline\bu_n - \overline\bu\|_{0,d;\G} \xrightarrow{n\to\infty} 0,
\end{equation*}
that is, $K_{\overline\bu_n} \to K_{\overline\bu}$ in $\sL(\bW^{1,p}_t(\G),(\bW^{1,p^*}_t(\G))')$.
Given the higher integrability of $\overline\bu_n$ we can easily rewrite estimate \eqref{estimate-I}, thus obtaining
\begin{equation*}
\begin{cases}
\|K_{\overline\bu_n}(\bu)\|_{(\bW^{1,p^*}_t(\G))'} \leq \|\overline\bu_n\|_{0,d+\alpha;\G} \|\bu\|_{0,\frac{dp}{d-p}-\eta_1;\G}, & \text{if $p<d$},\\
\|K_{\overline\bu_n}^\mathrm{t}(\bv)\|_{(\bW^{1,p}_t(\G))'} \leq \|\overline\bu_n\|_{0,d+\alpha;\G} \|\bv\|_{0,\frac{dp^*}{d-p^*}-\eta_2;\G}, & \text{if $p^*<d$},\\
\end{cases}
\end{equation*}
where $\alpha > 0$ is fixed but arbitrary, and $\eta_1, \eta_2>0$ depend only on $\alpha$, $p$ and $d$.
In this way, using \autoref{prop:Rellich-Kondrachov} (compact Sobolev embedding) and \cite[Theorem 6.4]{Brezis2011}, we deduce that each $K_{\overline\bu_n}$ is compact.
Since the space of compact operators is closed \cite[Theorem 6.1]{Brezis2011}
and $K_{\overline\bu_n} \to K_{\overline\bu}$ in $\sL(\bW^{1,p}_t(\G),(\bW^{1,p^*}_t(\G))')$, we conclude that $K_{\overline\bu}$ is compact.
\end{proof}
\end{proposition}

We now establish the well-posedness of the Oseen equations.

\begin{thm}[well-posedness of the Oseen equations]\label{thm:Oseen-wellposedness}
Let $p \in (1,\infty)$ and assume $\G$ is of class $C^2$.
Let $r = r(\varepsilon,p,d)$ be as in \eqref{def:integ-ubar}.
Then, for each $\overline\bu \in \bL^r_{t,\sigma}(\G)$ and $\bf \in (\bW^{1,p^*}_t(\G))'$ there exists a unique pair $(\bu,\uppi) \in \bW^{1,p}_t(\G) \times \LL^p_\#(\G)$ such that
\begin{equation*}
\begin{aligned}
\int_\G \nablaG \bu : \nablaG \bv
+ \langle K_{\overline\bu}(\bu),\bv \rangle
- \int_\G \uppi \divG \bv & = \langle \bf , \bv \rangle
, \quad & \forall\, \bv \in \bW^{1,p^*}_t(\G),\\
- \int_\G q \divG \bu & = 0, \quad & \forall\, q \in \LL^{p^*}_\#(\G).
\end{aligned}
\end{equation*}
Moreover, there exists $C_{\overline\bu} > 0$, depending only on $\overline\bu$, $p$, $\G$ and, if $p=d=2$, also on $\varepsilon$, such that
\begin{equation}\label{linearized-NS-generalestimate}
\|\bu\|_{1,p;\G} + \|\uppi\|_{0,p;\G} \leq C_{\overline\bu} \|\bf\|_{(\bW^{1,p^*}_t(\G))'}.
\end{equation}
Finally, if $p = 2$, we have the following $\LL^2$-based estimates
\begin{equation}\label{linearized-NS-L2based-estimate}
\|\bu\|_{1,\G} \leq C_1 \|\bf\|_{(\bH^1_t(\G))'}, \qquad
\|\uppi\|_{0,\G} \leq C_2 \|\bf\|_{(\bH^1_t(\G))'},
\end{equation}
where $C_1 >0$ depends only on $\G$ and $C_2>0$ depends only on $\G$, affinely on $\|\overline\bu\|_{0,r;\G}$ and, if $p=d=2$, also on $\varepsilon$.

\begin{proof}
We first establish the well-posedness of the following reduced problem:
given $\bf \in (\bW^{1,p^*}_{t,\sigma}(\G))'$, find $\bu \in \bW^{1,p}_{t,\sigma}(\G)$ such that
\begin{equation}\label{linearized-NS-in-kernel}
\int_\G \nablaG \bu : \nablaG \bv
+ \langle K_{\overline\bu}(\bu),\bv \rangle
= \langle \bf , \bv \rangle, \qquad \forall \bv \in \bW^{1,p^*}_{t,\sigma}(\G).
\end{equation}
We observe that upon defining the linear and bounded operators $L,\wt K: \bW^{1,p}_{t,\sigma}(\G) \rightarrow (\bW^{1,p^*}_{t,\sigma}(\G))'$ by
\begin{equation*}
\langle L(\bu),\bv \rangle := \int_\G \nablaG \bu : \nablaG \bv,
\qquad\qquad
\langle \wt K(\bu),\bv \rangle
\stackrel{\eqref{def:K-operator}}{:=} \langle K_{\overline\bu}(\bu),\bv \rangle,
\end{equation*}
we can equivalently express \eqref{linearized-NS-in-kernel} as:
find $\bu \in \bW^{1,p}_{t,\sigma}(\G)$ such that
\begin{equation*}
(L+\wt K)\bu = \bf.
\end{equation*}
Applying \autoref{prop:non-Hilbert-BL-divergencefree-uniqueness} (injectivity of the adjoint operator) and \autoref{prop:infsup-BochnerLaplacian-nonHilbert} (surjectivity of the adjoint operator) to $L$, along with \cite[Theorem~A.2]{BenavidesNochettoShakipov2025-a} (Banach--Ne\v{c}as--Babu\v{s}ka theorem) yields the invertibility of $L$.
In turn, noticing that $\wt K = \iota^*_{p^*} \circ K_{\overline\bu} \circ \iota_p$, where $\iota_p: \bW^{1,p}_{t,\sigma}(\G) \to \bW^{1,p}_t(\G)$ and $\iota^*_{p^*}: (\bW^{1,p^*}_t(\G))' \to (\bW^{1,p^*}_{t,\sigma}(\G))'$ are respectively the canonical injection and bounded restriction map, we deduce from \autoref{prop:convective-compact} (linearized convective operator) that $\wt K$ is compact.

Since $L$ is bijective and $\wt K$ is compact, we deduce from the Fredholm's alternative for reflexive Banach spaces \cite[Theorem 6.6]{Brezis2011} (see also \cite[Teorema 6.9]{Gatica2024}) that $L+\wt K$ has closed range and $\dim \operatorname{ker}(L+\wt K) = \dim \operatorname{ker}((L+\wt K)^\mathrm{t}) < \infty$ (cf.~\eqref{Banach-duality} in \autoref{app:var-prob}).
We also know that it suffices to prove that $\ker(L+\wt K) = \{\mathbf{0}\}$ or $\ker(L+\wt K)^\mathrm{t} = \{\mathbf{0}\}$ in order to conclude that $L+\wt K$ is invertible.

If $p \geq 2$, let $\bu \in \ker(L+\wt K)$;
that is,
\begin{equation*}
\int_\G \nablaG \bu : \nablaG \bv
+ \langle K_{\overline\bu}(\bu),\bv \rangle
= 0, \qquad \forall \bv \in \bW^{1,p^*}_{t,\sigma}(\G).
\end{equation*}
Choosing $\bv = \bu \in \bW^{1,p}_{t,\sigma}(\G) \subseteq \bH^1_{t,\sigma}(\G) \subseteq \bW^{1,p^*}_{t,\sigma}(\G)$ and using that $\langle K_{\overline\bu}(\bu),\bu \rangle = \int_\G (\nablaG \bu)\overline{\bu} \cdot \bu = 0$, we obtain $\|\nablaG \bu\|_{0,\G} = 0$.
Thus, by \autoref{lem:Poincare} (Poincaré inequality), $\bu = \mathbf{0}$.

If $p < 2$, let $\bv \in \ker((L+\wt K)^\mathrm{t})$;
that is
\begin{equation*}
\langle (L+\wt K)^\mathrm{t} \bv, \bu \rangle = \langle (L+\wt K) \bu, \bv \rangle = \int_\G \nablaG \bu : \nablaG \bv
+ \langle K_{\overline\bu}(\bu),\bv \rangle
= 0, \qquad \forall \bu \in \bW^{1,p}_{t,\sigma}(\G).
\end{equation*}
Since $p^* > 2 > p$, we choose $\bu = \bv \in \bW^{1,p^*}_{t,\sigma}(\G) \subseteq \bH^1_{t,\sigma}(\G) \subseteq \bW^{1,p}_{t,\sigma}(\G)$ to arrive at $\|\nablaG \bv\|_{0,\G} = 0$, and so $\bv = \mathbf{0}$.

Overall, for $1<p<\infty$, we have proved that $\dim \operatorname{ker}(L+\wt K) = \dim \operatorname{ker}((L+\wt K)^\mathrm{t}) = 0$ and so $L+\wt K$ is bijective.
This proves the unique solvability of \eqref{linearized-NS-in-kernel}.
Moreover, a direct application of the Bounded Inverse Theorem \cite[Corollary 2.7]{Brezis2011} yields the existence of a positive constant $\wt C_{\overline\bu}$ depending only on $\G$, $p$ and $\overline\bu$ such that
\begin{equation}\label{eq:linearized-NS-gen-vel-est}
\|\bu\|_{1,p;\G} \leq \wt C_{\overline\bu} \|\bf\|_{(\bW^{1,p^*}_{t,\sigma}(\G))'}.
\end{equation}
We consider now $\bf \in (\bW^{1,p^*}_t(\G))'$ and deal with the pressure field $\uppi$.
We claim that there exists a unique $\uppi \in \LL^p_\#(\G)$ such that
\begin{equation}\label{eq:finding-pressure-existenceNS}
\int_\G \nablaG \bu : \nablaG \bv
+ \langle K_{\overline\bu}(\bu),\bv \rangle
- \int_\G \uppi \divG \bv
= \langle \bf , \bv \rangle
, \quad \forall\, \bv \in \bW^{1,p^*}_t(\G).
\end{equation}
Proceeding as in the proof of \autoref{prop:non-Hilbert-BL-divergencefree-uniqueness}, we notice that because of \autoref{prop:Helmholtzdecomposition} (Helmholtz--Weyl decomposition on $\G$), \eqref{eq:finding-pressure-existenceNS} is equivalent to
\begin{equation}\label{eq:recover-pressure-NS}
- \int_\G \uppi \, \DeltaG \phi
= \langle \bf , \nablaG \phi \rangle - \int_\G \nablaG \bu : \nablaG \nablaG \phi - \langle K_{\overline\bu}(\bu), \nablaG\phi \rangle
=: \langle h , \phi \rangle
, \qquad \forall \phi \in \WW^{2,p^*}_\#(\G).
\end{equation}
Consequently, the existence of a unique $\uppi \in \LL^p_\#(\G)$ follows from \autoref{lem:ultraweak-allp-LB} (ultra-weak solutions of the Laplace--Beltrami operator on $\G$) with right-hand side datum $h \in (\WW^{2,p^*}_\#(\G))'$.
Moreover, the following a-priori bound holds true
\begin{equation}\label{Oseen-pressure-bound}
\begin{split}
    \|\uppi\|_{0,p;\G}
    & \lesssim \|h\|_{(\WW^{2,p^*}_\#(\G))'}
    \lesssim \|\bf\|_{(\bW^{1,p^*}_t(\G))'} + \|\nablaG\bu\|_{0,p;\G} + \|K_{\overline\bu}(\bu)\|_{(\bW^{1,p^*}_t(\G))'} \\
    & \stackrel{\eqref{K-normbound}}{\lesssim} \|\bf\|_{(\bW^{1,p^*}_t(\G))'} + (1+\|\overline\bu\|_{0,r;\G}) \|\bu\|_{1,p;\G}.
\end{split}
\end{equation}
Combining \eqref{Oseen-pressure-bound} with \eqref{eq:linearized-NS-gen-vel-est} results in \eqref{linearized-NS-generalestimate}.

Finally, if $p = 2$ we choose $\bv = \bu$ in \eqref{linearized-NS-in-kernel} and notice that $\langle K_{\overline\bu}(\bu),\bu \rangle = 0$, to obtain
\begin{equation*}
\|\nablaG\bu\|_{0,\G}^2 = \|\nablaG\bu\|_{0,\G}^2 + \int_\G (\nablaG \bu)\overline\bu \cdot \bu  = \langle \bf , \bu \rangle \leq \|\bf\|_{(\bH^1_t(\G))'} \|\bu\|_{1,\G},
\end{equation*}
which combined with \autoref{lem:Poincare} (Poincaré inequality) for $p=2$, yields the estimate for $\bu$ in \eqref{linearized-NS-L2based-estimate}.
The estimate for $\uppi$ in \eqref{linearized-NS-L2based-estimate} follows from that of $\bu$ and \eqref{Oseen-pressure-bound}.
\end{proof}
\end{thm}

\begin{remark}[dependence on $\overline\bu$]\label{rem:Oseen-dep-bar-u}
We find it important to mention that the specific structure of the dependence on $\overline\bu$ of the constant $C_{\overline\bu}$ appearing in \eqref{linearized-NS-generalestimate} is unknown.
This is due to the fact that the Fredholm Alternative \cite[Theorem 6.6]{Brezis2011} only yields a statement about invertibility and the Bounded Inverse Theorem \cite[Corollary 2.7]{Brezis2011} does not provide an explicit continuity constant.
\end{remark}

\subsection{Existence of solutions of the incompressible tangent Navier--Stokes equations}\label{sec:existence-incompressibleNS-d234}

In this section we establish the existence of solutions to the tangent incompressible Navier--Stokes equations \eqref{eq:weak-manifold-Navier--Stokes} (i.e.~with $g=0$) for dimensions $d \in \{2,3,4\}$ and $p \geq 2$;
see \autoref{thm:existence-NS-d234} below.
For the case $p=2$ ($\LL^2$-based) we utilize the standard Faedo--Galerkin method based upon eigenfunctions of the Stokes operator (cf.~\autoref{sec:Stokes-operator}) to recover a velocity field $\bu$.
Given the lack of boundary of $\G$, we recover the pressure field $\uppi$ in terms of $\bu$ and $\bf$ by solving a Laplace--Beltrami problem for $\uppi$.
Moreover, the case $p>2$ follows from the case $p=2$ in conjunction with the well-posedness of the tangent Oseen equations (cf.~\autoref{thm:Oseen-wellposedness}).

\begin{proposition}[existence of a velocity field]\label{prop:NS-kernel}
Let $d \in \{2,3,4\}$ and assume $\G$ is of class $C^2$.
Then, for each $\bf \in (\bH^1_{t,\sigma}(\G))'$ there exists $\bu \in \bH^1_{t,\sigma}(\G)$ such that
\begin{equation}\label{eq:NS-kernel}
\int_\G \nablaG\bu : \nablaG \bv + \int_\G (\nablaG\bu)\bu \cdot \bv = \langle \bf,\bv \rangle, \qquad \forall \bv \in \bH^1_{t,\sigma}(\G).
\end{equation}
\begin{proof}
We proceed by a Faedo--Galerkin method.
Let $\{\bv_n\}_{n \in \NN} \subseteq \bigcap_{2 \leq p<\infty} \bW^{1,p}_{t,\sigma}(\G)$ be the family of functions provided by \autoref{thm:Stokes-eigen-minreg} (eigendecomposition of the Stokes operator).
For each fixed $n \in \NN$, we seek an approximated velocity $\bu_n \in \vspan\{\bv_i\}_{i=1}^n$ such that
\begin{equation}\label{eq:Galerkin-kernel}
\int_\G \nablaG\bu_n : \nablaG \bv_i + \int_\G (\nablaG\bu_n)\bu_n \cdot \bv_i = \langle \bf,\bv_i \rangle, \qquad \forall i \in \{1,\dotsc,n\}.
\end{equation}
\textbf{Step 1: Existence of Galerkin solutions.}
The existence of solutions of \eqref{eq:Galerkin-kernel} will follow from the abstract result \cite[p.~110, Lemma 1.4]{Temam2001}, whose hypotheses we now verify.
Let $X_n := \vspan\{\bv_i\}_{i=1}^n$ be the finite dimensional vector space endowed with the inner product $(\bu,\bv)_H := \int_\G \nablaG\bu : \nablaG\bv$ from \eqref{equiv:H1tsig-inner-prod}.
Let $P: X_n \rightarrow X_n$ be the operator defined by
\begin{equation*}
(P(\bu),\bv )_H := \int_\G \nablaG\bu : \nablaG \bv + \int_\G (\nablaG\bu)\bu \cdot \bv - \langle \bf,\bv\rangle, \qquad \forall \bu,\bv \in X_n,
\end{equation*}
which is clearly continuous.
Moreover, by noticing $\int_\G (\nablaG\bu)\bu \cdot \bu = 0$ for each $\bu \in \bH^1_{t,\sigma}(\G)$ (because of identity \eqref{eq:skew-symmetry-convective-smoothfunctions} and density arguments) and resorting to \autoref{lem:Poincare} (Poincaré inequality), we have that
\begin{multline*}
( P(\bu),\bu)_H
= \int_\G \nablaG\bu : \nablaG \bu + \int_\G (\nablaG\bu)\bu \cdot \bu - \langle \bf,\bu\rangle
= \|\bu\|_H^2 - \langle \bf, \bu \rangle\\
\geq \|\bu\|_H^2 - \|\bf\|_{(\bH^1_{t,\sigma}(\G))'} \|\bu\|_{1,\G}
\geq \|\bu\|_H (\|\bu\|_H - c\|\bf\|_{(\bH^1_{t,\sigma}(\G))'}),
\qquad \forall \bu \in X_n.
\end{multline*}
Hence, it is clear that for each $K > c\|\bf\|_{(\bH^1_{t,\sigma}(\G))'}$ it holds that
\begin{equation*}
(P(\bu),\bu)_H > 0, \qquad \forall \bu \in X_n, \quad \|\bu\|_H = K.
\end{equation*}
This finishes the verification of the hypotheses of \cite[p.~110, Lemma 1.4]{Temam2001}, whence we deduce that there exists $\bu \in X_n$ such that $P(\bu) = \textbf{0}$, or equivalently, that the Galerkin system \eqref{eq:Galerkin-kernel} admits a solution $\bu_n \in X_n$.
\\ \\
\textbf{Step 2: Passage to the limit.}
Since $\bu_n \in X_n$ and \eqref{eq:Galerkin-kernel} is linear in $\bv_i$, we can take $\bv_i \gets \bu_n$ in \eqref{eq:Galerkin-kernel} and use again \eqref{eq:skew-symmetry-convective-smoothfunctions}, to arrive at
\begin{equation*}
\|\nablaG\bu_n\|_{0,\G}^2
= \|\nablaG\bu_n\|_{0,\G}^2 + \int_\G (\nablaG\bu_n)\bu_n \cdot \bu_n
= \langle \bf,\bu_n \rangle
\leq \|\bf\|_{(\bH^1_{t,\sigma}(\G))'} \|\bu_n\|_{1,\G}.
\end{equation*}
Hence, by \autoref{lem:Poincare} (Poincaré inequality), we obtain the a-priori estimate
\begin{equation*}
\|\bu_n\|_{1,\G} \leq C \|\bf\|_{(\bH^1_{t,\sigma}(\G))'}, \qquad \forall n \in \NN,
\end{equation*}
where $C > 0$ is independent of $n$.
Since the sequence of Galerkin solutions $\{\bu_n\}_{n \in \NN}$ is bounded in the Hilbert space $\bH^1_{t,\sigma}(\G)$, it follows from the Banach--Alaoglu Theorem \cite[Theorem 3.18]{Brezis2011} and \autoref{prop:Rellich-Kondrachov} (compact Sobolev embedding) that there exists $\bu \in \bH^1_{t,\sigma}(\G)$ such that (up to a subsequence)
\begin{align}
\label{galerkin-BA} \bu_n & \rightharpoonup \bu, \qquad \text{in $\bH^1_{t,\sigma}(\G)$},\\
\label{galerkin-compact} \bu_n & \to \bu, \qquad \text{in $\bL^{4-\varepsilon}(\G)$},
\end{align}
for some arbitrary, but fixed, $\varepsilon \in (0,1)$.
Notice that $\varepsilon$ can be taken as $0$ if $d \in \{2,3\}$.
On the other hand, even for $d=4$, $\bu \in \bL^4(\G)$ because of the continuous injection $\bH^1(\G) \hookrightarrow \bL^4(\G)$.

For any $\bw \in \bH^1_{t,\sigma}(\G)$ we write $\bw = \sum_{i=1}^\infty \alpha_i \bv_i$, with the series converging in the $\bH^1(\G)$-norm.
Denoting $\bw_N := \sum_{i=1}^N \alpha_i \bv_i \in X_N$ for fixed $N \leq n$, we deduce from \eqref{eq:Galerkin-kernel} that
\begin{equation}\label{eq:Galerkin-wN}
\int_\G \nablaG\bu_n : \nablaG \bw_N + \int_\G (\nablaG\bu_n)\bu_n \cdot \bw_N = \langle \bf,\bw_N \rangle.
\end{equation}
By utilizing the weak convergence \eqref{galerkin-BA}, we can first pass to the limit as $n\to\infty$ and then as $N \to \infty$ in the first term of the left-hand side of \eqref{eq:Galerkin-wN}, thus arriving at
\begin{equation}\label{eq:Galerkin-tolimit-1}
\int_\G \nablaG\bu_n : \nablaG \bw_N
\xrightarrow{n\to\infty}
\int_\G \nablaG\bu : \nablaG \bw_N
\xrightarrow{N\to\infty}
\int_\G \nablaG\bu : \nablaG \bw.
\end{equation}
Regarding the nonlinear term in \eqref{eq:Galerkin-wN}, we first notice that using the strong convergence \eqref{galerkin-compact} it is easy to see that $\bu_n \otimes \bu_n \to \bu \otimes \bu$  in $\bbL^{\frac{4-\varepsilon}{2}}(\G)$.
Thus, taking also into account the incompressibility condition $\divG \bu_n = 0$ and identity \eqref{eq:skew-symmetry-convective-smoothfunctions} with $\overline\bu = \bu_n$ and the fact that $\nablaG \bw_N \in \bbL^p(\G)$ for every $p \in (1,\infty)$, it first transpires that
\begin{equation}\label{eq:Galerkin-tolimit-4}
\int_\G (\nablaG\bu_n)\bu_n \cdot \bw_N
= - \int_\G (\nablaG\bw_N): \bu_n \otimes \bu_n
\xrightarrow{n\to\infty}
- \int_\G (\nablaG\bw_N): \bu \otimes \bu.
\end{equation}
Moreover, because $\bu \otimes \bu \in \bbL^2(\G)$ and $\nablaG\bw_N \to \nablaG\bw$ in $\bbL^2(\G)$ as $N\to\infty$, identity \eqref{eq:skew-symmetry-convective-smoothfunctions} with $\overline\bu = \bu$, yields
\begin{equation}\label{eq:Galerkin-tolimit-2}
- \int_\G (\nablaG\bw_N): \bu \otimes \bu
\xrightarrow{N\to\infty}
- \int_\G (\nablaG\bw): \bu \otimes \bu
= \int_\G (\nablaG\bu)\bu \cdot \bw.
\end{equation}
Finally, since $\bw_N \to \bw$ in $\bH^1_{t,\sigma}(\G)$ as $N\to\infty$, it also holds that
\begin{equation}\label{eq:Galerkin-tolimit-3}
\langle \bf,\bw_N \rangle
\xrightarrow{N\to\infty}
\langle \bf,\bw \rangle.
\end{equation}
Overall, combining \eqref{eq:Galerkin-wN}, \eqref{eq:Galerkin-tolimit-1}, \eqref{eq:Galerkin-tolimit-4}, \eqref{eq:Galerkin-tolimit-2} and \eqref{eq:Galerkin-tolimit-3} yields \eqref{eq:NS-kernel} and finishes the proof.
\end{proof}
\end{proposition}

\begin{thm}[existence of solutions of the incompressible tangent Navier--Stokes equations]\label{thm:existence-NS-d234}
Let $p \in [2,\infty)$ and assume $\G$ is of class $C^2$ and of dimension $d \in \{2,3,4\}$.
Then, for each $\bf \in (\bW^{1,p^*}_t(\G))'$ there exists a solution $(\bu,\uppi) \in \bW^{1,p}_t(\G) \times \LL^p_\#(\G)$ of \eqref{eq:weak-manifold-Navier--Stokes} with $g = 0$.
Moreover, the following a-priori bounds hold for $p=2$:
\begin{equation}\label{incompNS-L2based-ests}
\|\bu\|_{1,\G} \leq C_1 \|\bf\|_{(\bH^1_{t,\sigma}(\G))'}, \qquad \|\uppi\|_{0,\G} \leq C_2 \|\bf\|_{(\bH^1_t(\G))'} (1+\|\bf\|_{(\bH^1_t(\G))'}),
\end{equation}
where $C_1, C_2 > 0$ only depend on $\G$.

\begin{proof}
From \autoref{prop:NS-kernel} (existence of a velocity field) we deduce the existence of $\bu \in \bH^1_{t,\sigma}(\G)$ that satisfies \eqref{eq:NS-kernel} with the given $\bf \in (\bW^{1,p^*}_t(\G))'$ restricted to $(\bH^1_{t,\sigma}(\G)) \subseteq \bW^{1,p^*}_t(\G)$:
\begin{equation*}
\int_\G \nablaG \bu : \nablaG \bv + \int_\G (\nablaG \bu)\bu \cdot \bv = \langle \bf , \bv \rangle
, \qquad \forall \bv \in \bH^1_{t,\sigma}(\G).
\end{equation*}
Let $r = r(\varepsilon,p,d)$ be given by \eqref{def:integ-ubar} for some fixed $\varepsilon>0$.
Then, noticing that $\bu \in \bL^r_{t,\sigma}(\G)$, because of \autoref{prop:GNG-embeddings} (Gagliardo--Nirenberg--Sobolev) and $d \leq 4$, a direct application of \autoref{thm:Oseen-wellposedness} (well-posedness of the Oseen equations) with $\overline\bu$ replaced by $\bu$ yields the existence of the pressure $\uppi \in \LL^p_\#(\G)$, and the higher regularity $\bu \in \bW^{1,p}_t(\G)$ and the fact that the pair $(\bu,\uppi) \in \bW^{1,p}_t(\G) \times \LL^p_\#(\G)$ solves \eqref{eq:weak-manifold-Navier--Stokes} with $g = 0$.
Moreover, estimates \eqref{incompNS-L2based-ests} follow from the first estimate in \eqref{linearized-NS-L2based-estimate}, and combining the case $p=2$ of estimate \eqref{Oseen-pressure-bound} with the continuous injection $\bH^1(\G) \hookrightarrow \bL^d(\G)$.
\end{proof}
\end{thm}

\begin{remark}[a-priori estimate for $p>2$]
It is worth stressing that, except for $p=2$, \autoref{thm:existence-NS-d234} does not provide an explicit bound for $\|\bu\|_{1,p;\G}$ in terms of $\|\bf\|_{(\bW^{1,p^*}_t(\G))'}$ according to \autoref{rem:Oseen-dep-bar-u} (dependence on $\overline\bu$).
The situation for the pressure $\uppi$ is different in view of estimate \eqref{Oseen-pressure-bound} with $\overline\bu = \bu$ within the proof of \autoref{thm:Oseen-wellposedness} (well-posedness of the Oseen equations) and the chain of inequalities $\|\bu\|_{0,r;\G} \lesssim \|\bu\|_{1,\G} \stackrel{\eqref{incompNS-L2based-ests}}{\lesssim} \|\bf\|_{(\bH^1_{t,\sigma}(\G))'} \lesssim \|\bf\|_{(\bW^{1,p^*}_t(\G))'}$, which reads
\begin{equation*}
\|\uppi\|_{0,p;\G}
\leq C_3 \left( \|\bf\|_{(\bW^{1,p^*}_t(\G))'} + (1+ \|\bf\|_{(\bW^{1,p^*}_t(\G))'}) \|\bu\|_{1,p;\G} \right),
\end{equation*}
where $C_3>0$ depends only on $\G$ and $p$.
\end{remark}

\begin{remark}[uniqueness of the pressure]
We find it important to mention that, in spite of \autoref{thm:existence-NS-d234} not providing a unique pair $(\bu,\uppi)$, the pressure $\uppi \in \LL^p_\#(\G)$ is uniquely determined by the velocity field $\bu$ and right-hand side $\bf$.
This is due to \eqref{eq:recover-pressure-NS} with $\overline\bu = \bu$ within the proof of \autoref{thm:Oseen-wellposedness} (well-posedness of the Oseen equations).
\end{remark}

\subsection{Well-posedness of tangent Navier--Stokes for small data}\label{sec:ns-smalldata}

In \autoref{sec:existence-incompressibleNS-d234} we proved the existence of solutions of the tangent Navier--Stokes system \eqref{eq:weak-manifold-Navier--Stokes} for $g=0$ and any datum $\bf \in (\bW^{1,p^*}_t(\G))'$ for $p \in [2,\infty)$ and dimension $d \in \{2,3,4\}$.
In this section we complement such a study by showing that for any dimension $d \geq 2$ and sufficiently small data $(\bf,g) \in (\bW^{1,p^*}_t(\G))' \times \LL^p_\#(\G)$ (with $g$ not necessarily zero and for $p$ within the range given by \eqref{eq:range-for-p}), there exists a unique solution of system \eqref{eq:weak-manifold-Navier--Stokes} with $\|\bu\|_{1,p;\G} \leq R$, with $R$ as in the following theorem.
Before stating our well-posedness result for problem \eqref{eq:weak-manifold-Navier--Stokes} we recall the following property satisfied by the variational formulation \eqref{eq:weak-manifold-stokes-nonHilbert} for the tangent Stokes equations:
since said problem defines a continuous linear isomorphism from $\bW^{1,p}_t(\G) \times \LL^p_\#(\G)$ onto $(\bW^{1,p^*}_t(\G))' \times \LL^{p}_\#(\G)$, \cite[Theorem~A.2]{BenavidesNochettoShakipov2025-a}
(Banach--Ne\v{c}as--Babu\v{s}ka theorem) guarantees the existence of a positive constant $C_{\operatorname{Stokes}}>0$ such that for each $(\bv,q) \in \bW^{1,p^*}_t(\G) \times \LL^{p^*}_\#(\G)$,
\begin{equation}\label{Stokes-global-inf-sup-1}
\sup_{\substack{(\bu,\uppi) \in \bW^{1,p}_t(\G) \times \LL^p_\#(\G)\\(\bu,\uppi) \neq (\mathbf{0},0)}} \frac{\int_\G \nablaG \bu : \nablaG \bv
- \int_\G \uppi \divG \bv
- \int_\G q \divG \bu
}{\|\bu\|_{1,p;\G} + \|\uppi\|_{0,p;\G}}
\geq C_{\operatorname{Stokes}} \left( \|\bv\|_{1,p^*;\G} + \|q\|_{0,p^*;\G} \right),
\end{equation}
and for each $(\bu,\uppi) \in \bW^{1,p}_t(\G) \times \LL^p_\#(\G)$,
\begin{equation}\label{Stokes-global-inf-sup-2}
\sup_{\substack{(\bv,q) \in \bW^{1,p^*}_t(\G) \times \LL^{p^*}_\#(\G)\\(\bv,q) \neq (\mathbf{0},0)}} \frac{\int_\G \nablaG \bu : \nablaG \bv
- \int_\G \uppi \divG \bv
- \int_\G q \divG \bu
}{\|\bv\|_{1,p^*;\G} + \|q\|_{0,p^*;\G}}
\geq C_{\operatorname{Stokes}} \left( \|\bu\|_{1,p;\G} + \|\uppi\|_{0,p;\G} \right).
\end{equation}
\begin{thm}[well-posedness of the tangent Navier--Stokes equations in $\bW^{1,p}_t(\G) \times \LL^p_\#(\G)$ for small data]
Let $p \in [\frac{4}{3},\infty)$ if $d=2$, $p \in (\frac{3}{2}, \infty)$ if $d=3$, or $p \in [\frac{d}{2},\infty)$ if $d \geq 4$.
Assume $\G$ is of class $C^2$.

Let $0 < R < \frac{C_{\operatorname{Stokes}}}{2 C_{\operatorname{conv}}}$, where $C_{\operatorname{conv}}$ is given in \eqref{convection-estimate} and $C_{\operatorname{Stokes}}$ in \eqref{Stokes-global-inf-sup-1} and \eqref{Stokes-global-inf-sup-2}.
Then, for each $\bf \in (\bW^{1,p^*}_t(\G))'$ and $g \in \LL^p_\#(\G)$ such that $\|\bf\|_{(\bW^{1,p^*}_t(\G))'} + \|g\|_{0,p;\G} \leq R \frac{C_{\operatorname{Stokes}}}{2}$,
there exists a unique solution $(\bu,\uppi) \in \bW^{1,p}_t(\G) \times \LL^p_\#(\G)$ of \eqref{eq:weak-manifold-Navier--Stokes} with $\|\bu\|_{1,p;\G} \leq R$.
Moreover, the following a-priori bound holds
\begin{equation*}
\|\bu\|_{1,p;\G} + \|\uppi\|_{0,p;\G} \leq \frac{2}{C_{\operatorname{Stokes}}} \left( \|\bf\|_{(\bW^{1,p^*}_t(\G))'} + \|g\|_{0,p;\G} \right).
\end{equation*}
\begin{proof}
Set $X_R := \{ \overline{\bu} \in \bW^{1,p}_t(\G): \|\overline{\bu}\|_{1,p;\G} \leq R \}$.
For each fixed $\overline{\bu} \in X_R$ we claim that the following linear problem is well-posed:
find $(\bu,\uppi) \in \bW^{1,p}_t(\G) \times \LL^p_\#(\G)$ such that
\begin{equation}\label{eq:fixed-point-operator}
\begin{aligned}
\int_\G \nablaG \bu : \nablaG \bv
+ \int_\G (\nablaG \bu)\overline{\bu} \cdot \bv - \int_\G \uppi \divG \bv & = \langle \bf , \bv \rangle
, \quad & \forall\, \bv \in \bW^{1,p^*}_t(\G),\\
- \int_\G q \divG \bu & = -\int_\G g \, q, \quad & \forall\, q \in \LL^{p^*}_\#(\G).
\end{aligned}
\end{equation}
Indeed, by combining the global inf-sup conditions \eqref{Stokes-global-inf-sup-1}, \eqref{Stokes-global-inf-sup-2} and estimate \eqref{convection-estimate}, we easily see that
\begin{equation*}
\begin{aligned}
\sup_{\textbf{0} \neq (\bu,\uppi) \in \bW^{1,p}_t(\G) \times \LL^p_\#(\G)} & \frac{\int_\G \nablaG \bu : \nablaG \bv
+ \int_\G (\nablaG \bu)\overline{\bu} \cdot \bv
- \int_\G \uppi \divG \bv
- \int_\G q \divG \bu
}{\|\bu\|_{1,p;\G} + \|\uppi\|_{0,p;\G}}\\
& \geq C_{\operatorname{Stokes}} \left( \|\bv\|_{1,p^*;\G} + \|q\|_{0,p^*;\G} \right) - C_{\operatorname{conv}} \|\overline\bu\|_{1,p;\G} \|\bv\|_{1,p^*;\G},\\
& \geq (C_{\operatorname{Stokes}} - C_{\operatorname{conv}} \|\overline\bu\|_{1,p;\G}) \left( \|\bv\|_{1,p^*;\G} + \|q\|_{0,p^*;\G} \right)\\
& \geq (C_{\operatorname{Stokes}} - C_{\operatorname{conv}} R) \left( \|\bv\|_{1,p^*;\G} + \|q\|_{0,p^*;\G} \right)\\
& \geq \frac{C_{\operatorname{Stokes}}}{2} \left( \|\bv\|_{1,p^*;\G} + \|q\|_{0,p^*;\G} \right),
\qquad \forall (\bv,q) \in \bW^{1,p^*}_t(\G) \times \LL^{p^*}_\#(\G),
\end{aligned}
\end{equation*}
because $R < \frac{C_{\operatorname{Stokes}}}{2 C_{\operatorname{conv}}}$.
Analogously,
\begin{equation*}
\begin{aligned}
\sup_{\textbf{0} \neq (\bv,q) \in \bW^{1,p^*}_t(\G) \times \LL^{p^*}_\#(\G)} & \frac{\int_\G \nablaG \bu : \nablaG \bv
+ \int_\G (\nablaG \bu)\overline{\bu} \cdot \bv
- \int_\G \uppi \divG \bv
- \int_\G q \divG \bu
}{\|\bv\|_{1,p^*;\G} + \|q\|_{0,p^*;\G}}\\
& \geq \frac{C_{\operatorname{Stokes}}}{2} \left( \|\bu\|_{1,p;\G} + \|\uppi\|_{0,p;\G} \right),
\qquad \forall (\bu,\uppi) \in \bW^{1,p}_t(\G) \times \LL^p_\#(\G).
\end{aligned}
\end{equation*}
Consequently, by \cite[Theorem~A.2]{BenavidesNochettoShakipov2025-a}
(Banach--Ne\v{c}as--Babu\v{s}ka theorem), we deduce that the linear problem \eqref{eq:fixed-point-operator} is well-posed.
In particular, it satisfies
\begin{equation*}
\|\bu\|_{1,p;\G} + \|\uppi\|_{0,p;\G}
\leq \frac{2}{C_{\operatorname{Stokes}}} \sup_{\substack{(\bv,q) \in \bW^{1,p^*}_t(\G) \times \LL^{p^*}_\#(\G)\\(\bv,q) \neq \textbf{0}}} \frac{\langle \bf , \bv \rangle -\int_\G g \, q}{\|\bv\|_{1,p^*;\G} + \|q\|_{0,p^*;\G}}
\leq \frac{2}{C_{\operatorname{Stokes}}} \left( \|\bf\|_{(\bW^{1,p^*}_t(\G))'} + \|g\|_{0,p;\G} \right).
\end{equation*}
Moreover, if $\bf$ and $g$ satisfy $\|\bf\|_{(\bW^{1,p^*}_t(\G))'} + \|g\|_{0,p;\G} \leq R \frac{C_{\operatorname{Stokes}}}{2}$, then $\bu \in X_R$.
Under this last assumption let us define the operator $T: X_R \rightarrow X_R$ that maps each $\overline{\bu} \in X_R$ to $T(\overline{\bu}) = \bu$, where $\bu$ is the first component of the solution $(\bu,\uppi)$ of \eqref{eq:fixed-point-operator} with the given $\overline{\bu}$.
Now, for each $\overline{\bu}_1, \overline{\bu}_2 \in X_R$, let $(\bu_1,\uppi_1)$ and $(\bu_2,\uppi_2)$ be the solutions of problem \eqref{eq:fixed-point-operator} with $\overline{\bu}_1$ and $\overline{\bu}_2$, respectively.
By subtracting their corresponding equations \eqref{eq:fixed-point-operator}, we deduce that $(\bu_1-\bu_2,\uppi_1-\uppi_2) \in \bW^{1,p}_t(\G) \times \LL^p_\#(\G)$ satisfies the variational formulation for the tangent Stokes equations \eqref{eq:weak-manifold-stokes-nonHilbert} with right-hand side functions $(\nablaG \bu_2) \overline{\bu}_2 - (\nablaG \bu_1) \overline{\bu}_1$ instead of $\bf$ and $0$ instead of $g$, the former understood as an element of $(\bW^{1,p^*}_t(\G))'$ by virtue of \eqref{convection-estimate}.
Consequently, the global inf-sup condition \eqref{Stokes-global-inf-sup-2} yields the following estimate
\begin{align*}
\|\bu_1-\bu_2\|_{1,p;\G} & + \|\uppi_1-\uppi_2\|_{0,p;\G}
\leq \frac{1}{C_{\operatorname{Stokes}}} \sup_{\substack{\bv \in \bW^{1,p^*}_t(\G)\\ \bv \neq \mathbf{0}} } \frac{\int_\G ((\nablaG \bu_2) \overline{\bu}_2 - (\nablaG \bu_1) \overline{\bu}_1) \cdot \bv}{\|\bv\|_{1,p^*;\G} }\\
& = \frac{1}{C_{\operatorname{Stokes}}} \sup_{\substack{\bv \in \bW^{1,p^*}_t(\G)\\ \bv \neq \mathbf{0} }} \frac{\int_\G ( \nablaG(\bu_2 - \bu_1) \overline{\bu}_2 + \nablaG \bu_1 (\overline{\bu}_2 - \overline{\bu}_1) ) \cdot \bv}{\|\bv\|_{1,p^*;\G} }\\
& \stackrel{\eqref{convection-estimate}}{\leq} \frac{C_{\operatorname{conv}}}{C_{\operatorname{Stokes}}} (\|\bu_1 - \bu_2\|_{1,p;\G} \|\overline\bu_2\|_{1,p;\G} + \|\bu_1\|_{1,p;\G} \|\overline\bu_1 - \overline\bu_2\|_{1,p;\G})\\
& \leq R \frac{C_{\operatorname{conv}}}{C_{\operatorname{Stokes}}} (\|\bu_1 - \bu_2\|_{1,p;\G} + \|\overline\bu_1 - \overline\bu_2\|_{1,p;\G})
\leq \frac{1}{2} \|\bu_1 - \bu_2\|_{1,p;\G} + R \frac{C_{\operatorname{conv}}}{C_{\operatorname{Stokes}}} \|\overline\bu_1 - \overline\bu_2\|_{1,p;\G};
\end{align*}
which, in particular implies
\begin{equation*}
\|T(\overline{\bu}_1)-T(\overline{\bu}_2)\|_{1,p;\G} = \|\bu_1-\bu_2\|_{1,p;\G}
\leq L \|\overline\bu_1 - \overline\bu_2\|_{1,p;\G},
\qquad \forall \overline{\bu}_1, \overline{\bu}_2 \in X_R.
\end{equation*}
with $L := 2R \frac{C_{\operatorname{conv}}}{C_{\operatorname{Stokes}}} < 1$;
that is, $T: X_R \rightarrow X_R$ is a contraction.
Therefore, the assertion follows from a straightforward application of the Banach fixed-point theorem \cite[Theorem 5.7]{Brezis2011}.
\end{proof}
\end{thm}

\subsection{Higher \texorpdfstring{$\LL^p$}{Lᵖ}-based Sobolev regularity of tangent Navier--Stokes}\label{sec:ns-regularity}

In this section we investigate the pick-up $\LL^p$-based regularity of solutions to the tangent Navier--Stokes equations beyond the basic regularity $(\bu,\uppi) \in \bW^{1,p}_t(\G) \times \LL^p_\#(\G)$ allowed by data $(\bf,g)$ smoother than in $(\bW^{1,p^*}_t(\G))' \times \LL^p_\#(\G)$.

The results are consistent with those in \autoref{sec:higher-reg-Stokes} for the tangent Stokes system.

\begin{thm}[higher $\LL^p$-based regularity of the tangent Navier--Stokes equations]\label{thm:Lpbased-regularity-NS}
Let $p \in (\frac{4}{3},\infty)$ if $d=2$ or $p \in (\frac{d}{2},\infty)$ if $d \geq 3$.
Moreover, assume $\G$ is of class $C^{m+2,1}$ for some nonnegative integer $m$.
If $(\bu,\uppi) \in \bW^{1,p}_t(\G) \times \LL^p_\#(\G)$ is a solution of \eqref{eq:weak-manifold-Navier--Stokes} with $\bf \in \bW^{m,p}_t(\G)$ and $g \in \WW^{m+1,p}_\#(\G)$, then $(\bu,\uppi)$ belongs to the space $\bW^{m+2,p}_t(\G) \times \WW^{m+1,p}_\#(\G)$.
Moreover, the following estimate holds
\begin{equation}\label{eq:strong-NavierStokes-apriori}
\|\bu\|_{m+2,p;\G} + \|\uppi\|_{m+1,p;\G} \leq F_{m,p}\big(\|\bf\|_{m,p;\G},\|g\|_{m+1,p;\G}, \|\bu\|_{1,p;\G}\big),
\end{equation}
where $F_{m,p}: [0,\infty)^3 \rightarrow [0,\infty)$ is a smooth function, strictly increasing in each of its three components, and satisfies $F_{m,p}(0,0,0) = 0$.
\end{thm}

\begin{proof}
It suffices to prove
\begin{equation}\label{eq:strong-NavierStokes-apriori-WEAKERESTIMATE}
\|\bu\|_{m+2,p;\G} + \|\uppi\|_{m+1,p;\G} \leq F_{m,p}\big(\|\bf\|_{m,p;\G},\|g\|_{m+1,p;\G}, \|\bu\|_{m+1,p;\G}\big),
\end{equation}
from which \eqref{eq:strong-NavierStokes-apriori} easily follows from successive nesting.
The proof hinges on successive applications of \autoref{thm:Lp-based-regularity-tangentStokes} (higher $\LL^p$-based regularity of the tangent Stokes equations), \autoref{prop:GNG-embeddings} (Gagliardo--Nirenberg--Sobolev inequality) and \autoref{prop:Morrey-inequality} (Morrey's inequality).

We first prove \eqref{eq:strong-NavierStokes-apriori-WEAKERESTIMATE} for $m=0$.
For general $m$ the proof will follow by induction.
\\ \\
\textbf{Step 1: Case $m=0$.}
Notice that $(\bu,\uppi) \in \bW^{1,p}_t(\G) \times \LL^p_\#(\G)$ is a solution of the (weak form of) the tangent Stokes equations \eqref{eq:weak-manifold-stokes-nonHilbert} with data $\overline\bf := \bf - (\nablaG \bu)\bu \in (\bW^{1,p^*}_t(\G))'$ (cf.~\eqref{convection-estimate}) and $\overline g := g \in \WW^{1,p}_\#(\G)$.
Hence, we will prove that $(\nablaG \bu)\bu \in \bL^p_t(\G)$ and
\begin{equation}\label{convective-term-to-bound}
\|(\nablaG \bu)\bu\|_{0,p;\G} \leq \wt F_p\big(\|\bf\|_{0,p;\G},\|g\|_{1,p;\G}, \|\bu\|_{1,p;\G}\big),
\end{equation}
from which the assertion for $m=0$ follows from a direct application of \autoref{thm:Lp-based-regularity-tangentStokes}.
We will simply show the case $d\geq 3$ (the case $d=2$ is easier), for which there are three subcases: $p>d$, $p=d$ and $\frac{d}{2} < p <d$.

\begin{enumerate}
\item If $p>d$, then by \autoref{prop:Morrey-inequality}, we deduce that $\bu \in \bC(\G)$, and hence $(\nablaG\bu)\bu \in \bL^p_t(\G)$ with $\|(\nablaG\bu)\bu\|_{0,p;\G} \leq \|\nablaG\bu\|_{0,p;\G} \|\bu\|_{0,\infty;\G} \leq C \|\bu\|_{1,p;\G}^2$.

\item If $p=d$, then from \autoref{prop:GNG-embeddings}, we deduce that $\bu \in \bL^q_t(\G)$ for each $q \in (1,\infty)$, and hence by H\"older's inequality, $(\nablaG\bu)\bu \in \bL^q_t(\G)$ for each $q \in (1,p) = (1,d)$.
By \autoref{thm:Lp-based-regularity-tangentStokes}, we infer that $(\bu,\uppi) \in \bW^{2,q}_t(\G) \times \WW^{1,q}_\#(\G)$ for each $q \in (1,d)$.
Choosing $q \in (\frac{d}{2},d)$, we further obtain by \autoref{prop:GNG-embeddings} and \autoref{prop:Morrey-inequality} that $\nablaG \bu \in \bbL^{\frac{dq}{d-q}}(\G)$ and $\bu \in \bC(\G)$, with the estimates
\begin{equation*}
\|\bu\|_{0,\infty;\G}
\lesssim \|\bu\|_{2,q;\G}
\lesssim \|\bf\|_{0,q;\G} + \|g\|_{1,q;\G} + \|(\nablaG\bu)\bu\|_{0,q;\G}
\end{equation*}
and
\begin{equation*}
\|(\nablaG\bu)\bu\|_{0,q;\G} \leq \|\nablaG\bu\|_{0,d;\G} \|\bu\|_{0,\frac{dq}{d-q};\G} \lesssim \|\bu\|_{1,d;\G}^2.
\end{equation*}
This implies $(\nablaG\bu) \bu \in \bL^{\frac{dq}{d-q}}_t(\G) \subseteq \bL^d_t(\G)$ (because $\frac{dq}{d-q} \geq d=p$) along with
\begin{equation*}
\|(\nablaG\bu) \bu\|_{0,d;\G}
\leq \|\nablaG\bu\|_{0,d;\G} \|\bu\|_{0,\infty;\G}
\lesssim \|\bu\|_{1,d;\G} \left( \|\bf\|_{0,d;\G} + \|g\|_{1,d;\G} + \|\bu\|_{1,d;\G}^2  \right).
\end{equation*}

\item\label{it:d/2-less-p-less-d} If $p \in (\frac{d}{2},d)$, then from \autoref{prop:GNG-embeddings}, we deduce $\bu \in \bL^{\frac{dp}{d-p}}_t(\G)$, and hence by H\"older's inequality, $(\nablaG\bu) \bu \in \bL^{q_0}_t(\G)$ and $\|(\nablaG\bu) \bu\|_{0,q_0;\G} \lesssim \|\bu\|_{1,p;\G}^2$,
where $q_0 := \frac{dp}{2d-p} \in (\frac{d}{3},p)$.
Then, by \autoref{thm:Lp-based-regularity-tangentStokes}, we deduce that $(\bu,\uppi) \in \bW^{2,q_0}_t(\G) \times \WW^{1,q_0}_\#(\G)$ with the estimate
\begin{equation}\label{eq:initial-estimate}
\|\bu\|_{2,q_0;\G} + \|\uppi\|_{1,q_0;\G} \leq C\left( \|\bf\|_{0,p;\G} + \|g\|_{1,p;\G} + \|\bu\|_{1,p;\G}^2  \right).
\end{equation}
If $q_0 < \frac{d}{2}$, then from \autoref{prop:GNG-embeddings} we deduce that $\nablaG\bu \in \bbL^{\frac{dq_0}{d-q_0}}(\G)$ and $\bu \in \bL^{\frac{dq_0}{d-2q_0}}_t(\G)$, whence by H\"older's inequality $(\nablaG\bu) \bu \in \bL^{q_1}_t(\G)$ for $q_1 := \frac{dq_0}{2d-3q_0}$.
Invoking again \autoref{thm:Lp-based-regularity-tangentStokes} yields $(\bu,\uppi) \in \bW^{2,\min\{p,q_1\}}_t(\G) \times \WW^{1,\min\{p,q_1\}}_\#(\G)$.
We can iterate this procedure as follows:
if $q_n < \frac{d}{2} < p$, define
\begin{equation}\label{Sobolev-exponent-sequence}
q_{n+1} := \frac{dq_n}{2d-3q_n}.
\end{equation}
Then, if $(\bu,\uppi) \in \bW^{2,q_n}_t(\G) \times \WW^{1,q_n}_\#(\G)$, H\"older's inequality and \autoref{prop:GNG-embeddings} imply that $(\nablaG\bu)\bu \in \bL^{q_{n+1}}(\G)$ with the estimate $\|(\nablaG\bu)\bu\|_{0,q_{n+1};\G} \leq \|\nablaG\bu\|_{0,\frac{d q_n}{d-q_n};\G} \|\bu\|_{0,\frac{d q_n}{d-2q_n};\G} \lesssim \|\bu\|_{2,q_n;\G}^2$.
Therefore, by \autoref{thm:Lp-based-regularity-tangentStokes}, we deduce that $(\bu,\uppi) \in \bW^{2,r_{n+1}}_t(\G) \times \WW^{1,r_{n+1}}_\#(\G)$, with $r_n := \min\{p,q_n\}$ and
\begin{equation}\label{apriori-Sobolev-NS-sequence}
\begin{aligned}
\|\bu\|_{2,r_{n+1};\G} + \|\uppi\|_{1,r_{n+1};\G}
& \leq C_{n,p} \left( \|\bf\|_{0,r_{n+1};\G} + \|g\|_{1,r_{n+1};\G} + \|(\nablaG\bu)\bu\|_{0,q_{n+1};\G} \right)\\
& \leq \wt C_{n,p} \left( \|\bf\|_{0,r_{n+1};\G} + \|g\|_{1,r_{n+1};\G} + \|\bu\|_{2,q_n;\G}^2 \right).
\end{aligned}
\end{equation}
Notice that if $q_{n+1}<\frac{d}{2} < p$, then $r_{n+1} = q_{n+1}$.
Consequently, if there is $K \in \NN \cup \{0\}$ such that $q_{n+1} < \frac{d}{2}$ for each $n \in \{0,\dotsc,K-1\}$ (with the set being empty for $K=0$), then subsequent applications of estimates \eqref{apriori-Sobolev-NS-sequence} starting from $n=K$ and finishing at $n=0$, in conjunction with \eqref{eq:initial-estimate} and the simple estimates $\|\cdot\|_{0,q_n;\G} \lesssim \|\cdot\|_{0,p;\G}$ and $\|\cdot\|_{1,q_n;\G} \lesssim \|\cdot\|_{1,p;\G}$, yield
\begin{equation}\label{eq:ALMOST-convective-term-to-bound}
\|\bu\|_{2,r_{K+1};\G} \leq F_{p,K}\big(\|\bf\|_{0,p;\G},\|g\|_{1,p;\G}, \|\bu\|_{1,p;\G}\big).
\end{equation}
We claim that the algorithm described above terminates:
that is, there is $N \in \NN$ such that $q_N \geq \frac{d}{2}$.
By contradiction suppose $q_n < \frac{d}{2} < p$ (hence $r_n = q_n$) for each $n \in \NN$.
Since $q_0>\frac{d}{3}$, it is easy to see by induction that $\{q_n\}_{n \in \NN}$ is strictly increasing.
Hence, there exists $\overline{q}>\frac{d}{3}$ such that $\lim_{n\to\infty} q_n = \overline{q}$.
By taking limit as $n\to\infty$ in \eqref{Sobolev-exponent-sequence}, we obtain that $\overline{q} = \frac{d \overline{q}}{2d-3 \overline{q}}$, or equivalently, $\overline{q} = \frac{d}{3}$, which is a contradiction.

Let $N$ be the smallest integer such that $q_{N+1} \geq \frac{d}{2}$.
Hence, for each $n \in \{0,\dotsc,N\}$ it holds that $r_n = q_n < \frac{d}{2} < p$.
Since $\bu \in \bW^{2,r_{N+1}}(\G) \subseteq \bW^{2,\frac{d}{2}}(\G)$ (because $r_{N+1} := \min\{p, q_{N+1}\} \geq \frac{d}{2}$), we deduce from \autoref{prop:GNG-embeddings} that $\nablaG \bu \in \bbL^d(\G)$ and $\bu \in \bL^q_t(\G)$ for each $q \in (1,\infty)$,
whence, by H\"older's inequality, $(\nablaG\bu)\bu \in \bL^q_t(\G)$ for each $q \in (1,d)$ with the estimate $\|(\nablaG\bu)\bu\|_{0,q;\G} \leq C_q \|\bu\|_{2,\frac{d}{2};\G}^2$.
In particular, choosing $q= p \in (1,d)$, using estimate \eqref{eq:ALMOST-convective-term-to-bound} for $K=N$ allows us to write
\begin{equation*}
\|(\nablaG\bu)\bu\|_{0,p;\G}
\lesssim \|\bu\|_{2,\frac{d}{2};\G}^2
\lesssim \|\bu\|_{2,r_{N+1};\G}^2
\leq \hat F\big(\|\bf\|_{0,p;\G},\|g\|_{1,p;\G}, \|\bu\|_{1,p;\G}),
\end{equation*}
which is precisely \eqref{convective-term-to-bound}.

Finally, we mention that if $q_0 \geq \frac{d}{2}$, then $\bu \in \bW^{2,q_0}(\G) \subseteq \bW^{2,\frac{d}{2}}(\G)$ already, and the previous argument applies verbatim with $r_{N+1}$ replaced by $q_0$.
\end{enumerate}

\noindent \textbf{Step 2: Case $m \geq 1$.}
Assume that \eqref{eq:strong-NavierStokes-apriori-WEAKERESTIMATE} holds for a fixed $m \geq 0$.
Assume $\G$ is of class $C^{m+3,1}$ and $\bf \in \bW^{m+1,p}_t(\G)$ and $g \in \WW^{m+2,p}_\#(\G)$.
By the induction hypothesis we deduce that $(\bu,\uppi) \in \bW^{m+2,p}_t(\G) \times \WW^{m+1,p}_\#(\G)$ and
\begin{equation}\label{eq:hig-apriori-by-indhyp}
\|\bu\|_{m+2,p;\G} + \|\uppi\|_{m+1,p;\G} \leq F_{m,p}\big(\|\bf\|_{m,p;\G},\|g\|_{m+1,p;\G}, \|\bu\|_{m+1,p;\G}\big).
\end{equation}
We claim that
$(\nablaG\bu)\bu \in \bW^{m+1,p}_t(\G)$, its first $m+1$ extrinsic derivatives are given by the usual product rule, and
\begin{equation}\label{claim:convective-regularity}
\|(\nablaG\bu)\bu\|_{m+1,p;\G} \lesssim \|\bu\|_{m+2,p;\G}^2.
\end{equation}
To see this, we restrict ourselves to the case $p \in (\frac{d}{2},d)$ because the case $p \geq d$ is easier.
First observe that, since $\bu \in \bW^{m+2,p}(\G)$, \autoref{prop:GNG-embeddings} and \autoref{prop:Morrey-inequality} imply that $\bu$ and all of its extrinsic derivatives of order less than or equal to $m$ are in $\LL^\infty(\G)$ and that its extrinsic derivatives of order exactly $m+1$ belong to $\LL^{\frac{dp}{d-p}}(\G)$.
Moreover, notice that $(\nablaG\bu)\bu = \bP (\nablaM \bu)\bu$ whence, adopting Einstein summation convention, $[(\nablaG\bu)\bu]_i = P_{ij} (\uD_k u_j) u_k$.
Therefore, upon recalling \autoref{def:Sobolev-spaces-manifold} for Sobolev spaces on $\G$ and the definition of differential operators on $\G$ \cite[Definition 3.3]{BenavidesNochettoShakipov2025-a}, we utilize the Leibniz rule for weak derivatives in Sobolev spaces in parametric domain (see, e.g.~\cite[Exercise 11.51(ii)]{Leoni:2009}) to obtain the following explicit formula for the extrinsic derivatives of order $m+1$ of $(\nablaG\bu)\bu$:
\begin{equation*}
\uD_{\br} (P_{ij} \uD_k u_j u_k)
= \sum_{\substack{\balpha, \bbeta,\bgamma \in \{0,1,\dotsc,d+1\}^{m+1} \\\balpha + \bbeta + \bgamma = \br\\ \|\balpha\|_0 + \|\bbeta\|_0 + \|\bgamma\|_0 = m+1 }} \uD_{\balpha} P_{ij} \uD_{\bbeta}(\uD_k u_j) \uD_{\bgamma} u_k, \qquad \forall \br \in \{1,\dotsc,d+1\}^{m+1}.
\end{equation*}
An explanation of the above compact notation follows:
for any $\bt \in \{\br,\balpha,\bbeta,\bgamma\}$ we denote $\uD_{\bt} := \uD_{t_1} \dotsc \uD_{t_{m+1}}$, with the convention that $\uD_0$ indicates the identity operator, and $\|\bt\|_0$ denotes the number of nonzero components of $\bt$.
Recalling that $p \in (\frac{d}{2},d)$ and $\G$ is of class $C^{m+3,1}$, we employ H\"older's inequalities and the aforementioned Sobolev embeddings to bound the sum above as follows.
For $\bbeta=\br$ we bound the corresponding summand as
\begin{equation*}
\|P_{ij} \uD_{\br}(\uD_k u_j) u_k\|_{0,p;\G}
\lesssim \|\bu\|_{m+2,p;\G} \|\bu\|_{0,\infty;\G}
\lesssim \|\bu\|_{m+2,p;\G}^2,
\end{equation*}
whereas for $\bbeta \neq \br$, we instead write
\begin{equation*}
\|\uD_{\balpha} P_{ij} \uD_{\bbeta}(\uD_k u_j) \uD_{\bgamma} u_k\|_{0,p;\G}
\lesssim \|\uD_{\bbeta}(\uD_k u_j)\|_{0,\frac{dp}{d-p};\G} \|\uD_{\bgamma} u_k\|_{0,d;\G}
\lesssim \|\bu\|_{m+1,\frac{dp}{d-p};\G} \|\bu\|_{m+1,d;\G}
\lesssim \|\bu\|_{m+2,p;\G}^2,
\end{equation*}
because $\frac{dp}{d-p} > p > \frac{d}{2}$.
This finishes the proof of claim \eqref{claim:convective-regularity}.

With \eqref{claim:convective-regularity} at hand, then it is clear that $(\bu,\uppi)$ is a solution of the (weak form of) the tangent Stokes equations \eqref{eq:weak-manifold-stokes-nonHilbert} with data $\overline\bf := \bf - (\nablaG \bu)\bu \in \bW^{m+1,p}_t(\G)$ and $g \in \WW^{m+2,p}_\#(\G)$ satisfying $\|\overline\bf\|_{m+1,p;\G} \lesssim \|\bf\|_{m+1,p;\G} + \|\bu\|_{m+2,p;\G}^2$.
Hence, by a direct application of \autoref{thm:Lp-based-regularity-tangentStokes} we deduce that $(\bu,\uppi) \in \bW^{m+3,p}_t(\G) \times \WW^{m+2,p}_\#(\G)$ and, in conjunction with the estimate \eqref{eq:hig-apriori-by-indhyp}, that
\begin{multline*}
\|\bu\|_{m+3,p;\G} + \|\uppi\|_{m+2,p;\G}
\stackrel{\eqref{eq:allp-Lpbased-Stokes_higher-apriori}}{\lesssim} \|\overline\bf\|_{m+1,p;\G} + \| g\|_{m+2,p;\G}
\lesssim \|\bf\|_{m+1,p;\G} + \|\bu\|_{m+2,p;\G}^2 + \|g\|_{m+2,p;\G}\\
\stackrel{\eqref{eq:hig-apriori-by-indhyp}}{\leq} \|\bf\|_{m+1,p;\G} + \left(F_{m,p}\big(\|\bf\|_{m,p;\G},\|g\|_{m+1,p;\G}, \|\bu\|_{m+1,p;\G}\big)\right)^2  + \|g\|_{m+2,p;\G}.
\end{multline*}
This finishes the induction step and the proof of \autoref{thm:Lpbased-regularity-NS}.
\end{proof}

\begin{remark}
We now comment on a couple of critical cases.
\begin{itemize}
\item If $d=2$ and $p=\frac{4}{3}$ (hence $p^*=4$), then $\bW^{1,\frac{4}{3}}(\G)$ is embedded into $\bL^4(\G)$.
Consequently, we can only ensure $(\nablaG \bu)\bu \in \bL^1_t(\G)$, which is not covered by \autoref{thm:Lp-based-regularity-tangentStokes} (higher $\LL^p$-based Sobolev regularity of the tangent Stokes equations).
\item If $d\geq3$ and $p=\frac{d}{2}$, then $\bW^{1,\frac{d}{2}}(\G)$ is embedded into $\bL^d(\G)$, and so $(\nablaG\bu)\bu \in \bL^{q_0}_t(\G)$, for $q_0 = \frac{d}{3}$.
If $d=3$, then $q_0=1$, which is not covered by \autoref{thm:Lp-based-regularity-tangentStokes}.
On the other hand, if $d \geq 4$ then the iterative process described in \autoref{it:d/2-less-p-less-d} within the proof of \autoref{thm:Lpbased-regularity-NS} fails to increase the integrability of the solution $(\bu,\uppi)$.
More precisely, $q_n = \frac{d}{3}$ for each $n \in \NN$, whence we are left with the ``suboptimal'' result:
$(\bu,\uppi) \in \bW^{2,\frac{d}{3}}_t(\G) \times \WW^{1,\frac{d}{3}}_\#(\G)$ and
\begin{equation*}
\|\bu\|_{2,\frac{d}{3};\G} + \|\uppi\|_{1,\frac{d}{3};\G}
\lesssim \|\bf\|_{0,\frac{d}{3};\G} + \|g\|_{1,\frac{d}{3};\G} + \|(\nablaG\bu)\bu\|_{0,\frac{d}{3};\G}
\lesssim \|\bf\|_{0,\frac{d}{3};\G} + \|g\|_{1,\frac{d}{3};\G} + \|\bu\|_{1,\frac{d}{2};\G}^2.
\end{equation*}
\end{itemize}
\end{remark}

\begin{remark}
Should we have a stability bound of the form $\|\bu\|_{1,p;\G} \lesssim \|\bf\|_{(\bW^{1,p^*}_t(\G))'} + \|g\|_{0,p;\G}$, then estimate \eqref{eq:strong-NavierStokes-apriori} in \autoref{thm:Lpbased-regularity-NS} could be improved to
\begin{equation*}
\|\bu\|_{m+2,p;\G} + \|\uppi\|_{m+1,p;\G} \leq \wt F_{m,p}(\|\bf\|_{m,p;\G},\|g\|_{m+1,p;\G}).
\end{equation*}
In particular, if $d \in \{2,3\}$, $p=2$ and $g = 0$, from \autoref{thm:Oseen-wellposedness} (well-posedness of the Oseen equations) with $\overline\bu = \bu$ we know that $\|\bu\|_{1,\G} \lesssim \|\bf\|_{(\bH^1_t(\G))'}$ and so
\begin{equation*}
\|\bu\|_{m+2,\G} + \|\uppi\|_{m+1,\G} \leq \wh F_m\big(\|\bf\|_{m,\G}\big).
\end{equation*}
\end{remark}

\section{Connection to other Laplacians}\label{sec:connection-other-laplacians}
As was remarked in the introduction, the Bochner Laplacian $\DeltaB = \bP \divG \nablaG$ is not the only way to define a Laplacian acting on tangent vector fields. The other two Laplacians that are of interest are the Hodge Laplacian $\DeltaH$ and the surface diffusion $\DeltaS = \bP \divG (\nablaG + \nablaG^T)$.
The operators $\DeltaS, \DeltaH$ are related to the Bochner Laplacian $\DeltaB$ through the following expressions:
\begin{equation*}
\DeltaS \bv = \DeltaB \bv + \bM \bv + \nablaG \divG \bv, \qquad \DeltaH \bv = \DeltaB \bv - \bM \bv,
\end{equation*}
where $\bM := \tr(\bB) \bB - \bB^2$ is the Ricci curvature tensor.

By considering $\DeltaS$ or $\DeltaH$ in lieu of $\DeltaB$ in \eqref{eq:manifold-BochnerPoisson}, \eqref{eq:manifold-stokes} and \eqref{eq:manifold-Navier--Stokes} we arrive at the following collection of problems:
\begin{align*}
\label{weak-VL_S} \tag{$\textbf{VL}_S$} -\DeltaB \bu - \bM \bu - \nablaG \divG \bu = -\DeltaS \bu & = \bf, & & \\
\label{weak-VL_H} \tag{$\textbf{VL}_H$} -\DeltaB \bu + \bM \bu = -\DeltaH \bu & = \bf, & & \\
\label{weak-S_S} \tag{$\textbf{S}_S$} -\DeltaB \bu - \bM \bu - \nablaG\divG \bu + \nablaG \uppi = -\DeltaS \bu + \nablaG \uppi & = \bf, & & \divG \bu = g, \\
\label{weak-S_H} \tag{$\textbf{S}_H$} -\DeltaB \bu + \bM \bu + \nablaG \uppi = -\DeltaH \bu + \nablaG \uppi & = \bf, & & \divG \bu = g, \\
\label{weak-NS_S} \tag{$\textbf{NS}_S$} -\DeltaB \bu - \bM \bu - \nablaG\divG \bu + (\nablaG \bu) \bu + \nablaG \uppi = -\DeltaS \bu + (\nablaG \bu) \bu + \nablaG \uppi & = \bf, & & \divG \bu = g, \\
\label{weak-NS_H} \tag{$\textbf{NS}_H$} -\DeltaB \bu + \bM \bu + (\nablaG \bu) \bu + \nablaG \uppi = -\DeltaH \bu + (\nablaG \bu) \bu + \nablaG \uppi & = \bf, & & \divG \bu = g.
\end{align*}
In the spirit of \eqref{eq:weak-manifold-BochnerPoisson-nonhilbertian}, \eqref{eq:weak-manifold-stokes-nonHilbert} and \eqref{eq:weak-manifold-Navier--Stokes}, integration-by-parts formulas induce variational formulations for these problems in which the flow velocity $\bu$ and pressure $\uppi$ (for tangent Stokes and tangent Navier--Stokes) are sought respectively in $\bW^{1,p}_t(\G)$ and $\LL^p_\#(\G)$, and test functions $\bv$ and $q$ lie respectively in $\bW^{1,p^*}_t(\G)$ and $\LL^{p^*}_\#(\G)$;
we omit their explicit expressions, which are easy to derive.
We stress that in the cases of \eqref{weak-S_S} and \eqref{weak-NS_S}, the term $\nablaG\divG \bu$ appearing in the left-hand side equals $\nablaG g$, and hence can be regarded as data.

The question of well-posedness of these problems has to do with the validity of Korn-type inequalities in $\bW^{1,p}_t(\G)$, something we do not address in this paper.
Nevertheless, assuming a solution exists for each of these problems, we can ensure that it exhibits higher regularity provided that the data and manifold are regular enough.
\begin{thm}[higher regularity for problems with $\DeltaS, \DeltaH$]
Suppose $\G$ is of class $C^{m+2,1}$ for some integer $m\geq 0$ and let $p \in (1,\infty)$.
Then,
\begin{enumerate}
\item[(\textbf{VL})]\label{it:othersVL-regularity} If $\bu \in \bW^{1,p}_t(\G)$ is a weak solution of either \eqref{weak-VL_S} or \eqref{weak-VL_H} with datum $\bf \in \bW^{m,p}_t(\G)$, then $\bu$ actually lies in $\bW^{m+2,p}_t(\G)$.
Moreover, the following a-priori estimate holds
\begin{equation}\label{eq:Lpbased-apriori-VL_S-VL_H}
\|\bu\|_{m+2,p;\G} \leq C_{m,p} \left( \|\bf\|_{m,p;\G} + \|\bu\|_{0,p;\G} \right).
\end{equation}
\item[(\textbf{S})] \label{it:othersS-regularity} If $(\bu, \uppi) \in \bW^{1,p}_t(\G) \times \LL^p_\#(\G)$ is a weak solution of either \eqref{weak-S_S} or \eqref{weak-S_H} with data $(\bf, g) \in \bW^{m,p}_t(\G) \times \WW^{m+1,p}_\#(\G)$, then $(\bu, \uppi)$ actually lies in $\bW^{m+2, p}_t(\G) \times \WW^{m+1, p}_\#(\G)$.
Moreover, the following a-priori estimate holds
\begin{equation*}
\|\bu\|_{m+2,p;\G} + \|\uppi\|_{m+1,p;\G} \leq C (\|\bf\|_{m,p;\G} + \|g\|_{m+1,p;\G} + \|\bu\|_{0,p;\G}).
\end{equation*}
\item[(\textbf{NS})] \label{it:othersNS-regularity} Suppose $p \in (\frac43,\infty)$ for $d=2$ and $p \in (\frac{d}{2},\infty)$ for $d\geq3$.
If $(\bu, \uppi) \in \bW^{1,p}_t(\G) \times \LL^p_\#(\G)$ is a weak solution of either \eqref{weak-NS_S} or \eqref{weak-NS_H} with data $(\bf, g) \in \bW^{m,p}_t(\G) \times \WW^{m+1,p}_\#(\G)$, then $(\bu, \uppi)$ actually lies in $\bW^{m+2, p}_t(\G) \times \WW^{m+1, p}_\#(\G)$.
Moreover, the following a-priori estimate holds
\begin{equation*}
\|\bu\|_{m+2,p;\G} + \|\uppi\|_{m+1,p;\G} \leq F_{m,p}(\|\bf\|_{m,p;\G},\|g\|_{m+1,p;\G}, \|\bu\|_{m+1,p;\G}),
\end{equation*}
where $F_{m,p}: [0,\infty)^3 \rightarrow [0,\infty)$ is a smooth function, strictly increasing in each of its three components, and satisfies $F_{m,p}(0,0,0) = 0$.

\end{enumerate}
\begin{proof}
The proofs for \eqref{weak-VL_H}, \eqref{weak-S_H}, \eqref{weak-NS_H} respectively follow from successive  applications of \autoref{thm:allp-based-regularity-BochnerLaplacian}, \autoref{thm:Lp-based-regularity-tangentStokes} and \autoref{thm:Lpbased-regularity-NS} with datum $\bf \gets \bf - \bM \bu$ as many times as the regularity of $\bf$ and $\G$ permits.
Moreover, the proofs for \eqref{weak-S_S} and \eqref{weak-NS_S} similarly follow from applying \autoref{thm:Lp-based-regularity-tangentStokes} and \autoref{thm:Lpbased-regularity-NS} with data $\bf \gets \bf + \bM\bu + \nablaG g$.

The proof for \eqref{weak-VL_S} is more delicate due to the presence of $\int_\G \divG \bu \, \divG \bv$, which is of the same order as $\int_\G \nablaG \bu : \nablaG \bv$.
For the sake of clarity we express the previously omitted variational formulation for \eqref{weak-VL_S};
it reads
\begin{equation}
\label{weak-VL_S-explicit} \int_\G \nablaG \bu : \nablaG \bv - \int_\G \bM \bu \cdot \bv + \int_\G \divG \bu \, \divG \bv = \langle \bf, \bv \rangle, \qquad \forall\bv \in \bW^{1,p^*}_t(\G).
\end{equation}
By \autoref{prop:Helmholtzdecomposition} (Helmholtz--Weyl decomposition on $\G$) we can decompose the solution $\bu$ of \eqref{weak-VL_S-explicit} as $\bu = \bu_\sigma + \nablaG \phi$, where $\bu_\sigma \in \bW^{1,p}_{t,\sigma}(\G)$ and $\phi \in \WW^{2,p}_\#(\G)$.
Similarly any test function $\bv \in \bW^{1,p^*}_t(\G)$ can be written as $\bv = \bv_\sigma + \nablaG \psi$ where $\bv_\sigma \in \bW^{1,p^*}_{t,\sigma}(\G)$ and $\psi \in \WW^{2,p^*}_\#(\G)$.
Now, choosing $\bv = \bv_\sigma$ in \eqref{weak-VL_S-explicit} yields
\begin{equation*}
\int_\G \nablaG \bu_\sigma : \nablaG \bv_\sigma + \int_\G \nablaG \nablaG \phi: \nablaG \bv_\sigma = \int_\G (\bf + \bM\bu) \cdot \bv_\sigma, \qquad \forall \bv_\sigma \in \bW^{1,p^*}_{t,\sigma}(\G).
\end{equation*}
Since $\divG \bv_\sigma = 0$, we can use \autoref{lem:exploiting-symmetry-covarianthessian} (weak commutator identity) to rewrite the second term in the left-hand side as $\int_\G \nablaG\nablaG \phi: \nablaG \bv_\sigma = - \int_\G \bM\nablaG\phi \cdot \bv_\sigma$, thus arriving at
\begin{equation*}
\int_\G \nablaG \bu_\sigma : \nablaG \bv_\sigma
= \int_\G (\bf + \bM\bu +\bM \nablaG \phi) \cdot \bv_\sigma
= \int_\G (\bf + \bM\bu_\sigma +2\bM \nablaG \phi) \cdot \bv_\sigma, \qquad \forall \bv_\sigma \in \bW^{1,p^*}_{t,\sigma}(\G).
\end{equation*}
Consequently, by \autoref{thm:Lp-based-regularity-tangentStokes} ($\LL^p$-based regularity for tangent Stokes), $\bu_\sigma \in \bW^{2,p}_t(\G)$ and
\begin{equation}\label{VL_S-divfreepart}
\|\bu_\sigma\|_{2,p;\G}
\lesssim \|\bf + \bM\bu_\sigma +2\bM \nablaG \phi\|_{0,p;\G}
\lesssim \|\bf\|_{0,p;\G} + \|\bu_\sigma\|_{0,p;\G} +\|\nablaG \phi\|_{0,p;\G}
\lesssim \|\bf\|_{0,p;\G} + \|\bu\|_{0,p;\G}.
\end{equation}
On the other hand, by choosing $\bv = \nablaG \psi$ in \eqref{weak-VL_S-explicit} we obtain
\begin{equation*}
\int_\G \nablaG \bu_\sigma : \nablaG\nablaG \psi + \int_\G \nablaG \nablaG \phi : \nablaG\nablaG \psi + \int_\G \DeltaG \phi \, \DeltaG \psi
= \int_\G (\bf + \bM\bu) \cdot \nablaG \psi, \qquad \forall \psi \in \WW^{2,p^*}_\#(\Gamma).
\end{equation*}
Again, by means of \autoref{lem:exploiting-symmetry-covarianthessian} we can rewrite the first two terms in the left-hand side as
\begin{align*}
\int_\G \nablaG \bu_\sigma : \nablaG\nablaG \psi & = - \int_\G \bM \bu_\sigma \cdot \nablaG \psi,\\
\int_\G \nablaG \nablaG \phi : \nablaG\nablaG \psi & = \int_\G \DeltaG \phi \, \DeltaG \psi - \int_\G \bM \nablaG \phi \cdot \nablaG \psi.
\end{align*}
Thus, recalling that $\bu = \bu_\sigma + \nablaG \phi$, we obtain
\begin{equation*}
2 \int_\G \DeltaG \phi \, \DeltaG \psi
= \int_\G (\bf + 2\bM\bu ) \cdot \nablaG \psi, \qquad \forall \psi \in \WW^{2,p^*}_\#(\Gamma),
\end{equation*}
whence by the $\LL^p$-based regularity of the biharmonic problem \cite[\S4]{BenavidesNochettoShakipov2025-a} we deduce that $\phi \in \WW^{3,p}_\#(\Gamma)$ and
\begin{equation}\label{VL_S-gradpart}
\|\phi\|_{3,p;\G} \lesssim \|\divG (\bf + 2\bM\bu)\|_{(\WW^{1,p^*}(\Gamma))'} \lesssim \|\bf\|_{0,p;\G} + \|\bu\|_{0,p;\G}.
\end{equation}
Combining estimates \eqref{VL_S-divfreepart} and \eqref{VL_S-gradpart} yields \eqref{eq:Lpbased-apriori-VL_S-VL_H} for $m=0$.
Higher regularity estimates follow from a straightforward bootstrapping argument.
\end{proof}
\end{thm}

\section*{Acknowledgements}
The authors are grateful to Professor Giuseppe Savaré (Bocconi University), whose sharp observations on a previous version of this manuscript motivated us to streamline the overall arguments in the paper and improve its results.

The original submission of this manuscript has been written and developed solely by the authors without the assistance of any AI tool.
After the first peer review round, we utilized Claude (Anthropic) and ChatGPT (OpenAI) to correct previously unidentified typos and minor errors and inconsistencies.

\appendix

\section{Variational problems in reflexive Banach spaces}\label{app:var-prob}
In this section we recall results on the well-posedness of variational problems in reflexive Banach spaces.
Classical references on the subject include \cite{ErnGuermond2004,Ciarlet2013,Brezis2011}.
We also mention the recent work \cite{Gatica2024} (in Spanish).

For any continuous linear operator $T: X \rightarrow Y'$ (that is, $T \in \sL(X,Y')$), where $X$ and $Y$ are Banach spaces, we let $T': Y'' \rightarrow X'$ denote its Banach adjoint defined by $T'(\sG) = \sG \circ T$ for each $\sG \in Y''$.
If $Y$ is reflexive, we also define $T^\mathrm{t}: Y \rightarrow X'$ by $T' \circ \cJ_Y$, where $\cJ_Y : Y \rightarrow Y''$ is the canonical embedding of $Y$ onto its bidual $Y''$.
In this case, $T$ and $T^\mathrm{t}$ are related via the identity
\begin{equation}\label{Banach-duality}
\langle T(x), y \rangle_{Y' \times Y} = \langle T^\mathrm{t}(y), x \rangle_{X' \times X}, \qquad \forall x \in X, \, y \in Y.
\end{equation}

The relation between the injectivity and surjectivity of the operators $T$ and $T^\mathrm{t}$ is summarized in \cite[Theorem~A.1]{BenavidesNochettoShakipov2025-a} based on \cite[Section 2.7]{Brezis2011}.
Its direct consequence is the celebrated Banach--Ne\v{c}as--Babu\v{s}ka theorem given in \cite[Theorem~A.2]{BenavidesNochettoShakipov2025-a}.

In several applications the underlying operator $T$ is defined in terms of product spaces and with a ``skew-triangular'' structure.
In this case, \cite[Theorem~A.2]{BenavidesNochettoShakipov2025-a}
can be equivalently reformulated as the so-called \emph{generalized Babu\v{s}ka--Brezzi} theorem.

\begin{thm}[generalized Babu\v{s}ka--Brezzi]\label{thm:gen-BB}
Let $X_2$, $M_1$, $X_1$ and $M_2$ be real Banach spaces such that $M_1$, $X_1$ and $M_2$ are reflexive, and consider $A \in \sL(X_2,X_1')$ and $B_i \in \sL(X_i,M_i')$.
Moreover, for each $i \in \{1,2\}$ define
\begin{equation*}
K_i := \operatorname{ker}(B_i) = \{x \in X_i : B_i(x) = 0\} = \{x_i \in X_i : \langle B_i(x_i), m_i \rangle = 0, \quad \forall m_i \in M_i\}.
\end{equation*}
Then, for each $(F,G) \in X_1' \times M_2'$ there exists a unique $(x_2,m_1) \in X_2 \times M_1$ solution of
\begin{equation*}
\begin{aligned}
A(x_2) + B_1^\mathrm{t}(m_1) & = F,\\
B_2(x_2) & = G.
\end{aligned}
\end{equation*}
if and only if
\begin{itemize}
\item[i)] For each $i \in \{1,2\}$ there exists $\beta_i > 0$ such that
\begin{equation*}
\sup_{\substack{x_i \in X_i\\ x_i \neq 0}} \frac{\langle B_i(x_i),m_i \rangle}{\|x_i\|} \geq \beta_i \|m_i\|, \qquad \forall m_i \in M_i.
\end{equation*}
\end{itemize}
and one of the following pairs of conditions holds

\vspace*{0.5cm}
\begin{minipage}{0.47\textwidth}
\begin{enumerate}
\item[ii-1)] If $x_2 \in K_2$ is such that $\langle A(x_2), x_1 \rangle = 0$ for each $x_1 \in K_1$, then $x_2 = 0$.
\item[ii-2)] There is $\alpha>0$ such that $$\sup_{\substack{x_2 \in K_2\\ x_2 \neq 0}} \frac{\langle A(x_2), x_1 \rangle}{\|x_2\|} \geq \alpha \|x_1\|, \quad \forall x_1 \in K_1.$$
\end{enumerate}
\end{minipage}%
\begin{minipage}{0.47\textwidth}
\begin{enumerate}
\item[ii-1)'] If $x_1 \in K_1$ is such that $\langle A(x_2), x_1 \rangle = 0$ for each $x_2 \in K_2$, then $x_1 = 0$.
\item[ii-2)'] There is $\alpha>0$ such that $$\sup_{\substack{x_1 \in K_1\\ x_1 \neq 0}} \frac{\langle A(x_2), x_1 \rangle}{\|x_1\|} \geq \alpha \|x_2\|, \quad \forall x_2 \in K_2.$$
\end{enumerate}
\end{minipage}

\noindent In this case, the solution $(x_2,m_1)$ satisfies the a-priori bounds
\begin{equation*}
\|x_2\| \leq \frac{1}{\alpha} \|F\| + \frac{1}{\beta_2} \left( 1 + \frac{\|A\|}{\alpha}  \right) \|G\|, \qquad
\|m_1\| \leq \frac{1}{\beta_1} \left( 1 + \frac{\|A\|}{\alpha} \right)\|F\| + \frac{\|A\|}{\beta_1\beta_2} \left( 1 + \frac{\|A\|}{\alpha}  \right) \|G\|.
\end{equation*}
\end{thm}
\autoref{thm:gen-BB} in the Hilbert and symmetric ($B_1=B_2$) case was proved by Brezzi \cite{Brezzi1974}, and later extended to the nonsymmetric case by Nicolaides \cite{Nicolaides1982}.
\autoref{thm:gen-BB} as stated was proved much later by Bernardi, Canuto and Maday \cite{BernardiCanutoMaday:1988}.

\section{Calculus identities on manifolds}\label{app:calculus}
In this section we present some results on tangential calculus that are instrumental in \autoref{sec:bl}.
Within this section we adopt the Einstein summation convention.

The following result is used in the proof of \autoref{thm:eq:weak-manifold-BochnerPoisson-nonhilbertian-wellposedness} (well-posedness of the Bochner Laplacian in $\bW^{1,p}_t(\G)$) to relate, in their weak forms, the Bochner Laplacian operator and the componentwise application of the Laplace--Beltrami operator.
\begin{proposition}\label{prop:weak_BL-LB-relation}
Suppose $\G$ is of class $C^2$.
Then,
\begin{equation*}
\nablaG(\bP\bv) = - (\bv \cdot \bnu) \bB + \nablaG \bv, \qquad \nablaG \bu : \nablaG (\bP\bv) = \nablaM \bu : \nablaM \bv - (\bv \cdot \bnu) \nablaM \bu : \bB + (\bnu \bu^T \bB) : \nablaM \bv,
\end{equation*}
for each $\bu \in \bC^1_t(\G)$ and $\bv \in \bC^1(\G)$ (not necessarily tangential).
\begin{proof}
We first write,
\begin{equation*}
[\nablaG(\bP\bv)]_{ij}
= [\bP \nablaM(\bP\bv)]_{ij}
= P_{ik} [\nablaM(\bP\bv)]_{kj}
= P_{ik} \uD_j (\bP\bv)_k
= P_{ik} \uD_j (P_{kl} v_l)
= P_{ik} \uD_j P_{kl} v_l + P_{ik} P_{kl} \uD_j v_l.
\end{equation*}
By recalling identity $B_{qr} u_r = - \uD_q u_r \nu_r$ (cf.~\eqref{WeingartenExtr-for-tangent}), the facts that $\bB \bP = \bB$ and $\bP \bnu = \mathbf{0}$, and noticing that $\uD_k P_{jl} = -B_{kj} \nu_l -B_{kl} \nu_j$ because $\bP_{jl} = \delta_{jl} - \nu_j \nu_l$ and $\bB = \nablaM \bnu$, we arrive at
\begin{equation*}
[\nablaG(\bP\bv)]_{ij}
= - P_{ik} B_{kj} \nu_l v_l - P_{ik} B_{lj} \nu_k v_l + P_{il} \uD_j v_l
= - (\bv \cdot \bnu) B_{ij} + (\bP \nablaM \bv)_{ij}.
\end{equation*}
Hence, it is clear that
\begin{equation*}
\nablaG \bu: \nablaG(\bP\bv)
= - (\bv \cdot \bnu) \nablaM \bu : \bB + (\bP \nablaM \bu) : \nablaM \bv
= - (\bv \cdot \bnu) \nablaM \bu : \bB + \nablaM \bu : \nablaM \bv - (\bN \nablaM \bu) : \nablaM \bv,
\end{equation*}
where $\bN := \bnu \otimes \bnu$.
We conclude the proof by noticing that, again by \eqref{WeingartenExtr-for-tangent}, it holds that
\begin{equation*}
(\bN \nablaM \bu) : \nablaM \bv
= \tr(\nablaM^T \bu \, \bnu \bnu^T \nablaM \bv)
= - \tr(\bB \bu \, \bnu^T \nablaM \bv)
= - (\bnu \bu^T \bB) : \nablaM \bv.
\end{equation*}
This concludes the proof.
\end{proof}
\end{proposition}
The following result shows that the covariant Hessian $\nablaG\nablaG v$ of a scalar function $v \in C^2(\G)$ is symmetric, provided $\G$ is of class $C^2$.

\begin{proposition}[symmetry of the covariant Hessian]\label{prop:sym-cov-Hess}
Suppose $\G$ is of class $C^2$.
Let the superscript $^e$ denote an operator that extends functions defined on $\G$ to a tubular neighborhood $\Omega_\delta \subseteq \RR^{d+1}$ of $\G$ with the following property: if $v \in C^2(\G)$ then $v^e: \Omega_\delta \rightarrow \RR$ is also of class $C^2$.
Then, identity \eqref{symmetry-cov-Hes} holds true, that is:
\begin{equation}
\tag{\ref{symmetry-cov-Hes}}
\nablaG \nablaG v = \bP (\nabla\nabla v^e)\rvert_\G \bP - ((\nabla v^e)\rvert_\G \cdot \bnu) \bB, \qquad \text{on $\G$.}
\end{equation}
\begin{proof}
Let $\sfd$ denote the \emph{signed distance function} \cite[Section 1.2.3]{BonitoDemlowNochetto2020} defined on a sufficiently small tubular neighborhood of $\G$, which we assume without loss of generality to coincide with $\Omega_\delta$.
Since $\G$ is of class $C^2$, we know that $\sfd \in C^2(\Omega_\delta)$.
The \emph{closest point projection} $\bP_{\sfd}: \Omega_\delta \rightarrow \RR^{d+1}$ is defined by $\bP_{\sfd}(\bx) = \bx - \sfd(\bx) \nabla \sfd(\bx)$ for each $\bx \in \Omega_\delta$.
Notice that $\nabla \bP_{\sfd} = \bP \circ \bP_{\sfd} - \sfd \nabla^2 \sfd$ in $\Omega_\delta$.
Let the superscript $^n$ denote the \emph{constant normal extension} \cite[eq.~(1.26)]{BonitoDemlowNochetto2020} of functions defined on $\G$ to $\Omega_\delta$.
More precisely, for each $v: \Gamma \rightarrow \RR$, the function $v^n:\Omega_\delta \rightarrow \RR$ is defined by $v^n := v \circ \bP_{\sfd}$.
Since $\G$ is of class $C^2$, then for each $v \in C^2(\G)$ we can only ensure $v^n : \Omega_\delta \rightarrow \RR$ to be of class $C^1$.
This is because differentiating $v^n = v \circ \bP_\sfd$ amounts to differentiating $\bP_\sfd = \mathrm{id} - \sfd \nabla \sfd$, which already contains $\nabla \sfd$.
From \cite[p.~15]{BonitoDemlowNochetto2020} we know that the extensions $^e$ and $^n$ are related with the differential operators $\nablaM$ on $\G$ acting on scalar and vector-valued functions $v$ and $\bv$ via the identity
\begin{align*}
\nablaM v = (\nabla v^n)\rvert_\G = \bP (\nabla v^e)\rvert_\G, \qquad
\nablaM \bv = (\nabla \bv^n)\rvert_\G = (\nabla \bv^e)\rvert_\G \bP, \qquad \text{on $\G$}.
\end{align*}
Recall also that $\nablaG v = \nablaM v$ and $\nablaG \bv = \bP \nablaM \bv$.

By the chain rule, we have that
\begin{equation*}
\nabla[ (\nabla v^e)\rvert_\G  ]^n
= \nabla\left( (\nabla v^e) \circ \bP_{\sfd}  \right)
= \left((\nabla^2 v^e) \circ \bP_{\sfd}\right) \nabla \bP_{\sfd}, \qquad \text{in $\Omega_\delta$},
\end{equation*}
which restricted to $\G$ reduces to
\begin{equation*}
\nablaM \left( (\nabla v^e)\rvert_\G  \right) = (\nabla^2 v^e)\rvert_\G \bP, \qquad \text{on $\G$}.
\end{equation*}
Moreover, by multiplying by $\bP$ from the left, we also obtain
\begin{equation}\label{ALMOST-symmetry-cov-Hes}
\nablaG \left( (\nabla v^e)\rvert_\G  \right) = \bP(\nabla^2 v^e)\rvert_\G \bP, \qquad \text{on $\G$}.
\end{equation}
Recalling that $\bP + \bN = \bI$, we can decompose $(\nabla v^e)\rvert_\G$ as $(\nabla v^e)\rvert_\G = \bP (\nabla v^e)\rvert_\G + \bN (\nabla v^e)\rvert_\G = \nablaG v + \bN (\nabla v^e)\rvert_\G$.
Moreover, using the first identity of \autoref{prop:weak_BL-LB-relation} with $\bv = (\nabla v^e)\rvert_\G$ we have that $\nablaG (\bN (\nabla v^e)\rvert_\G) = \nablaG ((\nabla v^e)\rvert_\G) - \nablaG (\bP (\nabla v^e)\rvert_\G) = ((\nabla v^e)\rvert_\G \cdot \bnu) \bB$.
Consequently, \eqref{ALMOST-symmetry-cov-Hes} becomes
\begin{equation*}
\nablaG \nablaG v + ((\nabla v^e)\rvert_\G \cdot \bnu) \bB = \bP(\nabla^2 v^e)\rvert_\G \bP,
\end{equation*}
which is precisely \eqref{symmetry-cov-Hes}.
\end{proof}
\end{proposition}

The symmetry of the covariant Hessian allows us to provide a simple proof of the commutator identity \eqref{commutators:second-tangential-derivatives}.

\begin{proposition}[Dziuk--Elliott's commutator identity]\label{prop:commutator-identity}
Suppose $\G$ is of class $C^2$.
Then, for each $v \in C^2(\G)$ and for each $i, j \in \{1,\dotsc,d+1\}$, identity \eqref{commutators:second-tangential-derivatives} holds true.
That is,
\begin{equation}
\tag{\ref{commutators:second-tangential-derivatives}}
\uD_i \uD_j v - \uD_j \uD_i v = (\bB \nablaM v)_j \nu_i - (\bB \nablaM v)_i \nu_j.
\end{equation}
\begin{proof}
First, notice that identity \eqref{commutators:second-tangential-derivatives} is equivalent to
\begin{equation*}
\nablaM^T \nablaM v - \nablaM \nablaM v
= \bnu (\bB \nablaM v)^T - (\bB \nablaM v) \bnu^T,
\end{equation*}
which we now prove.
Recalling identity $\bP + \bN = \bI$, and relation $\nablaG \bv = \bP\nablaM \bv$ for vector-valued functions $\bv$, we rewrite the second term on the left-hand side as follows
\begin{equation*}
\nablaM \nablaM v
= \nablaG \nablaM v + \bN\nablaM (\nablaM v)
= \nablaG \nablaM v - \bnu (\nablaM^T v) \bB,
\end{equation*}
where in the last equality we have used \eqref{ExtCov-for-tangent} with $\bv = \nablaM v \in \bC^1_t(\G)$ to write $\bN\nablaM (\nablaM v) = - \bnu (\nablaM^T v) \bB$.
Consequently, recalling from \autoref{prop:sym-cov-Hess} that $\nablaG \nablaG v = \nablaG \nablaM v$ is symmetric, we arrive at
\begin{equation*}
\nablaM^T \nablaM v - \nablaM \nablaM v
= - \bB (\nablaM v) \bnu^T + \bnu (\nablaG^T v) \bB.
\end{equation*}
This finishes the proof.
\end{proof}
\end{proposition}

We now prove the important \autoref{lem:exploiting-symmetry-covarianthessian} (weak commutator identity), that extends \cite[Lemma 2.1]{JankuhnOlshanskiiReusken2018}.

\begin{proof-thm}{{lem:exploiting-symmetry-covarianthessian}}
Let us first assume that $\bv \in \bC^2(\G)$ (not necessarily tangential) and $\phi \in C^2(\G)$.
Then, the symmetry of the covariant Hessian $\nablaG\nablaG \phi$ \eqref{symmetry-cov-Hes}, integration-by-parts formula \cite[Lemma B.7]{BouckNochettoYushutin2024} and the fact that $\nablaG^T \bv \, \bnu = \textbf{0}$ yield
\begin{equation*}
-\int_\G \nablaG \bv: \nablaG\nablaG \phi
= -\int_\G \nablaG^T \bv: \nablaG\nablaG \phi
= -\int_\G \nablaG^T \bv: \nablaM\nablaG \phi
= \int_\G \bP \divG\nablaG^T \bv \cdot \nablaG \phi.
\end{equation*}
Let us carefully manipulate the quantities appearing in the right-hand side of this last equation.
By means of the product rule and the commutator identity \eqref{commutators:second-tangential-derivatives}, we first write
\begin{align*}
(\bP \divG\nablaG^T \bv)_i
& = P_{ik} (\divG \nablaG^T \bv)_k
= P_{ik} \uD_l (\nablaG^T \bv)_{kl}\\
& = P_{ik} \uD_l (\nablaM^T \bv \, \bP)_{kl}
= P_{ik} \uD_l (\uD_k v_m P_{ml})
= P_{ik} (\uD_l \uD_k v_m) P_{ml} + P_{ik} (\uD_k v_m) (\uD_l P_{ml})\\
& \stackrel{\eqref{commutators:second-tangential-derivatives}}{=}
P_{ik} \left(\uD_k \uD_l v_m + (\bB \nablaM v_m)_k \nu_l - (\bB \nablaM v_m)_l \nu_k  \right) P_{ml}
+ P_{ik} (\uD_k v_m) (\uD_l P_{ml}),
\end{align*}
which by recalling that $\bP = \bI - \bnu \otimes \bnu$ and $\bB \bnu = \textbf{0}$, and identities $\uD_k P_{jl} = -B_{kj} \nu_l -B_{kl} \nu_j$ and $B_{qr} u_r = - \uD_q u_r \nu_r$ (see the proof of \autoref{prop:weak_BL-LB-relation}), can be further rewritten as follows:
\begin{align*}
(\bP \divG\nablaG^T \bv)_i
& = (\uD_i \uD_l v_m) P_{ml} - P_{ik} (\uD_k v_m) B_{ll} \nu_m - P_{ik} (\uD_k v_m) B_{ml} \nu_l\\
& = (\uD_i \uD_l v_m) P_{ml} - \tr(\bB) (\uD_i v_m) \nu_m\\
& = \uD_i \uD_l v_l - (\uD_i \uD_l v_m) \nu_m \nu_l - \tr(\bB) (\uD_i v_m) \nu_m.
\end{align*}
Whence, together with integration-by-parts formula \cite[Lemma B.5]{BouckNochettoYushutin2024}, we get
\begin{multline*}
\int_\G \bP \divG\nablaG^T \bv \cdot \nablaG \phi
= \int_\G \nablaG \divG \bv \cdot \nablaG \phi - \int_\G (\uD_i \uD_l v_m) \nu_m \nu_l \, \uD_i \phi - \int_\G \tr(\bB) \left( \nablaM^T \bv \, \bnu \right) \cdot \nablaG \phi\\
= - \int_\G \divG \bv \, \DeltaG \phi
+ \int_\G (\uD_l v_m) \uD_i (\nu_m \nu_l \, \uD_i \phi) - \int_\G \tr(\bB) (\uD_l v_m) \nu_m \nu_l \cancel{(\uD_i \phi) \nu_i} 
- \int_\G \tr(\bB) \left( \nablaM^T \bv \, \bnu \right) \cdot \nablaG \phi\\
= - \int_\G \divG \bv \, \DeltaG \phi
+ \int_\G (\uD_l v_m) (\uD_i \nu_l) \nu_m \, \uD_i \phi + \int_\G \cancel{(\uD_l v_m) \nu_l} \uD_i (\nu_m \, \uD_i \phi ) - \int_\G \tr(\bB) \left( \nablaM^T \bv \, \bnu \right) \cdot \nablaG \phi\\
= - \int_\G \divG \bv \, \DeltaG \phi + \int_\G (\bB \nablaM^T \bv \, \bnu) \cdot \nablaG \phi - \int_\G \tr(\bB) \left( \nablaM^T \bv \, \bnu \right) \cdot \nablaG \phi.
\end{multline*}
So far, we have proved that for each $\phi \in C^2(\G)$ and $\bv \in \bC^2(\G)$, we have that
\begin{equation*}
-\int_\G \nablaG \bv: \nablaG\nablaG \phi
= - \int_\G \divG \bv \, \DeltaG \phi + \int_\G (\bB \nablaM^T \bv \, \bnu) \cdot \nablaG \phi - \int_\G \tr(\bB) \left( \nablaM^T \bv \, \bnu \right) \cdot \nablaG \phi;
\end{equation*}
which, by means of a density argument, can be easily extended to $\bv \in \bC^1(\G)$.
In particular, by choosing $\bv \in \bC^1_t(\G)$ and using that $\bB \bv = - (\nablaM^T \bv) \bnu$ (cf.~\eqref{WeingartenExtr-for-tangent}), we arrive at \eqref{exploiting-symmetry-covarianthessian}.
This finishes the proof.
\end{proof-thm}

\section{Riemann curvature tensor}\label{sec:RCT}
This section is devoted to extending the notion of the classical Riemann Curvature tensor \cite{doCarmo1992} to a ``weaker'' setting.
This is instrumental in the proof of \autoref{lem:Poincare} (Poincaré inequality in $\bW^{1,p}_t(\G)$).

Within this section we write $\partial_{\ba} \bv := (\nablaM \bv)\ba$, $D_{\ba} \bv := (\nablaG \bv)\ba$ and $[\ba,\bb] := \partial_{\ba} \bb - \partial_{\bb} \ba$.

For $\G$ of class $C^3$, the Riemann Curvature tensor \cite[\textsection 4.2, Definition 2.1]{doCarmo1992} $R$ is defined by
\begin{equation*}
R(\ba,\bb,\bv) := D_{\ba} D_{\bb} \bv - D_{\bb} D_{\ba} \bv - D_{[\ba,\bb]} \bv \in \bC_t(\G),
\end{equation*}
for all $\ba,\bb \in \bC^1_t(\G)$ and $\bv \in \bC^2_t(\G)$.
Notice that $[\ba,\bb]$ is tangential, because
\begin{equation*}
[\ba,\bb] \cdot \bnu
= \bnu^T(\nablaM \bb)\ba -\bnu^T(\nablaM \ba)\bb
\stackrel{\eqref{WeingartenExtr-for-tangent}}{=} -\bb^T \bB \ba + \ba^T \bB \bb = 0.
\end{equation*}
Relying on the product rule $\nablaM (\alpha \bv) = \alpha \nablaM \bv + \bv \nablaM^T \alpha$, and using the identities $\bP \bnu = \textbf{0}$, $\bP \bB = \bB$ and $\bnu^T \partial_{\bb} \bv \stackrel{\eqref{WeingartenExtr-for-tangent}}{=} - \bb^T \bB \bv$, we proceed as in \cite[Lemma 2.1]{HansboLarsonLarsson2020} to initially obtain
\begin{equation}\label{app:proto-R-reduction}
R(\ba,\bb,\bv)
= \bP ( \partial_{\ba} \partial_{\bb} \bv - \partial_{\bb} \partial_{\ba} \bv - \partial_{[\ba,\bb]} \bv) + (\bv^T \bB \bb) \bB \ba - (\bv^T \bB \ba) \bB \bb.
\end{equation}
Furthermore, by taking into account the easily-verifiable identity
\begin{equation*}
(\partial_{\ba} \partial_{\bb} \bv)_i = \ba^T \left\{ (\nablaM \nablaM v_i)\bb + (\nablaM^T \bb)(\nablaM v_i)  \right\},
\end{equation*}
and its analogue for $(\partial_{\bb} \partial_{\ba} \bv)_i$, and noticing that $\ba^T (\nablaM \nablaM v_i) \bb = \bb^T (\nablaM \nablaM v_i) \ba$ because of \eqref{commutators:second-tangential-derivatives} and the tangentiality of $\ba$ and $\bb$, it is easy to see that $\partial_{\ba} \partial_{\bb} \bv - \partial_{\bb} \partial_{\ba} \bv = \partial_{[\ba,\bb]} \bv$;
hence \eqref{app:proto-R-reduction} reduces to
\begin{equation}\label{app:final-R-reduction}
R(\ba,\bb,\bv)
= (\bv^T \bB \bb) \bB \ba - (\bv^T \bB \ba) \bB \bb, \qquad \forall \ba,\bb \in \bC^1_t(\G), \bv \in \bC^2_t(\G).
\end{equation}
That is, $R$ reduces to a zeroth order differential operator.
Our objective is now to extend $R$ to less regular manifolds and vector fields.

Indeed, assume from now on that $\G$ is only of class $C^2$, integration-by-parts formulas \cite[Ap.~B]{BouckNochettoYushutin2024} induce the following definition: for each $\ba,\bb,\bv \in \bC^1_t(\G)$ we define the ``weak'' Riemann Curvature tensor $R(\ba,\bb,\bv)$ as a linear functional acting on vector fields $\bw \in \bC^1_t(\G)$ defined by
\begin{equation}\label{weak-R-definition}
\langle R(\ba,\bb,\bv), \bw \rangle
= - \int_\G (\nablaG \bv) \bb \cdot \divG(\bw \otimes \ba) + \int_\G (\nablaG \bv) \ba \cdot \divG(\bw \otimes \bb) - \int_\G \Big[\nablaG \bv \,( (\nablaM \bb) \ba - (\nablaM \ba) \bb ) \Big] \cdot \bw.
\end{equation}
By means of the easily-verifiable identity $\divG(\ba \otimes \bb) = (\nablaM \ba)\bb + (\divG \bb) \ba$, we can equivalently rewrite \eqref{weak-R-definition} as
\begin{align*}
\langle R(\ba,\bb,\bv), \bw \rangle
& = \int_\G (\nablaG \bv) (\ba \otimes \bb - \bb \otimes \ba) : \nablaM \bw + \int_\G \Big[ \nablaG \bv \, \left( (\nablaM \ba) \bb + (\divG \bb)\ba - (\nablaM \bb)\ba  - (\divG \ba)\bb    \right) \Big] \cdot \bw\\
& = \int_\G (\nablaG \bv) (\ba \otimes \bb - \bb \otimes \ba) : \nablaM \bw + \int_\G \Big[ \nablaG \bv \, \divG(\ba \otimes \bb - \bb \otimes \ba) \Big] \cdot \bw.
\end{align*}
Moreover, by using the identities $\nablaG \bv = \nablaM \bv + \bnu \bv^T \bB$ (cf.~\eqref{ExtCov-for-tangent}) and $\bB \bw = -(\nablaM^T \bw)\bnu$ (cf.~\eqref{WeingartenExtr-for-tangent}), and the tangentiality of $\bw$, we can further write
\begin{multline*}
\langle R(\ba,\bb,\bv), \bw \rangle
= \int_\G ((\bv^T \bB \bb) \bB \ba - (\bv^T \bB \ba) \bB \bb) \cdot \bw\\
+ \int_\G \Big\{\left[\nablaM \bv (\ba \otimes \bb - \bb \otimes \ba)\right] : \nablaM \bw + \left[ \nablaM \bv \, \divG(\ba \otimes \bb - \bb \otimes \ba) \right] \cdot \bw\Big\} .
\end{multline*}
We now aim to prove that this last integral vanishes.
For ease of notation, let $\bA$ be the skew-symmetric tensor field given by $\bA := \ba \otimes \bb - \bb \otimes \ba$.
Adopting the Einstein summation convention, we have that
\begin{multline}\label{app:manipulate-residue}
(\nablaM \bv) (\ba \otimes \bb - \bb \otimes \ba) : \nablaM \bw + \Big[ \nablaM \bv \, \divG(\ba \otimes \bb - \bb \otimes \ba) \Big] \cdot \bw\\
= (\nablaM \bv) \bA : \nablaM \bw + \Big[ \nablaM \bv \, \divG\bA \Big] \cdot \bw
= (\nablaM \bv \, \bA)_{ij} (\nablaM \bw)_{ij} + (\nablaM \bv \, \divG\bA)_i w_i\\
= (\nablaM \bv)_{ik} A_{kj} (\nablaM \bw)_{ij} + (\nablaM \bv)_{ik} (\divG\bA)_k w_i\\
= \uD_k v_i A_{kj} \uD_j w_i + \uD_k v_i \uD_j A_{kj} w_i
= \uD_k v_i \uD_j (A_{kj} w_i ).
\end{multline}
Now, let us notice that combining the integration-by-parts formula \cite[Lemma B.5]{BouckNochettoYushutin2024} with identity \eqref{commutators:second-tangential-derivatives} and using a density argument, we easily arrive at the following identity:
\begin{equation}\label{app:weak-commutator}
\int_\G \uD_i u \, \uD_j \varphi - \uD_j u \, \uD_i \varphi = \int_\G \left\{(\bB \nablaM u)_j \nu_i - (\bB \nablaM u)_i \nu_j + \tr(\bB) \left( \uD_i u \, \nu_j - \uD_j u \,\nu_i  \right)\right\} \varphi,
\end{equation}
for each $u,\varphi \in C^1(\G)$ and for each pair of indices $i,j \in \{1,\dotsc,d+1\}$.
Accordingly, applying this last identity to the right-hand side of \eqref{app:manipulate-residue}, recalling the fact that $\bA = \ba \otimes \bb - \bb \otimes \ba$ is skew-symmetric and that $\ba$ and $\bb$ are tangential, we arrive at
\begin{multline*}
\int_\G \uD_k v_i \uD_j (A_{kj} w_i )
\stackrel{\eqref{app:weak-commutator}}{=} \int_\G \uD_j v_i \uD_k (A_{kj} w_i )
+ \int_\G \left\{(\bB \nablaM v_i)_j \nu_k - (\bB \nablaM v_i)_k \nu_j + \tr(\bB) \left( \uD_k v_i \, \nu_j - \uD_j v_i \,\nu_k  \right)\right\} A_{kj} w_i\\
= \int_\G \uD_j v_i \uD_k (A_{kj} w_i )
= - \int_\G \uD_j v_i \uD_k (A_{jk} w_i )
= - \int_\G \uD_k v_i \uD_j (A_{kj} w_i ).
\end{multline*}
Hence, $\int_\G \uD_k v_i \uD_j (A_{kj} w_i ) = 0$ and so
\begin{equation}\label{weak-RCT-simplified}
\langle R(\ba,\bb,\bv), \bw \rangle
= \int_\G ((\bv^T \bB \bb) \bB \ba - (\bv^T \bB \ba) \bB \bb) \cdot \bw, \qquad \ba,\bb,\bv,\bw \in \bC^1_t(\G).
\end{equation}

\bibliographystyle{abbrv}
\bibliography{references}
\end{document}